\documentclass{amsart}

\usepackage{hyperref}
\usepackage{amsmath}
\usepackage{amssymb}
\usepackage{amsfonts}
\usepackage{amsthm}
\usepackage{tikz,tikz-cd,float}
\usepackage{latexsym}
\usepackage{esint}
\usepackage{mathtools}
\usepackage{mathrsfs}
\usepackage{graphicx}
\usepackage{bbm}
\usepackage{color}
\usepackage{bm}
\usepackage{todonotes}
\usepackage{comment} 
\usepackage{appendix}
\usepackage[top=1in, bottom=1.25in, left=1.10in, right=1.10in]{geometry}
\usepackage{enumerate}

\usepackage{caption}
\usepackage{subcaption}

\newtheorem{theorem}{Theorem}[section]
\newtheorem{lemma}[theorem]{Lemma}
\newtheorem{proposition}[theorem]{Proposition}
\newtheorem{corollary}[theorem]{Corollary}

\theoremstyle{definition}
\newtheorem{definition}[theorem]{Definition}

\theoremstyle{remark}
\newtheorem{remark}[theorem]{Remark}

\numberwithin{equation}{section}

\usepackage{stmaryrd}
\def\p{\partial}
\def\normal{{\hat {\mathbf{n}}}}

\def\domain{{\mathcal D}}

\newcommand{\bq}{\mathbf{q}}
\newcommand{\bu}{\mathbf{u}}

\newcommand{\bff}{\mathbf{f}}

\newcommand{\bs}[1]{{\boldsymbol{#1}}}
\newcommand{\mbf}[1]{{\mathbf{#1}}}
\newcommand{\mc}[1]{{\mathcal{#1}}}
\newcommand{\zh}{{\bs{\widehat z}}}

\newcommand{\dd}{\,\mathrm{d}}

\newcommand{\dx}{\, \mathrm{d} \mathbf{x}}

\newcommand{\dt}{\, \mathrm{d}t}
\newcommand{\ds}{\, \mathrm{d}\sigma}

\newcommand{\OL}[1]{\textcolor{violet}{\bf{#1}}}

\newcommand{\expect}{\mathbb{E}}

\def\domain{\mathcal D}
\def\p{\partial}
\def\normal{\bf n}
\def\bq{\begin{pmatrix}
 b \\
 q
 \end{pmatrix}}
\def\myH{{2}}
\def\myK{\nabla^\perp(\Delta - 1)^{-1}}

\begin{document}

\title[The Stochastic Thermal Quasi-Geostrophic Model]{Theoretical Analysis and Numerical Approximation for the Stochastic thermal quasi-geostrophic model}
\author{D Crisan, DD Holm, O Lang, PR Mensah, W Pan}
\address{DEPARTMENT OF MATHEMATICS, IMPERIAL COLLEGE, LONDON SW7 2AZ, UK.}
\email{d.crisan@imperial.ac.uk, d.holm@imperial.ac.uk, o.lang15@imperial.ac.uk,} \email{p.mensah@imperial.ac.uk, wei.pan@imperial.ac.uk}

\begin{abstract}
%
     This paper investigates the mathematical properties of a stochastic version of the balanced 2D thermal quasigeostrophic (TQG) model of potential vorticity dynamics. This stochastic TQG model is intended as a basis for parametrisation of the dynamical creation of unresolved degrees of freedom in computational simulations of upper ocean dynamics when horizontal buoyancy gradients and bathymetry affect the dynamics, particularly at the submesoscale (250m--10km). Specifically, we have chosen the SALT (Stochastic Advection by Lie Transport) algorithm introduced in \cite{Holm2015} and applied in \cite{cotter2018modelling, cotter2019numerically} as our modelling approach. The SALT approach preserves the Kelvin circulation theorem and an infinite family of integral conservation laws for TQG. 
The goal of the SALT algorithm is to quantify the uncertainty in the process of up-scaling, or coarse-graining of either observed or synthetic data at fine scales, for use in computational simulations at coarser scales. The present work provides a rigorous mathematical analysis of the solution properties of the thermal quasigeostrophic (TQG) equations with stochastic advection by Lie transport (SALT) \cite{holm2019stochastic,HLP2021}.
\end{abstract}

\maketitle

\color{blue}


\color{black}

\section{Introduction}




The deterministic TQG (thermal QG) model augments the standard QG (quasigeostrophic) model for quasigeostrophically balanced planar incompressible fluid flow. Namely, TQG augments the QG balance between pressure gradient and Coriolis force for small Rossby number by also introducing the thermal gradient. See, e.g., \cite{BeronVera2021,BeronVera2021arXiv,HLP2021} for the history and details about how the TQG model can be derived via a balanced asymptotic expansion in small dimension-free ocean dynamics parameters, as well as results of initial computational simulations. 

In TQG, as in other models of planar incompressible fluid flow, the divergence-free vector field $\bu(x,y,t)$ representing horizontal fluid velocity may be defined in terms of its \emph{stream function} $\psi(x,y,t)$ as
\begin{align}
\bu = \nabla^\perp\psi := \bs{\widehat z}\times\nabla\psi \,.
\label{u-psi-def}
\end{align}
The defining expression for the divergence-free TQG velocity vector field $\bu$ is an assumed balance among Coriolis force, hydrostatic pressure force and also the horizontal gradient of buoyancy, which arises at order $O(Ro)$ in the \emph{formal} asymptotic expansion in Rossby number $Ro\ll1$ of the thermal rotating shallow water (TRSW) equations. Namely,
\begin{equation}
\bu := \zh\times\nabla\psi =: \zh\times\nabla\left(\zeta + \frac{1}{2}b\right) 
 + O(Ro) \,.
\label{trsw-tg}
\end{equation}
In the asymptotic expansion approach to deriving the TQG equation, one now discovers a linear relation among the stream function $\psi$, the surface elevation $\zeta$ and the buoyancy $b$ at order $O(Ro)$ as
\begin{align}
\psi = \zeta + b/2 + O(Ro) \,.
\label{u-psi-balance}
\end{align}
Imposing the linear relation \eqref{u-psi-balance} that arises from the assumed balance relation \eqref{trsw-tg} at order $O(Ro)$ in the asymptotic expansion of the TRSW equations now yields the TQG equations for the dynamics of potential vorticity (PV) $q$ in the horizontal $(x,y)$ plane which also involves the dynamics of the buoyancy $b$. The resulting TQG equations are given by  \cite{HLP2021}
\begin{align}
\begin{split}
\partial_t q+ \bm{u}\cdot \nabla q
&
= \bs{\widehat{z}}\cdot \nabla\Big(\zeta - \frac{1}{2}  h_1\Big)\times\nabla b
=: \bs{\widehat{z}}\cdot \nabla\big(\psi -  h_1/2\big)\times\nabla b\,,
\\
\partial_t b + \mbf{u}\cdot \nabla b &= 0
\quad\hbox{with}\quad 
q:= (\Delta - 1)\psi  + f_1
\,,
\end{split}
\label{TQG-eqns-def}
\end{align}
in which $f_1=\mbf{\widehat{z}}\cdot{\rm curl}\mbf{R}(x,y)$ with local rotation velocity $\mbf{R}(x,y)$ is the Coriolis parameter and $h_1(x,y)$ is the bathymetry, both of whose gradients are assumed to be of order $O(Ro)$. Figure \ref{fig:angrydolphin} shows a snapshot of the solution of the TQG equations for potential vorticity $q$ and buoyancy $b$. 

\begin{figure}[!ht]
    \centering
    \includegraphics[width=\textwidth ]{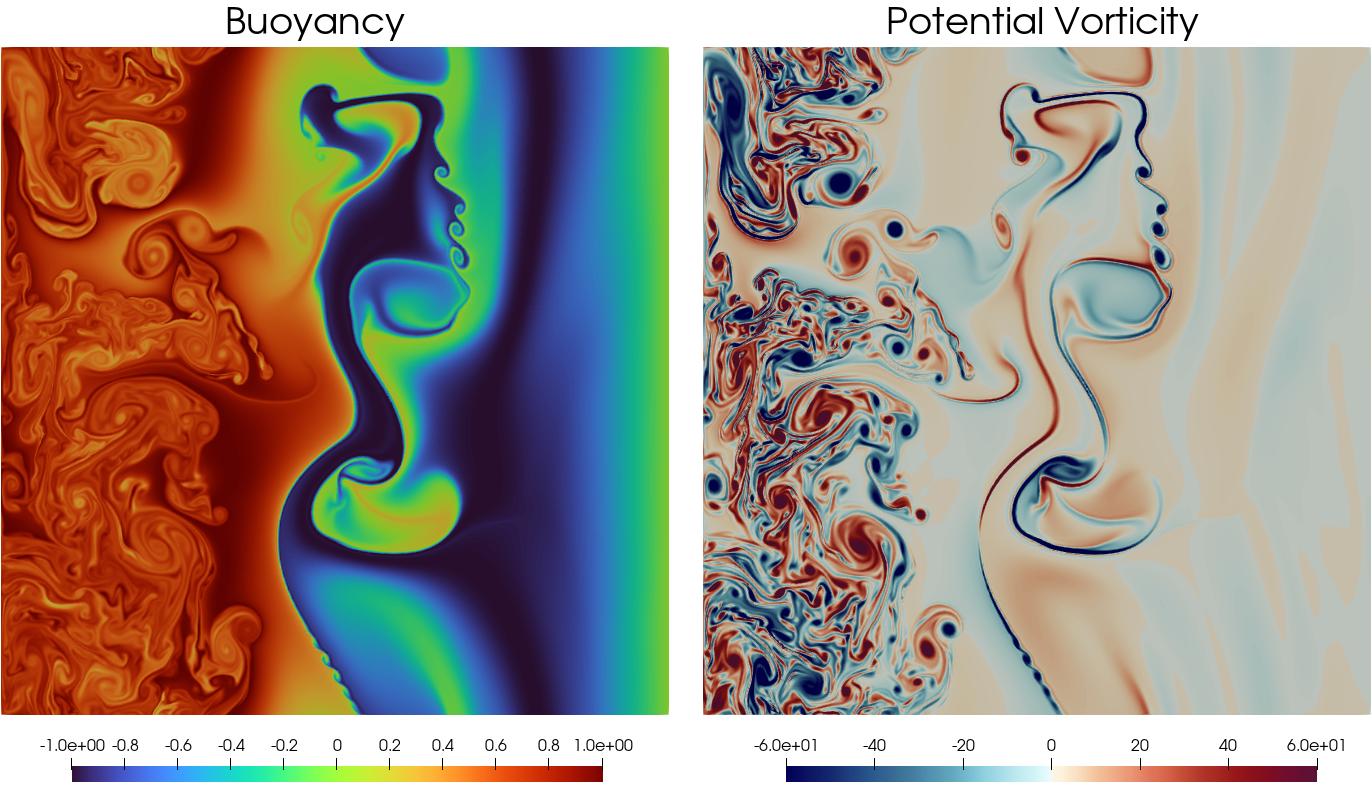}
    \caption{These images are snapshots of buoyancy (left) and potential vorticity (right) from a numerical evolution of the deterministic TQG. The domain is a vertical channel with periodic boundary conditions in $\hat{y}$ and no-slip in $\hat{x}$. It illustrates the result of the creation of a buoyancy front from an initial state and its subsequent emergent instabilities, in which the interaction between buoyancy gradient and bottom topography plays a key role. {Since both buoyancy $(b)$ and PV $(q)$ are advected by the same velocity, eventually their level sets track each other.}}
    \label{fig:angrydolphin}
\end{figure}

\paragraph{\bf The Kelvin circulation theorem for TQG}
To write the Kelvin circulation theorem for TQG, we first introduce the convolution kernel $K$ defined by its action on fluid velocity
\begin{align}
K*\bu = \big(1-\Delta^{-1}\big)\bu
\,.
\label{K-def}    
\end{align}
The utility of the kernel $K$ is that its curl summons the Helmholtz operator; namely,
\begin{align}
{\rm curl}(K*\bu) = \big(\Delta-1\big)\psi
\,.
\label{K-curl}    
\end{align}
The property \eqref{K-curl} under the curl of the kernel $K$ in \eqref{K-def} facilitates writing the Kelvin circulation theorem for TQG as
\begin{align}
\frac{\dd}{\dt}\oint_{c(u)} \big((K*\mbf{u}) + \mbf{R}(x,y)\big)\cdot \dd \mbf{x}
=
\oint_{c(u)} \big(\psi -  h_1/2\big) \nabla b\cdot \dd\mbf{x}
\,.
\label{TQG-KelThm}
\end{align}
Upon applying the Stokes theorem to the Kelvin circulation theorem for TQG in \eqref{TQG-KelThm}, one obtains the rate of change of the flux of potential vorticity $q$ though a planar surface patch whose boundary follows the Kelvin loop, $c(u)$. Namely, with $\bs{\widehat{z}}\cdot{\rm curl}\mbf{R}(x,y)=:f_1(x,y)$, one finds
\begin{align}
\frac{\dd}{\dt}\int_{\partial S=c(u)} \big( (\Delta - 1)\psi  + f_1\big)\dd x\wedge \dd y
=
\frac{\dd}{\dt}\int_{\partial S=c(u)} q\,\dd x\wedge \dd y
=
\int_{\partial S=c(u)} \dd\big(\psi -  h_1/2\big) \wedge \dd b
\,.
\label{TQG-StokesKelThm}
\end{align}

\paragraph{\bf Conservation laws for TQG} 

The deterministic TQG equations in \eqref{TQG-eqns-def} preserve a total energy, given by the sum of kinetic and potential energy as\footnote{An interesting feature is that this kinetic energy penalises a \emph{decrease} in wave number. Therefore, it does not support an inverse energy cascade. This is the source of the high wave-number instability of TQG travelling-wave solution \cite{HLP2021}.}
\begin{align}
E (\bm{u},b)=\int_{\mathcal{D}} \frac{1}{2} \bm{u}\cdot (K*\bm{u}) + \frac{1}{4} (b+h_1)^2 
\,\mathrm{d}x\,\mathrm{d}y\,,
\label{Erg-TQG}
\end{align}
and the kinetic energy involves the $K$ kernel convolution operator $K*$.

Perhaps more surprisingly, the deterministic TQG equations in \eqref{TQG-eqns-def} also preserve an infinity of integral conservation laws, determined by two arbitrary differentiable functions of buoyancy $\Phi(b )$ and $\Psi(b )$ as
\begin{equation}
C_{\Phi,\Psi} = \int_{\mathcal D}\Phi(b ) + q \Psi(b )\,\dd x\dd y\,.
\label{eq:casimirstqg}
\end{equation}
That TQG dynamics preserves the energy in \eqref{Erg-TQG} and the family of integral quantities in \eqref{eq:casimirstqg} can be verified by direct computations. However, the infinity of conservation laws in \eqref{eq:casimirstqg} indicates that the TQG system in \eqref{TQG-eqns-def} may possess a rich mathematical structure which will affect its solution behaviour. We discuss the geometrical aspects of this mathematical structure in Appendix \ref{App-A}.

\subsection{Hamiltonian Stochastic Advection by Lie Transport (SALT) for TQG}
\label{sec:HamiltonSTQG}
Hamiltonian Stochastic Advection by Lie Transport (SALT) can be derived for TQG by augmenting its deterministic Hamiltonian in \eqref{eq:erg-TQG2} to add a stochastic Hamiltonian whose flow under the Lie-Poisson bracket \eqref{TQG-brkt3} induces a Stratonovich stochastic flow along the characteristics of the symplectic vector field generated by the stochastic Hamiltonian process $\langle \varpi \,,\,\varsigma(x,y) \rangle \circ dW_t$. The semimartingale total Hamiltonian is then given by 
\begin{align}
{\rm d}h = {\mc H}_{TQG}(\varpi,b)\dd t + \int_{\mc D} \varpi \,\varsigma(x,y) \dd x \dd y\circ \dd W_t
\,,\label{SALT-HamShift}
\end{align}
whose Lie-algebra valued variations with respect to the \emph{dual} Lie-algebra valued variables $(\varpi ; b)\in (f_1\circledS f_2)^*$ for TQG are expressed in semimartingale form as
\begin{align}
\delta({\rm d}h) = \int \Big(\psi {\rm d}t + \varsigma(x,y) \circ \dd W_t \Big)\delta \varpi
+ \Big(\frac12(b + h_1)\dt - \big(\psi {\rm d}t + \varsigma(x,y) \circ \dd W_t\big)\Big) \delta b
\, \dd x \dd y
\,,\label{SALT-deltaHamShift}
\end{align}
with
\begin{align}
\frac{\delta ({\rm d}h)}{\delta \varpi} = \psi {\rm d}t + \varsigma(x,y) \circ \dd W_t 
\quad\hbox{and}\quad
\frac{\delta ({\rm d}h)}{\delta b} 
= \frac12(b + h_1)\dt - \big(\psi {\rm d}t + \varsigma(x,y) \circ \dd W_t\big)
\,.\label{SALT-deltaHamShift-ex}
\end{align}
SALT TQG dynamics is then expressed geometrically using the coadjoint operator ${\rm ad}^*: \mathfrak{g}\times \mathfrak{g}^*\to \mathfrak{g}^*$  in \eqref{SDP-TQG-LP} in stochastic integral form as
\begin{align}
\big({\rm d} \varpi ; {\rm d}b\big) 
= - \,{\rm ad}^*_{\big({\delta ({\rm d}h)}/{\delta \varpi}\,;\,{\delta ({\rm d}h)}/{\delta b}\big)}
\big(\varpi ; b\big)
\,.
\label{SDP-TQG-LP-SALT}
\end{align}
Thus, SALT TQG solutions evolve by stochastic coadjoint motion under right semidirect-product action of symplectic diffeomorphisms on the dual of its Lie algebra with dual (momentum-map) variables $(\varpi ; b)\in (f_1\circledS f_2)^*$. The symplectic semimartingale vector fields
$({\delta ({\rm d}h)}/{\delta \varpi}\,;\,{\delta ({\rm d}h)}/{\delta b})\in f_1\circledS f_2$ whose characteristic curves generate the stochastic coadjoint solution behaviour are given in equations \eqref{SALT-deltaHamShift-ex}. 

One may write the SALT TQG equations \eqref{SDP-TQG-LP-SALT} in Hamiltonian  matrix operator form as, cf. equation \eqref{TQG-Ham3},
\begin{align}
\begin{split}
 {\rm d} 
\begin{bmatrix}
\varpi \\ b
\end{bmatrix}
&=
\begin{bmatrix}
J(\varpi,\,\cdot\,) & J(b,\,\cdot\,) 
\\ 
J(b,\,\cdot\,)  & 0
\end{bmatrix}
\begin{bmatrix}
{\delta ({\rm d}h) } /{\delta \varpi} = \psi {\rm d}t + \varsigma(x,y) \circ \dd W_t
\\ 
{\delta ({\rm d}h)}/{\delta b} = \frac12(b + h_1)\dt - (\psi {\rm d}t + \varsigma(x,y) \circ \dd W_t)
\end{bmatrix}
\\
 {\rm d} 
\begin{bmatrix}
q \\ b
\end{bmatrix}
&= 
\begin{bmatrix}
J\big(q-b,\,\psi {\rm d}t + \varsigma(x,y) \circ \dd W_t\big) 
+ J\big(b,\,\frac12h_1 \dt \big) 
\\ 
J\big(b,\,\psi {\rm d}t + \varsigma(x,y) \circ \dd W_t\big)
\end{bmatrix}
\\&= -
\begin{bmatrix}
\big(\bm{u}\dt + \bs{\xi}(x,y) \circ \dd W_t\big)\cdot\nabla\big(q-b\big) 
+\bu_h \cdot \nabla b \dt 
\\ 
\big(\bm{u}\dt + \bs{\xi}(x,y) \circ \dd W_t\big)\cdot\nabla b
\end{bmatrix}
\,.
\end{split}
\label{TQG-HamSALT}
\end{align}
In the last step here, we have introduced notation for the fixed symplectic vector field $\bs{\xi}(x,y):=\nabla^\perp\varsigma$ that in principle must be obtained from observed data. Discussions of the methods for acquiring these fixed symplectic vector fields from observed data are beyond the scope of the present paper. However, our intention is to follow the methods of data analysis and data assimilation in  \cite{cotter2018modelling, cotter2019numerically} developed in our previous work for 2D Euler fluid equations and standard QG. Before passing to data analysis, though, one must study the solution properties of the SALT TQG equations. These solution properties of the Hamiltonian stochastic TQG equations in \eqref{TQG-HamSALT} will be investigated in the remainder of this paper. 

The SALT approach via Hamilton's variational principle preserves all of the geometric structure of the underlying deterministic equations \cite{Holm2015}. However, as we have seen for TQG, a derivation of the deterministic equations using Hamilton's variational principle may not always be available, especially when asymptotic expansions are applied to approximate higher level mathematical models. For TQG we have met this situation by using the residual Lie-Poisson Hamiltonian structure possessed by the TQG equations. Before we pass to the analysis of the solution properties of the SALT TQG equations, let us briefly mention an alternative approach to developing stochastic fluid equations in the QG family whose dynamics exactly preserves energy.

{The SALT approach preserves ideal fluid flow properties such as the Kelvin circulation theorem and conservation laws arising from the Lie algebraic structure and Lie-Poisson Hamiltonian properties. These fluid circulation properties and conservation laws are physical principles which benefit both the theoretical analysis and the physicality of computational simulation results. }

\begin{remark}[Energy preservation arising from Stochastic Forcing by Lie Transport (SFLT)]$\,$

Instead of SALT, one could have introduced an energy-preserving form of stochasticity for TQG. The energy-preserving SFLT approach can be written by taking the Lie-Poisson Hamiltonian matrix operator in \eqref{TQG-HamSALT} to be a semimartingale, as follows \cite{HH2021}
\begin{align}
\begin{split}
 {\rm d} 
\begin{bmatrix}
\varpi\\ b
\end{bmatrix}
&=
\begin{bmatrix}
J(\varpi \dt + \xi(x,y) \circ \dd W_t ,\,\cdot\,) & J(b \dt + \xi(x,y) \circ \dd W_t ,\,\cdot\,) 
\\ 
J(b \dt + \xi(x,y) \circ \dd W_t ,\,\cdot\,)  & 0
\end{bmatrix}
\begin{bmatrix}
{\delta {\mc H}_{TQG} } /{\delta \varpi} = \psi 
\\ 
{\delta ({\mc H}_{TQG})}/{\delta b} = \frac12(b + h_1)- \psi
\end{bmatrix}
\\&= - 
\begin{bmatrix}
(\bu\cdot \nabla)\big((\varpi - b) \dt + \xi(x,y) \circ \dd W_t \big) 
- \nabla^\perp\frac12(b + h_1) \cdot \nabla\big(b \dt + \xi(x,y) \circ \dd W_t \big) 
\\ 
(\bu\cdot \nabla) \big(b \dt + \xi(x,y) \circ \dd W_t \big)  
\end{bmatrix}
\,,
\end{split}
\label{TQG-SFLTeqn}
\end{align}
with velocity $\bu =\nabla^\perp \psi$. Because the Hamiltonian matrix operator for the  Lie-Poisson bracket in \eqref{TQG-SFLTeqn} is skew-symmetic under the $L^2$ pairing, one has energy conservation in the form ${\rm d}{\mc H}_{TQG}=0$.
Thus, SFLT preserves the deterministic energy Hamiltonian, ${\mc H}_{TQG}$, although its dynamics is stochastic. Here we do not follow the alternative SFLT approach, though, because it does not preserve the deterministic Casimir conservation laws in \eqref{eq:casimirstqg} that are preserved in the SALT approach by taking the deterministic form of the Lie-Poisson operator while taking the Hamiltonian to be a semimartingale. This approach has been applied to introduce SFLT into the 3D Primitive Equations for ocean circulation dynamics in \cite{HuPatching2022}. 
\end{remark}

\subsection{Content and plan of the paper}
In the sections above, we have explained   the underlying geometric mechanics background for stochastic modelling with SALT in the case of TQG. This is done in parallel with the stochastic modelling with SALT of the classic problem of Euler-Boussinesq convection (EBC), which is considerably simpler than the case of TQG. Geometrically, the dynamics of both TQG and EBC are understood as coadjoint Hamiltonian motion generated by the semidirect-product action of the Lie algebra $(f_1\circledS f_2)$ on function pairs $(f_1;f_2)\in (f_1\circledS f_2)^*$. This is the proper geometric framework for the application of the SALT approach to stochastic modelling in fluid dynamics. 

In Section \ref{sec:MainWellPose}, we begin by making precise the notion of a solution that we wish to construct, after which we state our main analytical results. Here, we are interested in the construction of a unique maximal strong pathwise solution to the SALT TQG equation. This solution is strong in both the stochastic and deterministic sense. 

The proof of the existence of the solution described above will rely on ideas developed in  \cite{breit2018local,breit2018stoch, breit2019stochastic, glatt2009strong, glatt2012local,  mensah2019theses}. In particular, the book \cite{breit2018stoch} on the mathematical analysis of stochastic compressible fluids will serve as our main guide.
In this regard, we begin the proof of our main result in Section \ref{sec:pathSolTrunc} where we construct a strong pathwise solution for an approximation of the SALT TQG whose nonlinearities have been truncated and are subjected to a subclass of initial conditions that have finite moments.
Here, we follow  the classical Yamada--Watanabe-type argument where the pathwise solution is derived from the construction of a stochastically weak solution and the establishment of pathwise uniqueness for the truncated system. The construction of the stochastically weak solution is achieved from a finite-dimensional approximation.

Our next goal will be to show the existence of a strong pathwise solution for the original system by removing the cut-off functions introduced into the system as well as the restriction that the initial condition is of bounded moments. This is carried out in Section \ref{sec:pathSolOrig}. To begin with, we trade an appropriate stopping time for the cut-off function for the nonlinear terms. In so doing, we obtain a local strong pathwise solution for the original system in place of the solution obtained for the truncated system. Having removed the cut-offs, we then proceed to remove the additional boundedness assumption imposed on our initial conditions. This is done by introducing yet another cut-off function, but one which cuts the initial conditions rather than the nonlinear terms. This cut-off allows us to deduce the existence of a local strong pathwise solution for the original system subject to general initial conditions, by piecing together `smaller' local solutions. Next, we deduce maximality of this local solutions after which we show uniqueness of these maximal solutions.
This is a straightforward adaptation of the earlier uniqueness result for the truncated system to subspaces $\Omega_K$ of the sample space $\Omega$. We finally pass to the limit with $K\rightarrow\infty$ to complete the proof of our main result in Theorem \ref{thm:main}.
    
    In Section  \ref{sec:blowup} we give a blow-up criterion for the breakdown of solutions to the SALT TQG. This may be considered as a stochastic analogue of the deterministic result for the Euler equation by Beale, Kato and Majda \cite{beale1984remarks} which has since been extended to the stochastic setting by Crisan, Flandoli and Holm \cite{crisan2019solution}. This result also complements the blowup criterion for the deterministic TQG recently shown in \cite{crisan2022breakdown} on the whole space. \footnote{We are grateful to T. Beale for thoughtful correspondence about the derivation of the celebrated BKM result in the case of periodic boundary conditions.}

    In Section \ref{sec:consitency}, we describe the adaptation of a deterministic explicit third order Runge-Kutta scheme to the SALT TQG system. Using the wellposedness results from earlier sections, we prove the scheme is numerically \emph{consistent} for the truncated SALT TQG system. For completeness, in the Appendix we describe the finite element method we use for the spatial derivatives. However, a full investigation of the numerics is beyond the scope of this paper. We also give the proof of Lemma \ref{rem:xiAppendix} that is used in the construction of the solution in the Appendix.

\section{Construction of a solution}
\label{sec:MainWellPose}


\subsection{Notations}
Our independent variables consists of spatial points $x:=\mathbf{x}=(x, y)\in \mathbb{T}^2$ on the $2$-torus $\mathbb{T}^2$   and a time variable  $t\in [0,T]$, $T>0$. For functions $F$ and $G$, we write $F \lesssim G$  if there exists  a generic constant $c>0$  such that $F \leq c\,G$.
We also write $F \lesssim_p G$ if the  constant  $c(p)>0$ depends on a variable $p$. If $F \lesssim G$ and $G\lesssim F$ both hold (respectively,  $F \lesssim_p G$ and $G\lesssim_p F$), we use the notation $F\sim G$ (respectively, $F\sim_p G$).
The symbol $\vert \cdot \vert$ may be used in four different contexts. For a scalar function $f\in \mathbb{R}$, $\vert f\vert$ denotes the absolute value of $f$. For a vector $\bff\in \mathbb{R}^2$, $\vert \bff \vert$ denotes the Euclidean norm of $\bff$. For a square matrix $\mathbb{F}\in \mathbb{R}^{2\times 2}$, $\vert \mathbb{F} \vert$ shall denote the Frobenius norm $\sqrt{\mathrm{trace}(\mathbb{F}^T\mathbb{F})}$. Finally, if $S\subseteq  \mathbb{R}^2$ is  a (sub)set, then $\vert S \vert$ is the $2$-dimensional Lebesgue measure of $S$.
\\
For $k\in \mathbb{N}\cup\{0\}$ and $p\in [1,\infty]$, we denote by $W^{k,p}(\mathbb{T}^2)$, the Sobolev space of Lebesgue measurable functions whose weak derivatives up to order $k$ belongs to $L^p(\mathbb{T}^2)$. Its associated  norm is
\begin{align}
\Vert v \Vert_{W^{k,p}(\mathbb{T}^2)} =\sum_{\vert \beta\vert\leq k} \Vert \partial^\beta v \Vert_{L^{p}(\mathbb{T}^2)},
\end{align}
where $\beta$ is a $2$-tuple multi-index of nonnegative integers  of length $\vert \beta \vert \leq k$.
The Sobolev space $W^{k,p}(\mathbb{T}^2)$ is a Banach space. Moreover, $W^{k,2}(\mathbb{T}^2)$ is a Hilbert space when endowed with the inner product
\begin{align}
\langle u,v \rangle_{W^{k,2}(\mathbb{T}^2)} =\sum_{\vert \beta\vert\leq k} \langle \partial^\beta u\,,\, \partial^\beta v \rangle,
\end{align}
where $\langle\cdot\,,\,\rangle$ denotes the standard $L^2$-inner product.  
In general, for $s\in\mathbb{R}$, we will define the Sobolev space $H^s(\mathbb{T}^2)$  as consisting of  distributions $v$ defined on $\mathbb{T}^2$ for which the norm
\begin{align}
\label{sobolevNorm}
\Vert  v\Vert_{H^s(\mathbb{T}^2)}
=
 \bigg(\sum_{\xi \in \mathbb{Z}^2} \big(1+\vert \xi\vert^2  \big)^s\vert  \widehat{v}(\xi)\vert^2
  \bigg)^\frac{1}{2}
  \equiv
  \Vert  v\Vert_{W^{s,2}(\mathbb{T}^2)}
\end{align}
defined in frequency space is finite. Here, $\widehat{v}(\xi)$ denotes the Fourier coefficients  of $v$.
To shorten notation, we will write $\Vert  \cdot\Vert_{s,2}$ for $\Vert \cdot\Vert_{W^{s,2}(\mathbb{T}^2)}$
and/or $\Vert \cdot\Vert_{H^s(\mathbb{T}^2)}$.
When $k=s=0$, we get the usual $L^2(\mathbb{T}^2)$ space whose norm we will denote by $\Vert \cdot \Vert_2$ for simplicity.
We will also use a similar convention for norms $\Vert \cdot \Vert_p$ of general $L^p(\mathbb{T}^2)$ spaces for any $p\in [1,\infty]$ as well as for the inner product $\langle\cdot,\cdot \rangle_{k,2}:=\langle\cdot,\cdot \rangle_{W^{k,2}(\mathbb{T}^2)}$ when $k\in \mathbb{N}$. Additionally, we will denote by  $W^{k,p}_{\mathrm{div}}(\mathbb{T}^2)$, the space of divergence-free vector-valued functions in $W^{k,p}(\mathbb{T}^2)$.

\subsection{Preliminary estimates}
In this section, we collect some useful estimates we shall use throughout our analysis.
We begin  with the following result, see \cite{CHLMP2021}, which follows from a direct computation  using the definition \eqref{sobolevNorm} of the Sobolev norms.
\begin{lemma} 
\label{lem:constrt0}
Let $k\in \mathbb{N}\cup\{0\}$ and assume that the triple $(\bu,w)$ satisfies
\begin{align}
\label{constrt0}
\bu =\nabla^\perp \psi, 
\qquad
w=(\Delta -1) \psi.
\end{align}
If $w\in W^{k,2}(\mathbb{T}^2)$, then the following estimate
\begin{align}
\label{lem:MasterEst}
\Vert \bu \Vert_{k+1,2}^2  &\lesssim \Vert w \Vert_{k,2}^2
\end{align}
holds.
\end{lemma}
In the following, we let $\mathcal{L}_{\bm{\xi}_i}:=\bm{\xi}_i\cdot\nabla$ be a two-dimensional advection operator with $\mathcal{L}_{\bm{\xi}_i}^2=(\bm{\xi}_i\cdot\nabla)\bm{\xi}_i\cdot\nabla$.
\begin{lemma}
\label{rem:xiAppendix}
Let $\bm{\xi}_i\in W^{k+1,\infty}_{\mathrm{div}}(\mathbb{T}^2)$ be such that $
\sum_{i\in \mathbb{N}} \Vert \bm{\xi}_i \Vert_{k+1,\infty}<\infty$. Then for any $b\in W^{k,2}(\mathbb{T}^2)$ and $q\in W^{k-1,2}(\mathbb{T}^2)$  where $k\in \{0,1,2,3\}$, we have that
\begin{align*}
&\big\langle \mathcal{L}_{\bm{\xi}_i}^2 b\,, \,b \big\rangle_{k,2}  
+
\big\langle \mathcal{L}_{\bm{\xi}_i} b\,, \,\mathcal{L}_{\bm{\xi}_i} b \big\rangle_{k,2} 
\lesssim_k \Vert b \Vert_{k,2}^2,
\\
&\big\langle \mathcal{L}_{\bm{\xi}_i}^2 (q-2b)\,, \,q \big\rangle_{k-1,2}  
+
\big\langle  \mathcal{L}_{\bm{\xi}_i} (q-b)\,, \,\mathcal{L}_{\bm{\xi}_i} (q-b) \big\rangle_{k-1,2} 
\lesssim_{k} \Vert q \Vert_{k-1,2}^2+ \Vert b \Vert_{k,2}^2.
\end{align*}
\end{lemma}
For the proof of the first inequality, see \cite{crisan2019well}. We provide the proof of the second inequality in the Appendix. 
Finally, let us now recall some Moser-type calculus (commutator estimates) which can be found in  \cite{klainerman1981singular}, for example.
\begin{enumerate}
    \item[I.] Assume that $u,v\in W^{k,2}(\mathbb{T}^2)$. Then for any multi-index $\beta$ with $\vert \beta\vert\leq k$, we have
    \begin{align}
        \Vert \partial^\beta (uv) \Vert_2 \lesssim_k \Vert u \Vert_\infty \Vert \nabla^k v\Vert_2 + \Vert \nabla^k u \Vert_2 \Vert v \Vert_\infty 
    \end{align}
    \item[II.] Assume that $u \in W^{k,2}(\mathbb{T}^2) \cap W^{1,\infty}(\mathbb{T}^2)$ and $v \in W^{k-1,2}(\mathbb{T}^2) \cap L^{\infty} (\mathbb{T}^2)$. Then for any multi-index $\beta$ with $\vert \beta\vert\leq k$, we have
\begin{equation}
\left\| \partial^\beta (uv) - u \partial^\beta v \right\|_{2} \lesssim_k
\left( \| \nabla u \|_{\infty} \|  \nabla^{k-1} v \|_2 + \| \nabla^k u \|_2
\| v \|_\infty  \right).
\end{equation}
\item[III.] Assume that $u \in W^{k,p_3}(\mathbb{T}^2) \cap W^{1,p_1}(\mathbb{T}^2)$ and $v \in W^{k-1,p_2}(\mathbb{T}^2) \cap L^{p_4} (\mathbb{T}^2)$ where   $k\leq 2$ and $p,p_2,p_3\in(1,\infty)$ and $p_1,p_4 \in (1,\infty]$ are such that
\begin{align*}
\frac{1}{p}=\frac{1}{p_1}+\frac{1}{p_2}=\frac{1}{p_3}+\frac{1}{p_4}.
\end{align*}
Then for any multi-index $\beta$ with $\vert \beta\vert\leq k$, we have
\begin{equation}
\left\| \partial^\beta (uv) - u \partial^\beta v \right\|_{p} \lesssim_k
\left( \| \nabla u \|_{p_1} \|  \nabla^{k-1} v \|_{p_2} + \| \nabla^k u \|_{p_3}
\| v \|_{p_4}  \right).
\end{equation}
\end{enumerate}

In the following, we let  $(\bm{\xi}_i)_{i\in \mathbb{N}}$, $\bu_h$ and $f$ be given time-independent functions representing a collection of symplectic vector fields obtained from observed data,  the skew bathymetry gradient, and the Coriolis parameter of a fluid, respectively. 

The equations of Stochastic Advection by Lie Transport for the Thermal Quasi-Geostrophic (SALT TQG), or simply, the Stochastic Thermal Quasi-Geostrophic (STQG) equations which are written in matrix form in \eqref{TQG-HamSALT} may also be expressed equivalently as a pair of coupled equations governing the evolution of the buoyancy $b$ and the potential vorticity $q$ in the following way
\begin{align} 
\dd b + (\bu \cdot \nabla) b \dt +(\bm{\xi}_i\cdot\nabla) b \circ \dd W^i =0,
\label{ce}
\\
\dd q + (\bu\cdot \nabla)( q -b)\dt + (\bm{\xi}_i\cdot\nabla) (q-b) \circ \dd W^i= -(\bu_h \cdot \nabla) b \dt, \label{me}
\\
b(0,\bm{x})=b_0(\bm{x}),\qquad q(0,\bm{x})=q_0(\bm{x}),
\label{initialTQG}
\end{align}
where we have applied the Einstein convention of summing repeated indices over their range and where
\begin{align}
\label{constrt}
\bu =\nabla^\perp \psi, 
\qquad
\bu_h =\frac{1}{2} \nabla^\perp h,
\qquad
q=(\Delta -1) \psi +f.
\end{align}
Here, $\psi$ is the streamfunction and $h$ is the spatial variation around a constant bathymetry profile.\\
Our given set of data is  $((\bm{\xi}_i)_{i\in \mathbb{N}}, \bu_h, f, b_0, q_0)$ with the following properties:
\begin{equation}
\begin{aligned}
\label{dataMain}
\bm{\xi}_i&\in W^{4,\infty}_{\mathrm{div}}(\mathbb{T}^2),
\quad \sum_{i\in \mathbb{N}} \Vert \bm{\xi}_i \Vert_{4,\infty}<\infty,\quad\bu_h \in W^{3,2}_{\mathrm{div}}(\mathbb{T}^2), \quad  f\in W^{2,2}(\mathbb{T}^2)\quad \text{and}
\\
&(b_0, q_0) \in  W^{3,2}(\mathbb{T}^2) \times W^{2,2}(\mathbb{T}^2) \text{ is a pair of } \mathcal{F}_0\text{-measurable random variables.}
\end{aligned}
    \end{equation}
\begin{remark}
The assumption on the time-independent divergence-free
vector fields $\bm{\xi}_i:\mathbb{T}^2 \rightarrow \mathbb{R} ^2 $ are to
ensure that the infinite sum of stochastic integrals in \eqref{ce}--\eqref{me}
are well-defined.
\end{remark}
\begin{remark}
\label{rem:meanzero}
Since we are working on the torus, and the velocity fields are defined by \eqref{constrt}, we have in
particular, $\int_{\mathbb{T}^2}\bu_h \dx=0$ and $\int_{\mathbb{T}^2}\bu \dx=0$. 
Consequently, for simplicity, we will assume that all
functions under consideration have zero averages.
\end{remark}
\begin{remark}
For our theoretical analysis, it is more convenient to express \eqref{ce}--\eqref{me} in It\^o form. To see clearly how Stratonovich to It\^o conversion works for these equations, it is useful to rewrite them in the following compact form
\begin{align}
\label{abstractTQG}
\dd \binom{b}{q} =
\mathcal{A}
\begin{pmatrix}
 b \\
 q
 \end{pmatrix}\dt
 +
\mathcal{G}_i
\begin{pmatrix}
 b \\
 q
 \end{pmatrix}\circ \dd W^i
,
 \qquad
\mathcal{A}
:= -
\begin{bmatrix}
\bu\cdot\nabla & 0    \\[0.3em]
(\bu_h-\bu)\cdot\nabla &\bu\cdot \nabla
\end{bmatrix}
,
\quad 
\mathcal{G}_i
:= -
\begin{bmatrix}
\bm{\xi}_i\cdot\nabla & 0    \\[0.3em]
-\bm{\xi}_i\cdot\nabla &\bm{\xi}_i\cdot \nabla
\end{bmatrix}.
\end{align}
Furthermore, the conversion from the Stratonovich to It\^o integral $\mathcal{G}_i\mathbf{g}\circ \dd W^i \mapsto \frac{1}{2}\mathcal{G}_i^2\mathbf{g}\dt+\mathcal{G}_i\mathbf{g}\dd W^i$, $\mathbf{g}:=(b,q)^T$ yields the following equivalent form for \eqref{ce}--\eqref{me} 
\begin{align}
\dd b + (\bu \cdot \nabla) b \dt
-
\frac{1}{2}
(\bm{\xi}_i\cdot\nabla)(\bm{\xi}_i\cdot\nabla)b 
\dt
+(\bm{\xi}_i\cdot\nabla) b \, \dd W^i =0,
\\
\dd q + (\bu\cdot \nabla)( q -b)\dt -
\frac{1}{2}(\bm{\xi}_i\cdot\nabla)(\bm{\xi}_i\cdot\nabla) (q-2b)\dt 
 + (\bm{\xi}_i\cdot\nabla) (q-b)  \dd W^i= -(\bu_h \cdot \nabla) b \dt.
\end{align}
\end{remark}
Our main goal is to construct a  solution of \eqref{ce}--\eqref{constrt} that lives on a maximal time interval and  that is strong in both PDE and probabilistic sense. To make this clearer, let us first make the following definitions.
\begin{definition}[Local strong pathwise solution]
\label{def:locStrongPathSol}
Let $(\Omega, \mathcal{F}, (\mathcal{F}_t)_{t\geq 0}, \mathbb{P})$ be a stochastic basis and $(W^i)_{i \in \mathbb{N}}$ a sequence of independent one-dimensional Brownian motions adapted to the complete and right-continuous filtration $(\mathcal{F}_t)_{t\geq 0}$. Let $((\bm{\xi}_i)_{i\in \mathbb{N}}, \bu_h, f, b_0, q_0)$ be a dataset satisfying \eqref{dataMain}. A triplet $(b,q,\tau)$ is called a \textit{local strong pathwise solution} of \eqref{ce}--\eqref{constrt} if:
\begin{itemize}
    \item $\tau$ is a $\mathbb{P}$-a.s. strictly positive $(\mathcal{F}_t)$-stopping time;
    \item $(b,q)$ is a pair of $(\mathcal{F}_t)$-progressively measurable stochastic processes such that
    \begin{align*}
        b(\cdot \wedge \tau) \in C([0,T];W^{3,2}(\mathbb{T}^2)), \qquad
        q(\cdot \wedge \tau) \in C([0,T];W^{2,2}(\mathbb{T}^2)) \quad \mathbb{P}\text{-a.s.;}
    \end{align*}
    \item the equations
    \begin{align*}
b(t\wedge \tau) =& b_0 - \int_0^{t\wedge \tau} (\bu \cdot \nabla) b \ds +\frac{1}{2} \int_0^{t\wedge \tau} (\bm{\xi}_i\cdot\nabla) (\bm{\xi}_i\cdot\nabla) b  \ds
- \int_0^{t\wedge \tau} (\bm{\xi}_i\cdot\nabla) b  \dd W^i_\sigma,
\\
q(t\wedge \tau) =& q_0 - \int_0^{t\wedge \tau}(\bu\cdot \nabla)( q -b)\ds -\int_0^{t\wedge \tau}(\bu_h \cdot \nabla) b \ds 
+
\frac{1}{2}\int_0^{t\wedge \tau} (\bm{\xi}_i\cdot\nabla)(\bm{\xi}_i\cdot\nabla) (q-\OL{2}b)  \ds
\\
&- \int_0^{t\wedge \tau} (\bm{\xi}_i\cdot\nabla) (q-b) \dd W^i_\sigma,
\end{align*}
    hold $\mathbb{P}$-a.s. for all $t\in[0,T]$.
\end{itemize}
\end{definition}

\begin{definition}[Maximal strong pathwise solution]
\label{def:MaxStrongPathSol}
Fix the preamble as contained in Definition \ref{def:locStrongPathSol}. A set $(b,q,(\tau_R)_{R\in \mathbb{N}}, \tau)$ is called a \textit{maximal strong pathwise solution} of  \eqref{ce}--\eqref{constrt} if:
\begin{itemize}
    \item $\tau$ is a $\mathbb{P}$-a.s. strictly positive $(\mathcal{F}_t)$-stopping time;
    \item $(\tau_R)_{R\in \mathbb{N}}$ is an increasing sequence of $(\mathcal{F}_t)$-stopping times such that $\tau_R <\tau$ on the set $[\tau<T]$, $\lim_{R\rightarrow \infty} \tau_R = \tau$ a.s. and
    \begin{align*}
        \sup_{t\in [0,\tau_R]}\big( \Vert \nabla b(t) \Vert_\infty + \Vert \nabla \bu(t) \Vert_\infty+ \Vert q(t) \Vert_\infty \big)\geq R \quad \text{on} \quad [\tau<T];
    \end{align*}
    \item each $(b,q,\tau_R), R\in \mathbb{N}$ is a local strong pathwise solution.
\end{itemize}
\end{definition}
\begin{remark}
\label{rem:weakBlowup}
Here, $\tau$  marks the maximal lifespan of the solution, which is determined by the time of explosion of the $L^\infty(\mathbb{T}^2)$-norm of either the buoyancy gradient, the velocity gradient, or the potential vorticity. Indeed, by constructing a log-Sobolov estimate in the spirit of Beale--Kato--Madja \cite[Eq. (15)]{beale1984remarks} for the Euler equation, one may control the lifespan of the STQG by just the buoyancy gradient or the potential vorticity.
\end{remark}
\begin{remark}
Our maximal solution is defined within the fixed, but arbitrary, interval $[0,T]$. By piecing together these finite interval maximal solutions, we can obtain a maximal solution without a constraint on a final time $T$.
\end{remark}
With these definitions in hand, we are now in the position to state our main result.
\begin{theorem}[Existence of a unique maximal solution]
\label{thm:main}
For $((\bm{\xi}_i)_{i\in \mathbb{N}}, \bu_h, f, b_0, q_0)$ satisfying \eqref{dataMain}, there exists a unique maximal strong pathwise
solution $(b,q, (\tau_R)_{R\in \mathbb{N}},\tau)$ of \eqref{ce}--\eqref{constrt}.
\end{theorem}
\begin{remark}
Uniqueness of a maximal strong pathwise solution as stated in our main result is to be understood in the sense that only the triplet $(b,q,\tau)$ is unique.
\end{remark}
In the following, we state a weak stability result for a superset of the maximal strong pathwise solutions to be constructed in Theorem \ref{thm:main}. This will consist of solutions that emanate from a subset of the dataset \eqref{dataMain} with bounded-in-$\omega$ initial conditions $(b_0,q_0)$.
\begin{theorem}[Weak continuity with respect to bounded initial data]
\label{thm:stability}
Define $G(t)$ as follow
\begin{align}
G(t):= c\,\bigg( 1+ \Vert \bu_h \Vert_{3,2}+\Vert f \Vert_{2,2}
+\sum_{i=1}^2 \Vert b_i(t) \Vert_{3,2}^3 +\sum_{i=1}^2\Vert q_i(t) \Vert_{2,2}^3 \bigg).
\end{align}
Any pair of maximal strong pathwise solutions $(b^j,q^j, (\tau_R^j)_{R\in \mathbb{N}},\tau^j)$, $j=1,2$ with respective dataset  $((\bm{\xi}_i)_{i\in \mathbb{N}}, \bu_h, f, b_0^j, q_0^j)$, $j=1,2$ satisfying \eqref{dataMain} and $(b_0^j, q_0^j)\in L^\infty(\Omega; W^{3,2}(\mathbb{T}^2)\times W^{2,2}(\mathbb{T}^2))$, $i=1,2$ will satisfy the bound
\begin{align*}
\mathbb{E}\Big[e^{-\int_0^t G(\sigma)\dd \sigma}\big(
    \Vert(b^1-b^2)(t) \Vert_{2,2}^2 +  \Vert(q^1-q^2)(t) \Vert_{1,2}^2
    \big) \Big]
    \leq 
     \mathbb{E}
 \big(
    \Vert b^1_0-b^2_0 \Vert_{2,2}^2 +  \Vert q^1_0-q^2_0 \Vert_{1,2}^2
    \big)
\end{align*}
$\mathbb{P}$-a.s. for all $t\in[0,\tau^1 \wedge \tau^2)$.
\end{theorem}
We now devote the following subsections to the proof of Theorem \ref{thm:main}.

\subsection{Strong pathwise solutions of the truncated system}
\label{sec:pathSolTrunc}
In the following, for all $n\in \mathbb{N}$, we let $P_n$ be the orthogonal projection operator  mapping $L^2(\mathbb{T}^2)$ onto $X_n:=\mathrm{span}\{\phi_1, \ldots, \phi_n\}$ where $\{ \phi_k\}_{k\in\mathbb{N}}$ is a complete
orthonormal system, see \cite[Page 104]{breit2018stoch}, \cite[Chapter 3]{grafakos2008classical} for further details. In particular, we recall that the $P_n$'s are continuous operators on all Hilbert spaces under consideration.
Finally, for a fixed $R>0$, we let  $\theta_R:[0,\infty) \rightarrow [0,1]$ be a smooth cut-off function satisfying
\begin{align}
\label{cutoff}
\theta_R(z)
= \left\{
  \begin{array}{lr}
    1 & : 0\leq z\leq R,\\
    0 & : z \geq R+1.
  \end{array}
\right.
\end{align}
Now we consider the dataset $((\bm{\xi}_i)_{i\in \mathbb{N}}, \bu_h, f, b_0, q_0)$ satisfying 
\begin{equation}
\begin{aligned}
\label{dataGarlerkin}
&\bm{\xi}_i\in W^{4,\infty}_{\mathrm{div}}(\mathbb{T}^2),
\qquad \sum_{i\in \mathbb{N}} \Vert \bm{\xi}_i \Vert_{4,\infty}<\infty,\qquad\bu_h \in W^{3,2}_{\mathrm{div}}(\mathbb{T}^2), \qquad  f\in W^{2,2}(\mathbb{T}^2),
\\
(b_0, &q_0) \in L^\infty\big(\Omega; W^{3,2}(\mathbb{T}^2) \times W^{2,2}(\mathbb{T}^2)\big) \text{ is a pair of } \mathcal{F}_0\text{-measurable random variables.}
\end{aligned}
\end{equation}
Our goal now is to construct a solution $(b,q)$ of
\begin{align}
\dd b + \theta_R[(\bu \cdot \nabla) b ]\dt + (\bm{\xi}_i\cdot\nabla) b \circ \dd W^i =0,
\label{ceCut}
\\
\dd q + \theta_R[(\bu\cdot \nabla)( q -b)]\dt + (\bu_h \cdot \nabla) b \dt  + (\bm{\xi}_i\cdot\nabla) (q- b) \circ \dd W^i=0, \label{meCut}
\\
b(0,x)=b_0(x),\qquad q(0,x)=q_0(x)
\label{initialTQGCut}
\end{align}
in  $W^{3,2}(\mathbb{T}^2) \times W^{2,2}(\mathbb{T}^2)$ for the given dataset \eqref{dataGarlerkin} and where
\begin{align}\label{eq: theta_r}
\theta_R:=
\theta_R(\Vert \nabla b \Vert_{\infty}+\Vert \nabla \bu \Vert_{\infty}+ \Vert q \Vert_{\infty})
\end{align}
is the cut-off function \eqref{cutoff}.
As before,
\begin{align}
\label{constrtCut}
\bu =\nabla^\perp \psi, 
\qquad
\bu_h =\frac{1}{2} \nabla^\perp h,
\qquad
q=(\Delta -1) \psi +f.
\end{align}
We will achieve this goal by using a Galerkin approximation. In this regard, for the given data \eqref{dataGarlerkin}, we can first construct a finite-dimensional solution $ (b_n,q_n)\in L^2(\Omega; C([0,T]; X_n\times X_n)$ of
\begin{align}
\dd b_n + \theta_R^nP_n [(\bu_n \cdot \nabla) b_n ]\dt
+\frac{1}{2}P_n \big[(\bm{\xi}_i\cdot\nabla)P_n[ (\bm{\xi}_i\cdot\nabla) b_n]\big]  \dt + P_n[(\bm{\xi}_i\cdot\nabla) b_n ]\dd W^i =0,
\label{ceCutGar}
\\
\dd q_n +  \theta_R^n P_n [(\bu_n \cdot \nabla)( q_n -b_n) ]\dt 
+\frac{1}{2}P_n \big[(\bm{\xi}_i\cdot\nabla)P_n[ (\bm{\xi}_i\cdot\nabla)(q_n - 2b_n)]\big]\dt + P_n[(\bu_h \cdot \nabla) b_n] \dt 
\nonumber
\\+ 
P_n[(\bm{\xi}_i\cdot\nabla) (q_n - b_n)]  \dd W^i_t= 0, \label{meCutGar}
\\
b_{0,n}:=b_n(0,x)=P_n b_0(x),\qquad q_{0,n}:=q_n(0,x)= P_n q_0(x),
\label{initialTQGCutGar}
\end{align}
where
\begin{align}
\theta_R^n:=
\theta_R(\Vert \nabla b_n \Vert_{\infty}+\Vert \nabla \bu_n \Vert_{\infty}+ \Vert q_n \Vert_{\infty})
\end{align}
and where
\begin{align}
\label{constrtGar}
\bu_n =\nabla^\perp \psi_n, 
\qquad
\bu_h =\frac{1}{2} \nabla^\perp h,
\qquad
q_n=(\Delta -1) \psi_n +f.
\end{align}
Indeed, since this is a finite  dimensional system of SDEs with locally Lipschitz  drift coefficients and globally Lipschitz diffusion coefficients, by a standard theorem, see for instant \cite{karatzas1988brownian}, we can infer the existence  of a unique local solution $(b_n,q_n)$.
The fact that the solutions are global will follow from the following a priori estimates.
%

\subsubsection{Uniform estimates}
\label{subsub:uniEst}
We now aim to show uniform bounds for $(b_n, q_n)$ in $W^{3,2}(\mathbb{T}^2) \times W^{2,2}(\mathbb{T}^2)$.
To do this, we apply $\partial^\beta $ to \eqref{ceCutGar} with $\vert \beta \vert\leq 3$ to obtain
\begin{align}
\dd \partial^\beta  b_n &+ \theta_R^n   P_n[(\bu_n \cdot \nabla \partial^\beta ) b_n] \dt  +\frac{1}{2}\partial^\beta  P_n \big[(\bm{\xi}_i\cdot\nabla)P_n[ (\bm{\xi}_i\cdot\nabla ) b_n]\big]  \dt
+  
\partial^\beta P_n[(\bm{\xi}_i\cdot\nabla  ) b_n ] \dd W^i_t 
=S_1 \dt 
\label{ceDiff1}
\end{align} 
where 
\begin{align*}
S_1&:= \theta_R^n \big[P_n( \bu_n \cdot \partial^\beta  \nabla) b_n
-
\partial^\beta P_n((\bu_n \cdot \nabla)  b_n )\big]
\end{align*}
is such that
\begin{align}
\label{estS1n}
\Vert S_1 \Vert_2  &\lesssim
\theta_R^n \big( \Vert \nabla \bu_n \Vert_{\infty}\Vert \nabla b_n \Vert_{2,2}
+
\Vert \nabla b_n \Vert_{\infty} \Vert \bu_n \Vert_{3,2}\big)
\lesssim
\theta_R^n \big(\Vert  b_n \Vert_{3,2}
+
(1+ \Vert q_n \Vert_{2,2})\big)
\end{align}
holds uniformly in $n\in \mathbb{N}$. Indeed, the constant only depends on  $\Vert f \Vert_{2,2}$ through Lemma \ref{lem:constrt0} applied to \eqref{constrtGar}.
If we now apply It\^o's formula to the mapping $t\mapsto \Vert\partial^\beta b_n(t)\Vert_2^2$ with $\vert \beta \vert\leq 3$, we obtain
\begin{equation}
\begin{aligned}
\label{b32estUniqY}
 \Vert\partial^\beta  b_n(t)\Vert_2^2 
&
= 
\Vert\partial^\beta  b_n(0)\Vert_2^2
- 
2\int_0^t\theta_R^n \int_{\mathbb{T}^2}[ (\bu_n \cdot \nabla \partial^\beta ) b_n  
] \partial^\beta  b_n \dx\ds 
\\&
+
2\int_0^t\int_{\mathbb{T}^2} S_1 \partial^\beta  b_n \dx\ds 
-  
2\int_0^t\int_{\mathbb{T}^2} \partial^\beta[(\bm{\xi}_i\cdot\nabla ) b_n ]
\partial^\beta  b_n \dx\dd W^i_\sigma
\\&
-
\int_0^t  \big\Vert P_n \partial^\beta  [ (\bm{\xi}_i\cdot\nabla ) b_n] \big\Vert_2^2 \ds
-
\int_0^t\int_{\mathbb{T}^2}  \partial^\beta   \big[(\bm{\xi}_i\cdot\nabla)P_n[ (\bm{\xi}_i\cdot\nabla ) b_n]\big]\, \partial^\beta  b_n \dx\ds
\end{aligned}
\end{equation}
$\mathbb{P}$-a.s for all $t\in[0,T]$. By squaring the resulting equation above, we obtain the following inequality
\begin{equation}
\begin{aligned}
\label{b32estUniqYx}
 \Vert\partial^\beta  b_n(t )\Vert_2^4 
&
\leq
\Vert\partial^\beta  b_n(0)\Vert_2^4
+ 
\bigg(2\int_0^{t }\theta_R^n \int_{\mathbb{T}^2}[ (\bu_n \cdot \nabla \partial^\beta ) b_n  
] \partial^\beta  b_n \dx\ds 
\bigg)^2 
\\&
+
\bigg(2\int_0^{t } \int_{\mathbb{T}^2} S_1  \partial^\beta  b_n \dx \bigg)^2\ds 
+  
\bigg(2\int_0^{t  }\int_{\mathbb{T}^2} \partial^\beta[(\bm{\xi}_i\cdot\nabla  ) b_n ]
\partial^\beta  b_n \dx\dd W^i_\sigma \bigg)^2
\\&
+
\bigg(
\int_0^t  \big\Vert P_n \partial^\beta  [ (\bm{\xi}_i\cdot\nabla ) b_n] \big\Vert_2^2 \ds
+
\int_0^t\int_{\mathbb{T}^2}  \partial^\beta   \big[(\bm{\xi}_i\cdot\nabla)P_n[ (\bm{\xi}_i\cdot\nabla ) b_n]\big]\, \partial^\beta  b_n \dx\ds
 \bigg)^2
\end{aligned}
\end{equation}
for all $t\in[0,T]$. However, we note that by integration by parts,
\begin{align*}
&2\int_0^{t }\theta_R^n \int_{\mathbb{T}^2}[ (\bu_n \cdot \nabla \partial^\beta ) b_n  
] \partial^\beta  b_n \dx\ds =0
\end{align*}
since our fluid is incompressible. Now, if we use \eqref{estS1n}, we obtain
\begin{align*}
\mathbb{E}
&\sum_{\vert \beta\vert\leq 3}
\bigg(\int_0^{t }\int_{\mathbb{T}^2} S_1  \partial^\beta  b_n \dx\ds \bigg)^2
\lesssim
\mathbb{E}
\int_0^{T }\theta_R^n 
\big(1+\Vert  b_n \Vert_{3,2}^4 + \Vert  q_n\Vert_{2,2}^4
\big)\ds.
\end{align*}
Next, by the Burkholder--Davis--Gundy inequality and \eqref{dataGarlerkin}, we obtain
\begin{align*}
\sum_{\vert \beta\vert\leq 3}
\mathbb{E}\bigg[\sup_{t\in[0,T]}
\bigg\vert
\int_0^{t } \int_{\mathbb{T}^2}\partial^\beta[(\bm{\xi}_i\cdot\nabla  ) b_n ]
\partial^\beta  b_n \dx\dd W^i_\sigma
\bigg\vert^2 \bigg]
&\lesssim
\sum_{\vert \beta\vert\leq 3}
\mathbb{E}\bigg[\int_0^{T } \bigg(\int_{\mathbb{T}^2} \partial^\beta[(\bm{\xi}_i\cdot\nabla  ) b_n ]
\partial^\beta  b_n   \dx \bigg)^2\dt 
\bigg]
\\
&\lesssim
\mathbb{E} \int_0^{T }
 \Vert  b_n \Vert_{3,2}^4
\dt .
\end{align*}
Finally, by using Lemma \ref{rem:xiAppendix}, we also have that
\begin{align*}
\mathbb{E}\sum_{\vert \beta\vert\leq 3} \bigg(
\int_0^t  \big\Vert P_n \partial^\beta  [ (\bm{\xi}_i\cdot\nabla ) b_n] \big\Vert_2^2 \ds
&+
\int_0^t\int_{\mathbb{T}^2}  \partial^\beta   \big[(\bm{\xi}_i\cdot\nabla)P_n[ (\bm{\xi}_i\cdot\nabla ) b_n]\big]\, \partial^\beta  b_n \dx\ds
 \bigg)^2
\lesssim
\mathbb{E} \int_0^{T }
 \Vert  b_n \Vert_{3,2}^4
\dt .
\end{align*}
By summing over $\vert \beta\vert \leq 3$,
we can therefore conclude from \eqref{b32estUniqYx} that
\begin{equation}
\begin{aligned}
\label{b32estUniqY1}
 \mathbb{E}\sup_{t\in[0,T]}&\Vert  b_n(t )\Vert_{3,2}^4 
\leq
 \mathbb{E}
\Vert  b_n(0)\Vert_{3,2}^4 
+c\,
\mathbb{E}
\int_0^{T }
\big(1+\Vert  b_n \Vert_{3,2}^4 + \Vert  q_n\Vert_{2,2}^4
\big)\ds
\end{aligned}
\end{equation}
where $c=c(R,\Vert f \Vert_{2,2})$. 
\\
Next, we apply $\partial^\beta$ to \eqref{meCutGar} with $\vert \beta\vert \leq 2$ to get
\begin{equation}
\begin{aligned}
\dd \partial^\beta q_n &+  \theta_R^nP_n(\bu_n \cdot \nabla \partial^\beta)( q_n -b_n)\dt
+
 P_n(\bu_h \cdot \nabla \partial^\beta) b_n\dt 
+
\frac{1}{2}\partial^\beta
P_n \big[(\bm{\xi}_i\cdot\nabla)P_n[ (\bm{\xi}_i\cdot\nabla)(q_n - 2b_n)]\big]\dt
\nonumber
\\&+
\partial^\beta P_n(\bm{\xi}_i\cdot\nabla ) (q_n - b_n) \dd W^i= (S_2 +S_3+S_4)\dt
\label{meDiff1}
\end{aligned}
\end{equation}
where
\begin{align*}
S_2&:= -\theta_R^n \big[P_n(\bu_n \cdot \partial^\beta  \nabla) b_n
-
\partial^\beta P_n(\bu_n \cdot \nabla  b_n ) \big],
\\
S_3&:= \theta_R^n \big[P_n(\bu_n \cdot \partial^\beta  \nabla) q_n
-
\partial^\beta P_n(\bu_n \cdot \nabla  q_n ) \big],
\\
S_4&:= - \big[P_n(\bu_h \cdot \partial^\beta  \nabla) b_n
-
\partial^\beta P_n(\bu_h \cdot \nabla  b_n ) \big]
\end{align*}
are such that
\begin{align}
\Vert S_2 \Vert_2 &\lesssim
\theta_R^n \big( \Vert \nabla \bu_n \Vert_{\infty}\Vert \nabla b_n \Vert_{1,2}
+
\Vert \nabla b_n \Vert_{\infty} \Vert \bu_n \Vert_{2,2}\big)
\lesssim
\theta_R^n \big( \Vert  b_n \Vert_{3,2}
+
 (1+\Vert q_n \Vert_{2,2})\big),
\label{estS2}
\\
\Vert S_3 \Vert_2 &\lesssim
\theta_R^n \big( \Vert \nabla \bu_n \Vert_{\infty}\Vert \nabla q_n \Vert_{1,2}
+
\Vert \nabla q_n \Vert_{4} \Vert \nabla\bu_n \Vert_{1,4}\big)
\nonumber
\\&
\lesssim
\theta_R^n \big( \Vert \nabla \bu_n \Vert_{\infty}\Vert \nabla q_n \Vert_{1,2}
+
\Vert  q_n \Vert_{\infty}^{1/2}\Vert \nabla q_n \Vert_{1,2}^{1/2} \Vert \nabla\bu_n \Vert_{\infty}^{1/2} \Vert\nabla \bu_n \Vert_{2,2}^{1/2}\big)
\nonumber
\\&
\lesssim
\theta_R^n \big( \Vert \nabla \bu_n \Vert_{\infty}\Vert \nabla q_n \Vert_{1,2}
+
\Vert \nabla \bu_n \Vert_{\infty}\Vert \nabla q_n \Vert_{1,2}
+
\Vert  q_n \Vert_{\infty} \Vert \bu_n \Vert_{3,2}\big)
\nonumber
\\&
\lesssim
\theta_R^n \big( 1+\Vert  q_n \Vert_{2,2}
\big),
\label{estS3}
\\
\Vert S_4 \Vert_2 &\lesssim
\big( \Vert \nabla \bu_h  \Vert_{\infty}\Vert \nabla b_n \Vert_{1,2}
+
\Vert \nabla b_n \Vert_{\infty} \Vert \bu_h \Vert_{2,2}\big)
\lesssim
\big(  \Vert  b_n \Vert_{3,2}
+
  1\big).
\label{estS4}
\end{align}
The constants only depend on $R$, $\Vert \bu_h \Vert_{3,2}$ and $\Vert f \Vert_{2,2}$ and in particular, are uniform in $n\in \mathbb{N}$. 
Similar to \eqref{b32estUniqYx}, if we apply It\^o's formula to the mapping $t\mapsto \Vert\partial^\beta q_n(t)\Vert_2^2$ and square the resulting equation, we obtain
\begin{equation}
\begin{aligned}
\label{q32estUniqYx}
 \Vert&\partial^\beta  q_n(t  )\Vert_2^4 
\leq
\Vert\partial^\beta  q_n(0)\Vert_2^4
+ 
\bigg(2\int_0^{t }\theta_R^n \int_{\mathbb{T}^2}[ (\bu_n \cdot \nabla \partial^\beta )q_n  
] \partial^\beta  q_n \dx\ds \bigg)^2 
\\&+ 
\bigg(2\int_0^{t }\theta_R^n \int_{\mathbb{T}^2}[ (\bu_n \cdot \nabla \partial^\beta ) b_n  
] \partial^\beta  q_n \dx\ds \bigg)^2 
+ 
\bigg(2\int_0^{t } \int_{\mathbb{T}^2}[ (\bu_h \cdot \nabla \partial^\beta )b_n  
] \partial^\beta q_n \dx\ds \bigg)^2 
\\&+
\bigg(2\int_0^{t }\int_{\mathbb{T}^2}(S_2 +S_3+S_4) \partial^\beta  q_n \dx\ds \bigg)^2
+  
\bigg(2\int_0^{t }\int_{\mathbb{T}^2}\partial^\beta [(\bm{\xi}_i\cdot\nabla  ) (q_n-b_n) ]
\partial^\beta  q_n \dx\dd W^i_\sigma \bigg)^2
\\&
+
\bigg( \int_0^t \int_{\mathbb{T}^2}\partial^\beta
\big[(\bm{\xi}_i\cdot\nabla)P_n[ (\bm{\xi}_i\cdot\nabla)(q_n - 2b_n)]\big] \partial^\beta  q_n \dx\ds  +
\int_0^t \big\Vert P_n \partial^\beta  [ (\bm{\xi}_i\cdot\nabla)(q_n - b_n)]\big\Vert_2^2 \ds   \bigg)^2.
\end{aligned}
\end{equation}
By integration by parts,
\begin{align*}
&2\int_0^{t }\theta_R^n \int_{\mathbb{T}^2}[ (\bu_n \cdot \nabla \partial^\beta ) q_n  
] \partial^\beta  q_n \dx\ds =0
\end{align*}
since our fluid is incompressible.
Next, by applying Lemma \ref{lem:constrt0} to \eqref{constrtGar}, we have that
\begin{align*}
\mathbb{E}
&\sum_{\vert \beta\vert\leq 2}
\bigg[
\bigg(\int_0^{t }\theta_R^n \int_{\mathbb{T}^2}[ (\bu_n \cdot \nabla \partial^\beta ) b_n  
] \partial^\beta  q_n \dx\ds \bigg)^2
+
\bigg(\int_0^{t } \int_{\mathbb{T}^2}[ (\bu_h \cdot \nabla \partial^\beta ) b_n  
] \partial^\beta  q_n \dx\ds \bigg)^2
\bigg]
\\&\lesssim
\mathbb{E}
 \int_0^{t }
\theta_R^n \big( 1+\Vert b_n \Vert_{3,2}^4 + \Vert q_n \Vert_{2,2}^4 \big)\ds 
\end{align*}
holds with a constant depending only on $\Vert \bu_h \Vert_{3,2}$ and $\Vert f \Vert_{2,2}$, and from \eqref{estS2}--\eqref{estS4}
\begin{equation}
\begin{aligned}
\mathbb{E}
&\sum_{\vert \beta\vert\leq 2}
\bigg(\int_0^{t }\int_{\mathbb{T}^2}(S_2 +S_3+S_4) \partial^\beta  q_n \dx\ds 
\bigg)^2
\lesssim
\mathbb{E}
 \int_0^{t }
\theta_R^n \big( 1+\Vert b_n \Vert_{3,2}^4 + \Vert q_n \Vert_{2,2}^4 \big)\ds.
\end{aligned}
\end{equation}
Also, by the Burkholder--Davis--Gundy inequality and \eqref{dataGarlerkin}, we have that
\begin{equation}
\begin{aligned}
\sum_{\vert \beta\vert\leq 2}
&\mathbb{E}\bigg[\sup_{t\in[0,T]}
\bigg\vert
2\int_0^{t }\int_{\mathbb{T}^2}\partial^\beta [(\bm{\xi}_i\cdot\nabla  ) (q_n-b_n) ]
\partial^\beta  q_n \dx\dd W^i_\sigma 
\bigg\vert \bigg]^2
\\
&\lesssim
\sum_{\vert \beta\vert\leq 2}
\mathbb{E}\bigg[\int_0^{T } \bigg(\int_{\mathbb{T}^2} \partial^\beta [(\bm{\xi}_i\cdot\nabla  ) q_n ]
\partial^\beta  q_n \dx \bigg)^2\dt 
\bigg]
\\&
+
\mathbb{E}\bigg[\int_0^{T } \bigg(\int_{\mathbb{T}^2} \partial^\beta [(\bm{\xi}_i\cdot\nabla  ) b_n ]
\partial^\beta  q_n \dx \bigg)^2\dt 
\bigg]
\\
&\lesssim
\mathbb{E} \int_0^{T } \big( \Vert  q_n \Vert_{2,2}^4 +\Vert  b_n \Vert_{3,2}^4 \big)
\dt 
\end{aligned}
\end{equation}
holds uniformly in $n \in \mathbb{N}$.
Finally, we also obtain from Lemma \ref{rem:xiAppendix} and the continuity of the projections,
\begin{equation}
\begin{aligned}
\mathbb{E}
\sum_{\vert \beta\vert\leq 2}
\bigg( \int_0^t \int_{\mathbb{T}^2}&\partial^\beta
\big[(\bm{\xi}_i\cdot\nabla)P_n[ (\bm{\xi}_i\cdot\nabla)(q_n - 2b_n)]\big] \partial^\beta  q_n \dx\ds  
+
\int_0^t \big\Vert P_n \partial^\beta  [ (\bm{\xi}_i\cdot\nabla)(q_n - b_n)]\big\Vert_2^2 \ds   \bigg)^2
\\&\lesssim
\mathbb{E} \int_0^{T } \big( \Vert  q_n \Vert_{2,2}^4 +\Vert  b_n \Vert_{3,2}^4 \big)
\dt.
\end{aligned}
\end{equation}
By summing over $\vert \beta\vert \leq 2$,
we can therefore conclude from \eqref{q32estUniqYx} that
\begin{equation}
\begin{aligned}
\label{q32estUniqY1}
 \mathbb{E}\sup_{t\in[0,T]}&\Vert  q_n(t )\Vert_{2,2}^4 
\leq
 \mathbb{E}\Vert  q_n(0 )\Vert_{2,2}^4
 +
 c\,
\mathbb{E}
 \int_0^{t }
 \big( 1+\Vert b_n \Vert_{3,2}^4 + \Vert q_n \Vert_{2,2}^4 \big)\ds 
\end{aligned}
\end{equation}
where $c=c(R,\Vert \bu_h \Vert_{3,2},\Vert f \Vert_{2,2})$. 
We can now sum up \eqref{b32estUniqY1} and \eqref{q32estUniqY1} and obtain
\begin{equation}
\begin{aligned}
 \mathbb{E}\sup_{t\in[0,T]}&\big(1+\Vert  b_n(t )\Vert_{3,2}^4+\Vert  q_n(t )\Vert_{2,2}^4\big) 
\leq
 \mathbb{E}
\Vert  b_n(0)\Vert_{3,2}^4 
+
 \mathbb{E}
\Vert  q_n(0)\Vert_{2,2}^4 
\\&
+c\,
\mathbb{E}
 \int_0^{T }
(\theta_R^n)^2 \big(1+ \Vert b_n(\sigma) \Vert_{3,2}^4 + \Vert q_n(\sigma) \Vert_{2,2}^4 \big)\ds
\end{aligned}
\end{equation}
where $c=c(R,\Vert \bu_h \Vert_{3,2},\Vert f \Vert_{2,2})$. 
By  Gr\"onwall’s lemma,
\begin{equation}
\begin{aligned}
 \mathbb{E}\sup_{t\in[0,T]}\big(\Vert  b_n(t )\Vert_{3,2}^4+\Vert  q_n(t )\Vert_{2,2}^4\big) 
&\leq
c\,\Big(
 \mathbb{E}
\Vert  b_n(0)\Vert_{3,2}^4 
+
 \mathbb{E}
\Vert  q_n(0)\Vert_{2,2}^4 
\Big)
\leq
c\,\Big(
 \mathbb{E}
\Vert  b_0\Vert_{3,2}^4 
+
 \mathbb{E}
\Vert  q_{0}\Vert_{2,2}^4 
\Big).
\end{aligned}
\end{equation}
From \eqref{dataGarlerkin}, it follows that
\begin{equation}
\begin{aligned}
\label{unifGarlen}
\sup_{n\in\mathbb{N}}
 \mathbb{E}\sup_{t\in[0,T]}&\big(\Vert  b_n(t)\Vert_{3,2}^4+\Vert  q_n(t )\Vert_{2,2}^4\big) 
\lesssim1.
\end{aligned}
\end{equation}

\subsubsection{H\"older continuity in time}
We can now proceed to show H\"older continuity in time of $\{(b_n,\omega_n)\}_{n\in \mathbb{N}}$. It will follow from the next lemma in conjunction with Kolmogorov's continuity result.
\begin{lemma}
\label{lem:fractTime}
For $((\bm{\xi}_i)_{i\in \mathbb{N}}, \bu_h, f, b_0, q_0)$ satisfying \eqref{dataGarlerkin}, we let  $\{(b_n,q_n)\}_{n\in \mathbb{N}}$ be the corresponding family of Galerkin approximations solving \eqref{ceCutGar}--\eqref{constrtGar}.  Then    for all $s,t\in[0,T]$, $s< t$, the estimate
\begin{align}
\label{unifEst1}
\mathbb{E}
 \Vert b_n(t) - b_n(s)\Vert_{2}^4+\mathbb{E}
 \Vert q_n(t) - q_n(s)\Vert_{2}^4 \lesssim_{R,T} \vert t- s \vert^2
\end{align}
holds uniformly in $n\in \mathbb{N}$.
\end{lemma}
\begin{proof}
In the following, we will only show that
\begin{align}\mathbb{E}
 \Vert q_n(t) - q_n(s)\Vert_{2}^4 \lesssim_{R,T} \vert t- s \vert^2
\end{align}
since the estimate for the buoyancy difference is similar and in fact easier.
\\
From \eqref{meCutGar}, we observe that for any $s,t\in[0,T]$, $s< t$
\begin{align*}
\mathbb{E}
 \Vert q_n(t) - q_n(s)\Vert_{2}^{4}
 &\leq
 \mathbb{E}
 \Big\Vert
\int_s^t
\theta_R^n P_n(\bu_n \cdot \nabla)( q_n -b_n) \dd \sigma
\Big\Vert_2^4
+ \mathbb{E}
 \Big\Vert
\int_s^t
 P_n\frac{1}{2}[(\bm{\xi}_i\cdot\nabla)(\bm{\xi}_i\cdot\nabla) ](q_n - b_n) \dd \sigma
\Big\Vert_2^4
\\&
+ \mathbb{E}
 \Big\Vert
\int_s^t
 P_n(\bu_h \cdot \nabla)b_n \dd \sigma
\Big\Vert_2^4
+
 \mathbb{E}
 \Big\Vert
 \int_s^t
 P_n(\bm{\xi}_i\cdot\nabla) (q_n - b_n)  \dd W^i_\sigma
\Big\Vert_2^4
\\&
=:I_1+I_2+I_3+I_4
\end{align*}
holds $\mathbb{P}$-a.s. where by the continuity property of $P_n$,
\begin{equation}
 \begin{aligned}
I_1
 &\lesssim
 \vert t-s \vert^4\,
\mathbb{E} 
\sup_{\sigma\in [0,T]}
\theta_R^n
\Vert
(\bu_n(\sigma) \cdot \nabla)( q_n(\sigma) -b_n(\sigma))
\Vert_2^4
\\&\lesssim T^2
\vert t-s \vert^2\,
\mathbb{E} 
\sup_{\sigma\in [0,T]}
\theta_R^n
\Vert
\bu_n(\sigma) 
\Vert_\infty^4
\Vert \nabla( q_n(\sigma) -b_n(\sigma))
\Vert_2^4
\\&
\lesssim_{R,T}
 \vert t-s \vert^2\,
\mathbb{E}\sup_{\sigma\in[0,T]}\big(\Vert  b_n(\sigma)\Vert_{3,2}^4+\Vert  q_n(\sigma )\Vert_{2,2}^4\big) 
\\&
\lesssim_{R,T}
 \vert t-s \vert^2
\end{aligned}
\end{equation} 
where the last estimate follow from \eqref{unifGarlen}. The same estimate holds for $I_2$ and $I_3$.
For the stochastic term, we use Burkholder--Davis--Gundy inequality and the continuity of $P_n$ to obtain
\begin{equation}
\begin{aligned}
I_4
&\lesssim
\mathbb{E}
\bigg(
 \int_s^t
 \Vert 
(\bm{\xi}_i\cdot\nabla) (q_n - b_n)\Vert_2^{2}\dd\sigma
\bigg)^2
\\&
\lesssim_R
\vert t-s\vert^2\,
\mathbb{E}
 \sup_{\sigma \in[0,T]}
 \Vert 
(\bm{\xi}_i\cdot\nabla) (q_n - b_n)\Vert_2^4
\\&
\lesssim_{R}
 \vert t-s \vert^2\,
\mathbb{E}\sup_{\sigma\in[0,T]}\big(\Vert  b_n(\sigma)\Vert_{3,2}^4+\Vert  q_n(\sigma )\Vert_{2,2}^4\big) 
\\&
\lesssim_{R}
 \vert t-s \vert^2.
\end{aligned}
\end{equation}
Summing the two estimates above yields our desired result. 
\end{proof}
We observe that due to Kolmogorov's continuity theorem, an immediate consequence of Lemma \ref{lem:fractTime} is the following.
\begin{corollary}
\label{coe:holderCont}
For  $((\bm{\xi}_i)_{i\in \mathbb{N}}, \bu_h, f, b_0, q_0)$ satisfying \eqref{dataGarlerkin}, we let  $\{(b_n,q_n)\}_{n\in \mathbb{N}}$ be the corresponding family of Galerkin approximations solving \eqref{ceCutGar}--\eqref{constrtGar}.   Then there exists a representation of $(b_n,q_n)$  with $\mathbb{P}$-a.s. H\"older continuous trajectories
of exponent $\gamma\in(0, 1/4)$. Furthermore, the estimate
\begin{align}
\label{unifEstHolder}
\mathbb{E}
 \Vert b_n\Vert_{C^\gamma([0,T];L^2(\mathbb{T}^2))}^4
 +
 \mathbb{E}
 \Vert q_n\Vert_{C^\gamma([0,T];L^2(\mathbb{T}^2))}^4
  \lesssim 1
\end{align}
holds uniformly in $n\in \mathbb{N}$.
\end{corollary}

\subsubsection{Compactness}
Our goal now is to define a family of measures associated with $\{(b_n,q_n ,\{W^i\}_{i\in \mathbb{N}})\}_{n\in \mathbb{N}}$ and then show that these measures are tight in suitable spaces.
In order to achieve our goal, we first define the path space
\begin{align*}
\chi=C([0,T];W^{s,2}(\mathbb{T}^2)) \times C([0,T];W^{s-1,2}(\mathbb{T}^2))
\times W^{3,2}(\mathbb{T}^2)
\times W^{2,2}(\mathbb{T}^2)
\times C([0,T])^\mathbb{N}.
\end{align*}
where $s\in(2,3)$ is close enough to $3$. Here, $S^\mathbb{N}$ denotes the $\mathbb{N}$th product space of $S$. Also note that  the embeddings $W^{s,2}(\mathbb{T}^2) \hookrightarrow W^{1,\infty}(\mathbb{T}^2)$ and $W^{s-1,2}(\mathbb{T}^2) \hookrightarrow L^{\infty}(\mathbb{T}^2)$ are continuous.
We then define the following probability measures:
\begin{align*}
&\mathcal{L}[{b_n}] \text{ is the law of  } b_n \text{  on } C([0,T];W^{s,2}(\mathbb{T}^2)),
\\  
&\mathcal{L}[{q_n}] \text{ is the law of } q_n \text{  on }C([0,T];W^{s-1,2}(\mathbb{T}^2)),
\\
&\mathcal{L}[b_{0,n}] \text{ is the law of  } b_{0,n} \text{  on } W^{3,2}(\mathbb{T}^2),
\\  
&\mathcal{L}[q_{0,n}] \text{ is the law of } q_{0,n} \text{  on } W^{2,2}(\mathbb{T}^2),
\\
&\mathcal{L}[W^i] \text{  is the law of } W^i \text{  on }C([0,T]),
\end{align*}
for each $i\in \mathbb{N}$. 
Now, by combining \eqref{unifGarlen} and Corollary \ref{coe:holderCont} with the abstract Arzel\'a--Ascoli theorem \cite[Theorem 1.1.1.]{Hof}, we can obtain the following result.
\begin{lemma}
The family of laws $\{ \mathcal{L}[{b_n}] :n \in \mathbb{N}\}$  is tight on  $C([0,T];W^{s,2}(\mathbb{T}^2))$.
 The family of laws $\{ \mathcal{L}[{q_n}] :n \in \mathbb{N}\}$  is tight on  $C([0,T];W^{s-1,2}(\mathbb{T}^2))$.
 The family of laws $\{ \mathcal{L}[b_{0,n}] :n \in \mathbb{N}\}$  is tight on  $W^{3,2}(\mathbb{T}^2)$.
 The family of laws $\{ \mathcal{L}[q_{0,n}] :n \in \mathbb{N}\}$  is tight on  $W^{2,2}(\mathbb{T}^2)$.
\end{lemma}
\begin{proof}
Due to the compact embedding
\begin{align*}
C([0,T];W^{3,2}(\mathbb{T}^2)) \cap C^\gamma([0,T];L^2(\mathbb{T}^2))
\hookrightarrow
C([0,T];W^{s,2}(\mathbb{T}^2)),
\end{align*}
which holds for $\gamma>0$ and $s<3$, see \cite[Proposition 5.2.5.]{Hof},  it follows that for any fixed $L>0$, the set
\begin{align*}
B_L:=\big\{b \in C([0,T];W^{3,2}(\mathbb{T}^2)) \cap C^\gamma([0,T];L^2(\mathbb{T}^2))
\,:\, \Vert b\Vert_{C([0,T];W^{3,2}(\mathbb{T}^2)) \cap C^\gamma([0,T];L^2(\mathbb{T}^2))}
\leq L \big\}
\end{align*}
is compact in $C([0,T];W^{s,2}(\mathbb{T}^2))$. Moreover, by  Chebyshev’s inequality,
\begin{align*}
\mathcal{L}[b_n](B_L^c) &=\mathbb{P} \big( \Vert b_n\Vert_{C([0,T];W^{3,2}(\mathbb{T}^2)) \cap C^\gamma([0,T];L^2(\mathbb{T}^2))}> L \big)
\\
&\leq\mathbb{P} \Big( \Vert b_n\Vert_{C([0,T];W^{3,2}(\mathbb{T}^2)) }> L/2 \big)
+
\mathbb{P} \big( \Vert b_n\Vert_{C^\gamma([0,T];L^2(\mathbb{T}^2))}> L/2 \Big)
\\&\leq \frac{16}{L^4}\mathbb{E}
\Big(
\Vert  b_n(t)\Vert_{C([0,T];W^{3,2}(\mathbb{T}^2))}^4
+
\Vert  b_n(t)\Vert_{C^\gamma([0,T];L^2(\mathbb{T}^2))}^4
\Big)
\end{align*}
so that  by \eqref{unifGarlen} and  Corollary \ref{coe:holderCont} ,
\begin{align*}
\sup_{n\in \mathbb{N}}\mathcal{L}[b_n](B_L^c)
\lesssim \frac{1}{L} \rightarrow 0
\end{align*}
as $L\rightarrow\infty$. This completes the proof of the first part of the lemma. The second part can be done analogously.
For the third part, we use the continuity property of the projection $P_n$ to get that $b_{0,n}:=P_nb_0 \rightarrow b_0$ a.s. This implies convergence in law which further implies tightness due to Prokhorov's theorem. The fourth part follows the same argument.
\end{proof}
We also note that since for each $i\in \mathbb{N}$, $\{ \mathcal{L}[W^i] \}$ is a singleton, it is trivially 
weakly compact and hence tight on $C([0,T])$. 
An immediate consequence of this remark and the lemma above is that:
\begin{corollary}
\label{cor:tightnessSmooth}
The family of measures $\{ \mathcal{L}[{b_n,q_n,b_{0,n}, q_{0,n},\{W^i\}_{i\in \mathbb{N}}}] :n \in \mathbb{N}\}$  is tight on  $\chi$.
\end{corollary}


With tightness of the measures in hand, we can now proceed to identify the limits of the sequence of functions by way of the Jakubowski--Skorokhod representation theorem \cite{jakubowski1998short}. See also, \cite[Theorem 2.7.1]{breit2018stoch}.
\begin{proposition}
\label{prop:jaku}
There exists a complete probability space $(\tilde{\Omega},\tilde{\mathcal{F}},\tilde{\mathbb{P}})$ with $\chi$-valued
Borel measurable random variables $\{(\tilde{b}_n,\tilde{q}_n ,\tilde{b}_{0,n},\tilde{q}_{0,n} ,\{\tilde{W}^i_n\}_{i\in \mathbb{N}})\}_{n\in \mathbb{N}}$ and $(\tilde{b},\tilde{q},\tilde{b}_0,\tilde{q}_0,\{\tilde{W}^i\}_{i\in \mathbb{N}})$ such that (up to a subsequence, not relabelled):
\begin{itemize}
\item the law of $\{(\tilde{b}_n,\tilde{q}_n ,\tilde{b}_{0,n},\tilde{q}_{0,n} ,\{\tilde{W}^i_n\}_{i\in \mathbb{N}})\}_{n\in \mathbb{N}}$ is $ \mathcal{L}[{b_n,q_n,b_{0,n}, q_{0,n},\{W^i\}_{i\in \mathbb{N}}}], \,n \in \mathbb{N}$;
\item the law of $(\tilde{b},\tilde{q},\tilde{b}_0, \tilde{q}_0,\{\tilde{W}^i\}_{i\in \mathbb{N}})$ is a Radon measure;
\item the following convergences
\begin{align}
&\tilde{b}_n \rightarrow \tilde{b}\qquad\text{in}\qquad C([0,T];W^{s,2}(\mathbb{T}^2))
\\
 &\tilde{q}_n \rightarrow \tilde{q}\qquad\text{in}\qquad C([0,T];W^{s-1,2}(\mathbb{T}^2))
\\
&\tilde{b}_{0,n} \rightarrow \tilde{b}_0\qquad\text{in}\qquad W^{3,2}(\mathbb{T}^2),
\\
 &\tilde{q}_{0,n} \rightarrow \tilde{q}_0\qquad\text{in}\qquad W^{2,2}(\mathbb{T}^2)
\\
&\tilde{W}^i_n  \rightarrow  \tilde{W}^i \qquad\text{in}\qquad C([0,T])
\end{align}
hold $\tilde{\mathbb{P}}$-a.s. as $n\rightarrow\infty$.
\end{itemize}
\end{proposition}
\subsubsection{Existence of a strong martingale solution for the system with cut-off}
\label{sec:strongMartWithCut}
Due to the equivalence of laws, as stated in Proposition \ref{prop:jaku}, as well as \cite[Theorem 2.9.1]{breit2018stoch}, we can conclude that on the new probability space $(\tilde{\Omega},\tilde{\mathcal{F}},\tilde{\mathbb{P}})$, the collection $(\tilde{b}_n, \tilde{q}_n,\tilde{b}_{0,n},\tilde{q}_{0,n} , \{\tilde{W}_n^i\}_{i\in \mathbb{N}})$ satisfies the $n$th order Galerkin approximation \eqref{ceCutGar}--\eqref{constrtGar}  where $\{\tilde{W}_n^i\}_{i\in \mathbb{N}}$ is a family of independent Brownian motions. Furthermore, we can endow this new space with the canonical filtration 
\begin{align*}
\tilde{\mathcal{F}}^n_t=\sigma \Big(\sigma_t[\tilde{b}_n] \cup \sigma_t[\tilde{q}_n] \cup_{i\in \mathbb{N}} \sigma_t[\tilde{W}^i_n] \Big), \quad t\in [0,T].
\end{align*}
Similarly, we can define the $\sigma$-algebra $(\tilde{\mathcal{F}}_t)_{t\geq 0}$ on $(\tilde{b},\tilde{q},\{\tilde{W}^i\}_{i\in \mathbb{N}})$ by using \cite[Lemma 2.9.3]{breit2018stoch} and Proposition \ref{prop:jaku}.
Now, because $(\tilde{b}_n,\tilde{q}_n)$ solves \eqref{ceCutGar}--\eqref{constrtGar}, it satisfies \eqref{unifGarlen}, and thus by Fatou's lemma, its limit $(\tilde{b},\tilde{q})$ satisfies the bound
\begin{equation}
\begin{aligned}
\label{beddtildeLimi}
 \mathbb{E}\sup_{t\in[0,T]}&\big(\Vert  \tilde{b}(t)\Vert_{3,2}^4+\Vert  \tilde{q}(t )\Vert_{2,2}^4\big) 
\lesssim1.
\end{aligned}
\end{equation}
We can however improve this to strong continuity in time for $(\tilde{b},\tilde{q})$ by mimicking the argument in \cite[Page 1665]{kim2011on}. Indeed, we observe that $\tilde{b}_n$ satisfies \eqref{b32estUniqY} since again, $(\tilde{b}_n,\tilde{q}_n)$ solves \eqref{ceCutGar}--\eqref{constrtGar}. By summing the equation over the multiindex $\vert \beta \vert\leq 3$, we can conclude that
\begin{equation}
\begin{aligned}
\label{itoequaTilde}
 \Vert  \tilde{b}_n(t)\Vert_{3,2}^2 
&
= 
\Vert  \tilde{b}_n(0)\Vert_{3,2}^2
- 
2\int_0^t \big\langle\theta_R^n  (\tilde{\bu}_n \cdot \nabla  ) \tilde{b}_n  
\,,\,   \tilde{b}_n \big\rangle_{3,2}\ds 
-  
2\int_0^t\big\langle(\bm{\xi}_i\cdot\nabla  ) \tilde{b}_n \,,\,
 \tilde{b}_n \big\rangle_{3,2}\dd \tilde{W}^i_{n,\sigma}
\\&
-
\int_0^t  \big\Vert P_n  (\bm{\xi}_i\cdot\nabla ) \tilde{b}_n \big\Vert_{3,2}^2 \ds
-
\int_0^t   \big\langle (\bm{\xi}_i\cdot\nabla)P_n (\bm{\xi}_i\cdot\nabla ) \tilde{b}_n\,,\,  \tilde{b}_n \big\rangle_{3,2}\ds
\end{aligned}
\end{equation}
holds $\mathbb{P}$-a.s for all $t\in[0,T]$. Since we have the a.s. converges result from Proposition \ref{prop:jaku} and the limit satisfies \eqref{beddtildeLimi}, we can pass to the limit as $n\rightarrow\infty$ in \eqref{itoequaTilde} to obtain
\begin{equation}
\begin{aligned}
\label{itoequaTildeLimit}
 \Vert  \tilde{b}(t)\Vert_{3,2}^2 
&
= 
\Vert  \tilde{b}(0)\Vert_{3,2}^2
- 
2\int_0^t \big\langle\theta_R  (\tilde{\bu} \cdot \nabla  ) \tilde{b} 
\,,\,   \tilde{b}\big\rangle_{3,2}\ds 
-  
2\int_0^t\big\langle(\bm{\xi}_i\cdot\nabla  ) \tilde{b} \,,\,
 \tilde{b} \big\rangle_{3,2}\dd \tilde{W}^i_{\sigma}
\\&
-
\int_0^t  \big\Vert (\bm{\xi}_i\cdot\nabla ) \tilde{b} \big\Vert_{3,2}^2 \ds
-
\int_0^t   \big\langle (\bm{\xi}_i\cdot\nabla) (\bm{\xi}_i\cdot\nabla ) \tilde{b}\,,\,  \tilde{b} \big\rangle_{3,2}\ds
\end{aligned}
\end{equation}
$\mathbb{P}$-a.s for all $t\in[0,T]$. Note  in particular that the convergence
$\theta_R^n \rightarrow \theta_R$ follows from the convergence result in Proposition \ref{prop:jaku} and the embeddings $W^{s,2}(\mathbb{T}^2)\hookrightarrow W^{1,\infty}(\mathbb{T}^2)$ and $W^{s-1,2}(\mathbb{T}^2)\hookrightarrow L^{\infty}(\mathbb{T}^2)$ which hold for $s>2$.
Moving on, we note that it follows from \eqref{itoequaTildeLimit} that $\Vert \tilde{b}(t) \Vert_{3,2}$ is continuous in $t$ since the right-hand side is continuous in time. Combining this with \eqref{beddtildeLimi} yields
$
\tilde{b} \in L^2 \big(\tilde{\Omega};C([0,T]; W^{3,2}(\mathbb{T}^2)) \big).
$
The same argument applies to $\tilde{q}$ and we obtain
$
\tilde{q} \in L^2 \big(\tilde{\Omega};C([0,T]; W^{2,2}(\mathbb{T}^2)) \big).
$
Furthermore, $(\tilde{b},\tilde{q})$ are $(\tilde{\mathcal{F}}_t)$-progressively measurable processes since they both have the continuity and measurability properties \cite[Proposition 2.1.18]{breit2018stoch}. Finally, using once more Proposition \ref{prop:jaku} and the properties of the projections, we are able to pass to the limit in \eqref{ceCutGar}--\eqref{constrtGar} (which is satisfied by $(\tilde{b}_n, \tilde{q}_n, \{\tilde{W}_n^i\}_{i\in \mathbb{N}})$)  to obtain our next result. 
In particular,   the passage to the limit in the stochastic integral term makes use of \cite[Lemma 2.6.6]{breit2018stoch} (see also, \cite[Lemma 2.1]{debussche2011local}).   



\begin{proposition}[Strong martingale solution]
\label{prop:smoothMartin}
Fix  $R>0$. Let $((\bm{\xi}_i)_{i\in \mathbb{N}}, \bu_h, f, b_0, q_0)$ satisfy \eqref{dataGarlerkin}. Then there exists a stochastic basis  $(\tilde{\Omega}, \tilde{\mathcal{F}},(\tilde{\mathcal{F}}_t)_{t\geq0}, \tilde{\mathbb{P}})$, a family of $(\tilde{\mathcal{F}}_t)$-Brownian motions $\{\tilde{W}^i\}_{i\in \mathbb{N}}$ and  $(\tilde{\mathcal{F}}_t)$-progressively measurable processes
\begin{align*}
&\tilde{b} \in L^2 \big(\tilde{\Omega};C([0,T]; W^{3,2}(\mathbb{T}^2)) \big),
\qquad
&\tilde{q} \in  L^2 \big(\tilde{\Omega};C([0,T]; W^{2,2}(\mathbb{T}^2)) \big)
\end{align*}
with law $\mathcal{L}[\tilde{b}_0,\tilde{q}_0] = \mathcal{L}[b_0, q_0]$ such that $(\tilde{b}, \tilde{q})$ satisfies
\begin{align}
\tilde{b}(t)=\tilde{b}_0 - \int_0^t \theta_R[(\tilde{\bu} \cdot \nabla) \tilde{b} ]\ds -\int_0^t (\bm{\xi}_i\cdot\nabla) \tilde{b} \circ \dd \tilde{W}^i_\sigma,
\\
\tilde{q}(t) = \tilde{q}_0 - \int_0^t\theta_R[(\tilde{\bu}\cdot \nabla)( \tilde{q} - \tilde{b})]\ds - \int_0^t(\bu_h \cdot \nabla) \tilde{b} \ds  - \int_0^t(\bm{\xi}_i\cdot\nabla) (\tilde{q}- \tilde{b}) \circ \dd \tilde{W}^i_\sigma
\end{align}
$\tilde{\mathbb{P}}$-a.s. for all $t\in[0,T]$.
\end{proposition}
This result establishes the existence of a strong martingale solution for the system with cut-off.  We can now proceed to show pathwise uniqueness for this cut-off system so that we are able to apply a Yamada--Watanabe type argument to establish the existence of an improved strong pathwise solution for the system with cut-off. 
\subsubsection{Pathwise uniqueness for the system with cut-off}
\label{sec:pathUniqueCutOff1}
In the following, we let $(b_1,q_2)$ and $(b_2,q_2)$ be two solutions of \eqref{ceCut}--\eqref{constrtCut} with the same data class \eqref{dataGarlerkin}.
We now set 
  $b_{12}:= b_1
  -b_2$, $q _{12}:= q _1 -q _2$, $\bu_{12}:= \bu_1-\bu_2$  and $\theta_R^{12}:=\theta_R^1 -\theta_R^2$ so that  $(b_{12},  q _{12}, \theta_R^{12})$ satisfies
\begin{equation}
\begin{aligned}
\label{b12CuttOff}
\dd b_{12} &+ \Big(\theta_R^{12} [(\bu_1 \cdot \nabla) b_1 ]
+\theta_R^2 [(\bu_{12} \cdot \nabla) b_1 ]
+\theta_R^2 [(\bu_2 \cdot \nabla) b_{12} ]  \Big)\dt
\\
&+\frac{1}{2}(\bm{\xi}_i\cdot\nabla)[ (\bm{\xi}_i\cdot\nabla) b_{12}] \dt 
+
(\bm{\xi}_i\cdot\nabla) b_{12}\dd W^i =0,
\end{aligned}
\end{equation}
and $(b_{12}, \bu_{12}, q _{12}, P_{12})$ satisfies
\begin{equation}
\begin{aligned}
\label{q12CuttOff}
\dd q_{12} &+ \Big(\theta_R^{12} [(\bu_1 \cdot \nabla)(q_1 - b_1) ]
+\theta_R^2 [(\bu_{12} \cdot \nabla)(q_1 - b_1) ]
+\theta_R^2 [(\bu_2 \cdot \nabla)(q_{12} - b_{12}) ]+
(\bu_h \cdot \nabla) b_{12}  \Big)\dt
\\
&
+\frac{1}{2}
(\bm{\xi}_i\cdot\nabla)[ (\bm{\xi}_i\cdot\nabla)(q_{12} - 2b_{12})] \dt 
+ 
(\bm{\xi}_i\cdot\nabla)(q_{12} - b_{12})\dd W^i =0.
\end{aligned}
\end{equation}
If we apply $\partial^\beta $ to the equation \eqref{b12CuttOff} for $b_{12}$ above where now $\vert \beta \vert\leq 2$, we obtain
\begin{equation}
\begin{aligned}
\dd \partial^\beta b_{12} &+ \Big(\theta_R^{12} [(\bu_1 \cdot \nabla \partial^\beta) b_1 ]
+\theta_R^2 [(\bu_{12} \cdot \nabla \partial^\beta) b_1 ]
+\theta_R^2 [(\bu_2 \cdot \nabla \partial^\beta) b_{12} ]  \Big)\dt
\\
&+\frac{1}{2} \partial^\beta  \big[(\bm{\xi}_i\cdot\nabla)[ (\bm{\xi}_i\cdot\nabla) b_{12}]\big] \dt 
+  \partial^\beta [(\bm{\xi}_i\cdot\nabla) b_{12} ] \dd W^i =(S_1 +  S_2+ S_3)\dt,
\label{ceDiff5}
\end{aligned}
\end{equation}
where  
\begin{align*}
S_1&:= \theta_R^{12}\big([(\bu_1 \cdot \partial^\beta  \nabla) b_1]
-
\partial^\beta [(\bu_1 \cdot \nabla)  b_1 ]\big)
,
\\
S_2&:= \theta_R^2\big([(\bu_{12} \cdot \partial^\beta  \nabla )b_1]
-
\partial^\beta[(\bu_{12} \cdot \nabla)  b_1 ]\big)
,
\\
S_3&:=\theta_R^2\big([(\bu_2 \cdot \partial^\beta  \nabla) b_{12}]
-
\partial^\beta[(\bu_2 \cdot \nabla ) b_{12} ]\big).
\end{align*}
For $s\in(2,3)$, we can now use the commutator estimate and the fact that 
\begin{equation}
\begin{aligned}
\label{estCutOff}
|\theta_R^{12}| &\lesssim_R \Vert b_{12}\Vert_{s,2}+\Vert \bu_{12}\Vert_{s,2}+ \Vert q_{12}\Vert_{s-1,2}
\\
&\lesssim_R \Vert b_{12}\Vert_{2,2}^{3-s}\Vert b_{12}\Vert_{3,2}^{s-2}+ \Vert q_{12}\Vert_{1,2}^{3-s}\Vert q_{12}\Vert_{2,2}^{s-2}
\\
&\lesssim_R \Vert b_{12}\Vert_{2,2}(1+\Vert b_1\Vert_{3,2}+\Vert b_2\Vert_{3,2})+ \Vert q_{12}\Vert_{1,2}(1+\Vert q_1\Vert_{2,2}+\Vert q_2\Vert_{2,2})
\\
&\lesssim_R( \Vert b_{12}\Vert_{2,2}+ \Vert q_{12}\Vert_{1,2})(1+\Vert b_1\Vert_{3,2}+\Vert b_2\Vert_{3,2}+\Vert q_1\Vert_{2,2}+\Vert q_2\Vert_{2,2})
\end{aligned}
\end{equation}
to obtain the estimates
\begin{align}
\Vert S_1 \Vert_2 &
\lesssim_R 
( \Vert b_{12}\Vert_{2,2}+ \Vert q_{12}\Vert_{1,2})(1+\Vert b_1\Vert_{3,2}^3+\Vert b_2\Vert_{3,2}^3+\Vert q_1\Vert_{2,2}^3+\Vert q_2\Vert_{2,2}^3),
\label{s1Estdiff}
\\
\Vert S_2 \Vert_2 
&
\lesssim_R  
\Vert  q_{12} \Vert_{1,2}\Vert b_1 \Vert_{3,2},
\label{s2Estdiff}
\\
\Vert S_3 \Vert_2 
&
\lesssim_R 
(1+ \Vert  q_2 \Vert_{2,2})\Vert b_{12} \Vert_{2,2}.
\label{s3Estdiff}
\end{align}
If we apply It\^o's formula to the mapping $t\mapsto \dd \Vert\partial^\beta b_{12}(t)\Vert_2^2$ with $\vert \beta \vert\leq 2$, we obtain
\begin{equation}
\begin{aligned}
\label{b32estCutOff}
 \dd\Vert \partial^\beta b_{12}(t) \Vert_2^2  &= - 2
  \int_{\mathbb{T}^2} \Big(\theta_R^{12} [(\bu_1 \cdot \nabla \partial^\beta) b_1 ]
+\theta_R^2 [(\bu_{12} \cdot \nabla \partial^\beta) b_1 ]
+\theta_R^2 [(\bu_2 \cdot \nabla \partial^\beta) b_{12} ]  \Big)\partial^\beta b_{12} \dx \dt
\\
&- \int_{\mathbb{T}^2} \partial^\beta  \big[(\bm{\xi}_i\cdot\nabla)[ (\bm{\xi}_i\cdot\nabla) b_{12}]\big] \partial^\beta b_{12} \dx\dt 
- \Vert \partial^\beta [(\bm{\xi}_i\cdot\nabla) b_{12}] \Vert_2^2\dt
\\&
- 2 \int_{\mathbb{T}^2} \partial^\beta [(\bm{\xi}_i\cdot\nabla) b_{12} ] \partial^\beta b_{12} \dx \dd W^i_t 
+
2
 \int_{\mathbb{T}^2}
(S_1 +  S_2+ S_3)\partial^\beta b_{12} \dx\dt.
\end{aligned}
\end{equation} 
We can now use \eqref{estCutOff} to obtain
\begin{equation}
\begin{aligned}
\bigg\vert  2
 \int_{\mathbb{T}^2} \theta_R^{12} [(\bu_1 \cdot \nabla \partial^\beta) b_1 ]
\partial^\beta b_{12} \dx \bigg\vert
\leq G_1(t)
\big(\Vert b_{12} \Vert_{2,2}^2+ \Vert q_{12} \Vert_{1,2}^2 \big)
\end{aligned}
\end{equation}
where
\begin{align}
G_1(t):= c(R)\bigg( 1
+\sum_{i=1}^2 \Vert b_i(t) \Vert_{3,2}^3 +\sum_{i=1}^2\Vert q_i(t) \Vert_{2,2}^3 \bigg).
\end{align}
Also,
\begin{equation}
\begin{aligned}
\bigg\vert  2
  \int_{\mathbb{T}^2} \theta_R^2 [(\bu_{12} \cdot \nabla \partial^\beta) b_1 ]
\partial^\beta b_{12} \dx \bigg\vert
\leq G_2(t)
\big( \Vert b_{12} \Vert_{2,2}^2+ \Vert q_{12} \Vert_{1,2}^2 \big)
\end{aligned}
\end{equation}
where
\begin{align}
G_2(t):= c(R) 
\Vert b_1(t) \Vert_{3,2}.
\end{align}
On the other hand, because $\bu_2$ is divergence free, by integration by parts,
\begin{equation}
\begin{aligned}
\bigg\vert  2
 \int_{\mathbb{T}^2} \theta_R^2 [(\bu_2 \cdot \nabla \partial^\beta) b_{12} ]
\partial^\beta b_{12} \dx \bigg\vert
=0.
\end{aligned}
\end{equation}
Next, from \eqref{s1Estdiff}--\eqref{s3Estdiff}, we have that
\begin{equation}
\begin{aligned}
2\bigg\vert  \int_{\mathbb{T}^2}
(S_1 +  S_2+ S_3)\partial^\beta b_{12} \dx \bigg\vert
&\leq G_3(t)
\big(\Vert b_{12} \Vert_{2,2}^2+ \Vert q_{12} \Vert_{1,2}^2 \big)
\end{aligned}
\end{equation}
where
\begin{align}
G_3(t):= c(R)\bigg( 1
+\sum_{i=1}^2 \Vert b_i(t) \Vert_{3,2} +\sum_{i=1}^2\Vert q_i(t) \Vert_{2,2} \bigg).
\end{align}
Finally, using Lemma \ref{rem:xiAppendix} we obtain the following estimate
\begin{equation}
\begin{aligned}
\bigg\vert \int_{\mathbb{T}^2}&\partial^\beta
\big[(\bm{\xi}_i\cdot\nabla)[ (\bm{\xi}_i\cdot\nabla)b_{12}]\big] \partial^\beta b_{12} \dx
+
 \big\Vert  \partial^\beta  [ (\bm{\xi}_i\cdot\nabla)b_{12}]\big\Vert_2^2 
 \bigg\vert
\lesssim
\Vert  b_{12} \Vert_{2,2}^2.
\end{aligned}
\end{equation}
In summary, by summing all terms in \eqref{b32estCutOff} over $\vert \beta \vert\leq 2$ we have shown that
\begin{equation}
\begin{aligned}
\label{b32estUniqSumm}
 \dd\Vert  b_{12}(t) \Vert_{2,2}^2  
 &\leq
 G_4(t)
 \big(\Vert b_{12} \Vert_{2,2}^2+ \Vert q_{12} \Vert_{1,2}^2 \big)\dt
- 2\sum_{\vert \beta \vert\leq 2} \int_{\mathbb{T}^2} \partial^\beta [(\bm{\xi}_i\cdot\nabla) b_{12} ] \partial^\beta b_{12} \dx \dd W^i_t 
\end{aligned}
\end{equation}
where
\begin{align}
\label{G4}
G_4(t)\sim G_1(t).
\end{align}
\\
Next, we return to \eqref{q12CuttOff} and apply $\partial^\beta $ with $\vert \beta \vert\leq 1$ to it to obtain
\begin{equation}
\begin{aligned}
\dd \partial^\beta q_{12} &+ \Big(\theta_R^{12} [(\bu_1 \cdot \nabla \partial^\beta)(q_1 - b_1) ]
+\theta_R^2 [(\bu_{12} \cdot \nabla \partial^\beta)(q_1 - b_1) ]
+\theta_R^2 [(\bu_2 \cdot \nabla \partial^\beta)(q_{12} - b_{12}) ]
\\
&+
(\bu_h \cdot \nabla \partial^\beta) b_{12}  \Big)\dt
+\frac{1}{2}
\partial^\beta\big[
(\bm{\xi}_i\cdot\nabla)[ (\bm{\xi}_i\cdot\nabla)(q_{12} - 2b_{12})]\big] \dt 
\\&+ 
\partial^\beta\big[(\bm{\xi}_i\cdot\nabla)(q_{12} - b_{12})\big]\dd W^i 
=(S_4+ \ldots+S_{10})\dt
\end{aligned}
\end{equation}
where 
\begin{align*}
S_4&:= \theta_R^{12}\big([(\bu_1 \cdot \partial^\beta  \nabla) q_1]
-
\partial^\beta [(\bu_1 \cdot \nabla)  q_1 ]\big)
,
&S_5:= -\theta_R^{12}\big([(\bu_1 \cdot \partial^\beta  \nabla) b_1]
-
\partial^\beta [(\bu_1 \cdot \nabla)  b_1 ]\big)
,
\\
S_6&:= \theta_R^2\big([(\bu_{12} \cdot \partial^\beta  \nabla )q_1]
-
\partial^\beta[(\bu_{12} \cdot \nabla)  q_1 ]\big)
,
&S_7:= - \theta_R^2\big([(\bu_{12} \cdot \partial^\beta  \nabla )b_1]
-
\partial^\beta[(\bu_{12} \cdot \nabla)  b_1 ]\big)
,
\\
S_8&:=\theta_R^2\big([(\bu_2 \cdot \partial^\beta  \nabla) q_{12}]
-
\partial^\beta[(\bu_2 \cdot \nabla ) q_{12} ]\big)
,
&S_9:=-\theta_R^2\big([(\bu_2 \cdot \partial^\beta  \nabla) b_{12}]
-
\partial^\beta[(\bu_2 \cdot \nabla ) b_{12} ]\big)
,
\\
S_{10}&:=[(\bu_h \cdot \partial^\beta  \nabla) b_{12}]
-
\partial^\beta[(\bu_h \cdot \nabla ) b_{12} ].
\end{align*}
Similar to the estimates for $S_1$, $S_2$ and $S_3$ above, we can show that
\begin{align}
\label{sis}
\Vert S_i \Vert_2 \leq G_5(t) \big( \Vert b_{12} \Vert_{2,2} + \Vert q_{12} \Vert_{1,2}\big),
\qquad i=4, \ldots, 10
\end{align}
where
\begin{align}
\label{G5}
G_5(t):= c(R)\bigg( 1+ \Vert \bu_h \Vert_{3,2}+\Vert f \Vert_{2,2}
+\sum_{i=1}^2 \Vert b_i(t) \Vert_{3,2}^3 +\sum_{i=1}^2\Vert q_i(t) \Vert_{2,2}^3 \bigg).
\end{align}
Therefore,
with $\vert \beta \vert\leq 1$, we have that
\begin{equation}
\begin{aligned}
2\bigg\vert  \int_{\mathbb{T}^2}
(S_4 + \ldots+ S_{10})\partial^\beta q_{12} \dx \bigg\vert
&\leq G_5(t)
\big(\Vert b_{12} \Vert_{2,2}^2+ \Vert q_{12} \Vert_{1,2}^2 \big).
\end{aligned}
\end{equation}
If we apply It\^o's formula to the mapping $t\mapsto \dd \Vert\partial^\beta q_{12}(t)\Vert_2^2$ with $\vert \beta \vert\leq 1$, we obtain
\begin{equation}
\begin{aligned}
\label{q32estCutOff}
 \dd\Vert \partial^\beta q_{12}(t) \Vert_2^2  &= - 2
  \int_{\mathbb{T}^2} \Big(\theta_R^{12} [(\bu_1 \cdot \nabla \partial^\beta) (q_1-b_1) ]
+\theta_R^2 [(\bu_{12} \cdot \nabla \partial^\beta) (q_1-b_1) ]
\\&
+\theta_R^2 [(\bu_2 \cdot \nabla \partial^\beta) (q_{12} - b_{12}) ] 
+(\bu_h\cdot \nabla \partial^\beta)b_{12} \Big)\partial^\beta q_{12} \dx \dt
\\
&- \int_{\mathbb{T}^2} \partial^\beta  \big[(\bm{\xi}_i\cdot\nabla)[ (\bm{\xi}_i\cdot\nabla) (q_{12} -2 b_{12})]\big] \partial^\beta q_{12} \dx\dt 
- \Vert \partial^\beta [(\bm{\xi}_i\cdot\nabla) (q_{12} - b_{12})] \Vert_2^2\dt
\\&
- 2 \int_{\mathbb{T}^2} \partial^\beta [(\bm{\xi}_i\cdot\nabla)(q_{12} - b_{12}) ] \partial^\beta q_{12} \dx \dd W^i_t 
+
2
 \int_{\mathbb{T}^2}
(S_4 +\ldots+ S_{10})\partial^\beta q_{12} \dx\dt.
\end{aligned}
\end{equation} 
If we now combine \eqref{sis} with an argument similar to how \eqref{b32estUniqSumm} was derived, we deduce by summing  all terms in \eqref{q32estCutOff} over $\vert \beta \vert\leq 1$  that
\begin{equation}
\begin{aligned}
\label{q32estUniqSumm}
 \dd\Vert  q_{12}(t) \Vert_{1,2}^2  
 &\leq G_5(t)
 \big(\Vert b_{12} \Vert_{2,2}^2+ \Vert q_{12} \Vert_{1,2}^2 \big)\dt
- 2\sum_{\vert \beta \vert\leq 1} \int_{\mathbb{T}^2} \partial^\beta [(\bm{\xi}_i\cdot\nabla) (q_{12}- b_{12}) ] \partial^\beta q_{12} \dx \dd W^i_t.
\end{aligned}
\end{equation}
Further summing up \eqref{b32estUniqSumm} and \eqref{q32estUniqSumm} yields
\begin{equation}
\begin{aligned}
\label{bandq32estUniqSumm}
 \dd\big( \Vert  b_{12}(t) \Vert_{2,2}^2
 +\Vert  q_{12}(t) \Vert_{1,2}^2 \big)  
 &\leq G(t)
 \big(\Vert b_{12} \Vert_{2,2}^2+ \Vert q_{12} \Vert_{1,2}^2 \big) \dt
 \\&- 2\sum_{\vert \beta \vert\leq 2} \int_{\mathbb{T}^2} \partial^\beta [(\bm{\xi}_i\cdot\nabla) b_{12} ] \partial^\beta b_{12} \dx \dd W^i_t
\\&- 2\sum_{\vert \beta \vert\leq 1} \int_{\mathbb{T}^2} \partial^\beta [(\bm{\xi}_i\cdot\nabla) (q_{12}- b_{12}) ] \partial^\beta q_{12} \dx \dd W^i_t 
\end{aligned}
\end{equation}
where $G(t) \sim G_5(t)$. Next, we make use of
the It\^o's product rule
\begin{equation}
\begin{aligned}
\label{bandq32estUniqSummProd}
 \dd\Big[e^{-\int_0^t G(\sigma)\dd \sigma}&\big( \Vert  b_{12}(t) \Vert_{2,2}^2
 +\Vert  q_{12}(t) \Vert_{1,2}^2 \big)  
 \Big]
\\&=
-G(t)e^{-\int_0^t G(\sigma)\dd \sigma}\big(  \Vert  b_{12}(t) \Vert_{2,2}^2
 +\Vert  q_{12}(t) \Vert_{1,2}^2 \big)\dt  
 \\&
+e^{-\int_0^t G(\sigma)\dd \sigma} \dd\big(  \Vert  b_{12}(t) \Vert_{2,2}^2
 +\Vert  q_{12}(t) \Vert_{1,2}^2 \big).
 \end{aligned}
\end{equation}
With \eqref{bandq32estUniqSumm} and \eqref{bandq32estUniqSummProd}  in hand, we obtain
\begin{equation}
\begin{aligned}
\label{bandq32estUniqSummProdx}
e^{-\int_0^t G(\sigma)\dd \sigma}&\big( \Vert  b_{12}(t) \Vert_{2,2}^2
 +\Vert  q_{12}(t) \Vert_{1,2}^2 \big)
 \leq
 \big(\Vert b_{12}(0) \Vert_{2,2}^2+ \Vert q_{12}(0) \Vert_{1,2}^2 \big) 
 \\
 &
 - 2 \int_0^te^{-\int_0^s G(\sigma)\dd \sigma} \sum_{\vert \beta \vert\leq 2}\int_{\mathbb{T}^2} \partial^\beta [(\bm{\xi}_i\cdot\nabla) b_{12} ] \partial^\beta b_{12} \dx \dd W^i_s
\\&- 2\int_0^te^{-\int_0^s G(\sigma)\dd \sigma}  \sum_{\vert \beta \vert\leq 1}\int_{\mathbb{T}^2} \partial^\beta [(\bm{\xi}_i\cdot\nabla) (q_{12}- b_{12}) ] \partial^\beta q_{12} \dx \dd W^i_s .
\end{aligned}
\end{equation}
Since the paths of $b_i$ and $q_i$, $i=1,2$ are a.s. continuous in $W^{3,2}(\mathbb{T}^2)$ and $W^{2,2}(\mathbb{T}^2)$ respectively, it follows that $e^{-\int_0^s G(\sigma)\dd \sigma}$ is a.s strictly positive. Recall that $G\sim G_5$ where the latter is given by \eqref{G5}. Furthermore, the expectation of the It\^o stochastic integrals in \eqref{bandq32estUniqSummProdx} are zero leading to
\begin{equation}
\begin{aligned}
\label{bandq32estUniqSummProdxY}
\mathbb{E}\Big[e^{-\int_0^t G(\sigma)\dd \sigma}&\big( \Vert  b_{12}(t) \Vert_{2,2}^2
 +\Vert  q_{12}(t) \Vert_{1,2}^2 \big) \Big]
 \leq
 \mathbb{E}
 \big(\Vert b_{12}(0) \Vert_{2,2}^2+ \Vert q_{12}(0) \Vert_{1,2}^2 \big) 
.
\end{aligned}
\end{equation}
Combining the strict positivity of $e^{-\int_0^s G(\sigma)\dd \sigma}$ with the fact that $b_1(0)=b_2(0)$  and $q_1(0)=q_2(0)$, we  conclude that
\begin{align}
\mathbb{E}\big( \Vert  b_{12}(t) \Vert_{2,2}^2
 +\Vert  q_{12}(t) \Vert_{1,2}^2 \big)=0.
\end{align}

\subsubsection{Existence of a strong pathwise solution for the system with cut-off}
\label{sec:pathUniqueCutOff2}
Having shown the existence of a strong martingale solution and pathwise uniqueness, we are now in the position to establish the existence of a strong pathwise solution  of \eqref{ceCut}--\eqref{constrtCut}. This follows from the application of the Gy\"ongy--Krylov
characterization of convergence in probability which requires the analysis of two families of approximate solutions. In this regard, we return to the Galerkin system \eqref{ceCutGar}--\eqref{constrtGar} and consider two of its solutions $(b_m, q_m)$ and $(b_n,q_n)$ with respective initial conditions \eqref{initialTQGCutGar}. If we let $\mathcal{L}[b_m,q_m,b_{0,m},q_{0,m},b_n,q_n,b_{0,n},q_{0,n},\{W^i\}_{i\in \mathbb{N}}]$ be the  law  on the extended space
\begin{align*}
\chi^2:=\big[C([0,T];W^{s,2}(\mathbb{T}^2)) \times C([0,T];W^{s-1,2}(\mathbb{T}^2)) \times W^{3,2}(\mathbb{T}^2) \times W^{2,2}(\mathbb{T}^2)\big]^2 \times C([0,T])^\mathbb{N}
\end{align*}
where for a space $S$, $S^p$ denotes the product space $S\times\ldots\times S$ ($p$-times), then we can show that
\begin{center}
The family of measures $\{ \mathcal{L}[{b_m,q_m,b_{0,m},q_{0,m},b_n,q_n,b_{0,n},q_{0,n},\{W^i\}_{i\in \mathbb{N}}}] :n \in \mathbb{N}\}$  is tight on  $\chi^2$.
\end{center}
%
%
The tightness argument is a direct adaptation of Corollary \ref{cor:tightnessSmooth}.
Let us now fix an arbitrary subsequence $\{ \mathcal{L}[{b_{m_k},q_{m_k},b_{0,m_k},q_{0,m_k},b_{k_n},q_{k_n}, b_{0,n_k},q_{0,n_k},\{W^i_k\}_{i\in \mathbb{N}}}] :k \in \mathbb{N}\}$ of $\{ \mathcal{L}[{b_m,q_m,b_{0,m},q_{0,m},b_n,q_n,b_{0,n},q_{0,n},\{W^i\}_{i\in \mathbb{N}}}] :n \in \mathbb{N}\}$, which will also be tight on $\chi^2$.  Then by passing to a further subsequence (not relabelled), just like in Proposition \ref{prop:jaku}, we obtain the following result by way of Jakubowski--Skorokhod representation theorem:
\begin{proposition}
\label{prop:jakuKrylov}
There exists a complete probability space $(\overline{\Omega},\overline{\mathcal{F}},\overline{\mathbb{P}})$ with $\chi^2$-valued
Borel measurable random variables $\{(\hat{b}_{m_k},\hat{q}_{m_k},\hat{b}_{0,m_k},\hat{q}_{0,m_k}, \check{b}_{n_k}, \check{q}_{n_k} , \check{b}_{0,n_k}, \check{q}_{0,n_k} ,\{\overline{W}^i_k\}_{i\in \mathbb{N}})\}_{k\in \mathbb{N}}$ and $(\hat{b},\hat{q},\hat{b}_0,\hat{q}_0, \check{b}, \check{q}, \check{b}_0, \check{q}_0, \{\overline{W}^i\}_{i\in \mathbb{N}})$ such that (up to a subsequence, not relabelled):
\begin{itemize}
\item the law of $\{(\hat{b}_{m_k},\hat{q}_{m_k},\hat{b}_{0,m_k},\hat{q}_{0,m_k}, \check{b}_{n_k}, \check{q}_{n_k} , \check{b}_{0,n_k}, \check{q}_{0,n_k} ,\{\overline{W}^i_k\}_{i\in \mathbb{N}})\}_{k\in \mathbb{N}}$ is 
\\
$ \mathcal{L}[{b_{m_k},q_{m_k},b_{0,m_k},q_{0,m_k}, b_{n_k},q_{n_k},b_{0,n_k},q_{0,n_k},\{W^i_k\}_{i\in \mathbb{N}}}], \,k \in \mathbb{N}$;
\item the law of $(\hat{b},\hat{q},\hat{b}_0,\hat{q}_0, \check{b}, \check{q},\check{b}_0, \check{q}_0,\{\overline{W}^i\}_{i\in \mathbb{N}})$ is a Radon measure;
\item the following convergence
\begin{align}
\label{jakuKrylov}
&(\hat{b}_{m_k},\hat{q}_{m_k},\hat{b}_{0,m_k},\hat{q}_{0,m_k}, \check{b}_{n_k}, \check{q}_{n_k} , \check{b}_{0,n_k}, \check{q}_{0,n_k} ,\{\overline{W}^i_k\}_{i\in \mathbb{N}}) \rightarrow (\hat{b},\hat{q},\hat{b}_0,\hat{q}_0, \check{b}, \check{q}, \check{b}_0, \check{q}_0, \{\overline{W}^i\}_{i\in \mathbb{N}}) \quad\text{in}\quad \chi^2
\end{align}
holds $\tilde{\mathbb{P}}$-a.s. as $k\rightarrow\infty$.
\end{itemize}
\end{proposition}
\noindent Having obtained Proposition \ref{prop:jakuKrylov}, we now apply the arguments in Section \ref{sec:strongMartWithCut} separately to
\begin{align*}
\hat{b}_{m_k},\hat{q}_{m_k},\hat{b}_{0,m_k},\hat{q}_{0,m_k},\{\overline{W}^i_k\}_{i\in \mathbb{N}},
\hat{b},\hat{q},\hat{b}_0,\hat{q}_0,  \{\overline{W}^i\}_{i\in \mathbb{N}}
\end{align*}
and 
\begin{align*}
 \check{b}_{n_k}, \check{q}_{n_k} ,\check{b}_{0,n_k}, \check{q}_{0,n_k} ,\{\overline{W}^i_k\}_{i\in \mathbb{N}}), \check{b}, \check{q}, \check{b}_0, \check{q}_0, \{\overline{W}^i\}_{i\in \mathbb{N}}.
\end{align*}
Just as in Proposition \ref{prop:smoothMartin}, it will follow that, relative to the same Browian motions $\{\overline{W}^i\}_{i\in \mathbb{N}}$ and the same stochastic basis
$(\overline{\Omega},\overline{\mathcal{F}},(\overline{\mathcal{F}})_{t\geq0},\overline{\mathbb{P}})$ where
\begin{align*}
\overline{\mathcal{F}}_t=\sigma \Big(\sigma_t[\hat{b}] \cup \sigma_t[\hat{q}] \cup
\sigma_t[\check{b}] \cup \sigma_t[\check{q}] \cup_{i\in \mathbb{N}} \sigma_t[\overline{W}^i] \Big), \quad t\in [0,T],
\end{align*}
$
(\hat{b},\hat{q})\text{ and }( \check{b}, \check{q})
$
are strong martingale solutions of  \eqref{ceCut}--\eqref{constrtCut} with corresponding initial conditions $
(\hat{b}_0,\hat{q}_0)\text{ and }( \check{b}_0, \check{q}_0)
$. Furthermore,
\begin{align*}
\overline{\mathbb{P}}\big(\omega\in \overline{\Omega}\, :\, [\hat{b}_0, \hat{q}_0] = [\check{b}_0, \check{q}_0]\big)=1.  
\end{align*}
Indeed, since $b_{0, m_k}= P_{m_k}b_0$ and $b_{0, n_k}= P_{n_k}b_0$,
for any $l\leq m_k \wedge n_k$,
\begin{align*}
\overline{\mathbb{P}}\big(\omega\in \overline{\Omega}\, :\, [P_l\hat{b}_{0,m_k}, P_l\hat{q}_{0,m_k}] = [P_l\check{b}_{0,n_k}, P_l\check{q}_{0,n_k}]\big)
=
\mathbb{P}\big(\omega\in \Omega\, :\, [P_lb_{0,m_k}, P_lq_{0,m_k}] = [P_lb_{0,n_k}, P_lq_{0,n_k}]\big)=1  
\end{align*}
by equality of the laws and thus the claim follows. 
\\Now denote by $\mu_{m_k,n_k}$ and $\mu$, the joint laws of $(\hat{b}_{m_k}, \hat{q}_{m_k}, \check{b}_{m_k}, \check{q}_{m_k})$ and $(\hat{b}, \hat{q}, \check{b}, \check{q})$ respectively. Then due to \eqref{jakuKrylov},  $\mu_{m_k,n_k} \rightharpoonup \mu$ as $k\rightarrow\infty$. Since $(\hat{b}_0, \hat{q}_0) = (\check{b}_0, \check{q}_0)$ holds $\overline{\mathbb{P}}$-a.s., by pathwise uniqueness, see Section \ref{sec:pathUniqueCutOff1}, $(\hat{b}, \hat{q}) = (\check{b}, \check{q})$ holds $\overline{\mathbb{P}}$-a.s. and as such, for $X:=C([0,T];W^{s,2}(\mathbb{T}^2)) \times C([0,T];W^{s-1,2}(\mathbb{T}^2))$,
\begin{align*}
    \mu \big((x,y)\in X \times X \,:\, x=y  \big)
=\overline{\mathbb{P}}\big(\omega\in \overline{\Omega}\,:\,  [\hat{b},\hat{q}] = [\check{b}, \check{q}] \big) =1.
\end{align*}
We now have all in hand to apply Gy\"ongy--Krylov's characterization of convergence \cite{gyongy1996existence} or its generalization to quasi-Polish spaces \cite[ Theorem 2.10.3]{breit2018stoch}. 
It implies that  the original sequence $(b_n,q_n)$ defined on the original probability space $(\Omega,\mathcal{F}, \mathbb{P})$ converges in probability to some random variables
$(b,q)$ in the topology of $C([0,T]; W^{s,2}(\mathbb{T}^2)) \times C([0,T]; W^{s-1,2}(\mathbb{T}^2)) $. By taking a subsequence if need be, recall \eqref{jakuKrylov}, we obtain this convergence almost surely. We can then repeat the identification-of-the-limit process for the original sequence on the original probability space (just as was done for the new processes on the new space in 
Section \ref{sec:strongMartWithCut}) to finally deduce that, indeed, the `original' limit is the unique strong pathwise
solution of \eqref{ceCut}--\eqref{constrtCut}.  
In summary, we have shown the following result on the existence of a strong pathwise solution of  \eqref{ceCut}--\eqref{constrtCut}.
\begin{proposition}
\label{prop:smoothPathwise}
Fix $R>0$. Let $((\bm{\xi}_i)_{i\in \mathbb{N}}, \bu_h, f, b_0, q_0)$ satisfy \eqref{dataGarlerkin}.
Let $(\Omega, \mathcal{F},(\mathcal{F}_t)_{t\geq0}, \mathbb{P})$ be a given stochastic
basis with a complete right-continuous filtration and let $\{W^i\}_{i\in \mathbb{N}}$ be a family of independent $(\mathcal{F}_t)$-Brownian motions. Then there exists $(\mathcal{F}_t)$-progressively measurable processes
\begin{align*}
&b \in L^2 \big( \Omega;C([0,T]; W^{3,2}(\mathbb{T}^2)) \big),
\qquad
&q \in L^2 \big( \Omega;C([0,T]; W^{2,2}(\mathbb{T}^2)) \big)
\end{align*}
such that $(b, q)$ uniquely solves \begin{align}
 {b}(t)= {b}(0) - \int_0^t \theta_R[( {\bu} \cdot \nabla)  {b} ]\ds -\int_0^t (\bm{\xi}_i\cdot\nabla)  {b} \circ \dd  {W}^i_\sigma,
\\
 {q}(t) =  {q}(0) - \int_0^t\theta_R[( {\bu}\cdot \nabla)(  {q} -  {b})]\ds - \int_0^t(\bu_h \cdot \nabla)  {b} \ds  - \int_0^t(\bm{\xi}_i\cdot\nabla) ( {q}-  {b}) \circ \dd  {W}^i_\sigma
\end{align}
$ {\mathbb{P}}$-a.s. for all $t\in[0,T]$.
\end{proposition}
To emphasize that this solution solves a system with cut-offs, we will denote this limit solution by $(b_R,q_R)$. Our next goal is to remove the cut-off to obtain a unique maximal strong pathwise solution of the original problem \eqref{ce}--\eqref{constrt}.

\subsection{Maximal strong pathwise solutions of the original system}
\label{sec:pathSolOrig}
In the previous section, we have shown the existence of a unique strong pathwise solution of the approximate system \eqref{ceCut}--\eqref{constrtCut} with cut-offs. Our goal in this section is to use the solution constructed for this approximate system to obtain a unique strong pathwise solution of our original system \eqref{ce}--\eqref{constrt} followed by the maximality and uniqueness of the stopping time. We shall now start off with the construction of the local solutions for our original system.

\subsubsection{Existence of a local strong pathwise solution for the original system}
In the following, we define the stopping time $\tau_R$ as
\begin{align*}
    \tau_R:=  \inf\big\{t \in [0,T] \,:\,\Vert \nabla b_R(t) \Vert_{\infty}+ \Vert \nabla\bu_R(t) \Vert_{\infty}+\Vert q_R(t) \Vert_{\infty} 
    \geq 
    R  \big\}
\end{align*}
with the convention that $\inf\emptyset=\infty$.
Note that $\tau_R$ is a well-defined stopping time since the map \\  $t \mapsto \sup_{s\in[0,t]}\Vert(\nabla b_R,\nabla \bu_R,q_R)(s)\Vert_\infty $ is continuous and adapted. This follows from the fact that
  $(b_R,q_R)$ has continuous trajectories in $W^{3,2}(\mathbb{T}^2)\times W^{2,2}(\mathbb{T}^2)$ and the following continuous embeddings
\begin{align*}
\Vert \nabla b_R \Vert_{\infty}\lesssim \Vert b_R \Vert_{3,2}, \qquad
\Vert \nabla\bu_R \Vert_{\infty}+\Vert q_R \Vert_{\infty} \lesssim 1+\Vert q_R \Vert_{2,2} 
\end{align*}
hold for any given $f\in W^{2,2}(\mathbb{T}^2)$. With this stopping time in hand, an immediate consequence of Proposition \ref{prop:smoothPathwise} is the following.
\begin{corollary}
\label{cor:smoothPathwise}
Fix  $R>0$. For $((\bm{\xi}_i)_{i\in \mathbb{N}}, \bu_h, f, b_0, q_0)$ satisfying \eqref{dataGarlerkin}, there exists a local strong pathwise solution $(b_R,q_R,\tau_R)$ of \eqref{ce}--\eqref{constrt}.
\end{corollary}
Note that this corollary follows from Proposition \ref{prop:smoothPathwise} because in both cases,  the initial condition $(b_0, q_0)$ is integrable in $\omega$, see \eqref{dataGarlerkin}.
The extension to general initial data as given in the statement of our overall main result Theorem \ref{thm:main} follows the argument in \cite[Section 5.3.3]{breit2018stoch}.  The main point here is to construct  a stopping time $\tau_{K(R)}$ so that on $[0,\tau_{K(R)})$, the bound
\begin{align}
\label{onWhich}
\sup_{t\in[0,\tau_{K(R)}]}\big(\Vert \nabla b_R(t)\Vert_{\infty} + \Vert \nabla \bu_R(t)\Vert_{\infty} + \Vert q_R(t)\Vert_{\infty}\big)<R
\end{align}
holds $\mathbb{P}$-a.s. This is done by choosing $K>0$  such that $K=K(R) \rightarrow \infty$ as $R\rightarrow \infty$ and $K(R)<R\min \{1, 1/c \}$. Here, $c=c(\Vert f \Vert_{2,2})$ is the constant such that the estimate
\begin{align*}
\Vert \nabla b_R \Vert_{\infty}
+
\Vert \nabla \bu_R \Vert_{\infty}
+
\Vert q_R \Vert_{\infty}
\leq c (1+\Vert b_R \Vert_{3,2}+\Vert q_R \Vert_{2,2} )
\end{align*}
holds. We then define $\tau_{K(R)}$
as
\begin{align*}
    \tau_{K(R)}:=  \inf\big\{t \in [0,T] \,:\,1+\Vert b_R(t) \Vert_{3,2}
    +
    \Vert q_R(t) \Vert_{2,2}
    \geq {K(R)}  \big\}
    \end{align*}
on which we obtain \eqref{onWhich}.
\\
With this information in hand, we can now consider the data $((\bm{\xi}_i)_{i\in \mathbb{N}}, \bu_h, f, b_0, q_0)$ satisfying \eqref{dataMain} rather than \eqref{dataGarlerkin}. In particular, the initial conditions $(b_0, q_0)$ are no longer assumed to be integrable with respect to the probability space. We then define the set
\begin{align*}
U_{K(R)}=\big\{(b,q)\in W^{3,2}(\mathbb{T}^2) \times W^{2,2}(\mathbb{T}^2)\,:\, \Vert b \Vert_{3,2} + \Vert q \Vert_{2,2}<K \big\}
\end{align*}
on which
\begin{align}
\label{indicatorIC}
(b_0,q_0)\bm{1}_{(b_0,q_0)\in\{U_{K(M)}\setminus \cup_{J=1}^{M-1}U_{K(J)} \}}
\end{align} 
is now bounded in $\omega$ for any initial condition $(b_0, q_0)$ satisfying \eqref{dataMain}. Therefore, for $(b_0, q_0)$ satisfying \eqref{dataMain}, Proposition \ref{prop:smoothPathwise}  applies to
\eqref{indicatorIC} leading to the existence of a unique pathwise solution $(b_M,q_M)$ of the cut-off system \eqref{ceCut}--\eqref{constrtCut} with $R=M$. Furthermore, similar to Corollary \ref{cor:smoothPathwise}, we can remove the cut-off and obtain instead, a local strong pathwise solution $(b_M,q_M,\tau_{K(M)})$ of the original system  \eqref{ce}--\eqref{constrt} with initial condition \eqref{indicatorIC}. Moreover, by summing up these solutions, we find that
\begin{align*}
(b,q)=\sum_{M\in \mathbb{N}}(b_M,q_M)\bm{1}_{(b_0,q_0)\in\{U_{K(M)}\setminus \cup_{J=1}^{M-1}U_{K(J)} \}}
\end{align*}
also solves the same problem up to the a.s. strictly positive stopping time
\begin{align*}
\tau =\sum_{M\in \mathbb{N}}\tau_{K(M)}\bm{1}_{(b_0,q_0)\in\{U_{K(M)}\setminus \cup_{J=1}^{M-1}U_{K(J)} \}}
\end{align*}
for any initial condition $(b_0, q_0)$ satisfying \eqref{dataMain}. In summary, we have shown the following result.
\begin{lemma}
\label{lem:smoothPathwise}
For $((\bm{\xi}_i)_{i\in \mathbb{N}}, \bu_h, f, b_0, q_0)$ satisfying \eqref{dataMain}, there exists a local strong pathwise solution $(b,q,\tau)$ of \eqref{ce}--\eqref{constrt}.
\end{lemma}
Our next goal is to extend the local strong pathwise solution constructed above to a maximal existence time.
\subsubsection{Maximal strong pathwise solution for the original system}

The idea used to extend the solution constructed in the previous section to a maximal time of existence $\tau$ follows the argument in \cite[Section 5.3.4]{breit2018stoch} which rely on \cite[Chapter 5, Section 18]{doob1994measure} to justify the existence of such a maximal time. Indeed,  letting
\begin{align*}
\mathfrak{d}_R=\tau \wedge \inf\big\{t\in [0,T] \, :\, \Vert \nabla b(t)\Vert_{\infty}+ \Vert\nabla \bu(t)\Vert_{\infty}+ \Vert q(t)\Vert_{\infty} \geq R  \big\},
\end{align*}
(where $(\mathfrak{d}_R,\tau)$ corresponds to $(\tau_R, \mathfrak{t})$ respectively in \cite[Section 5.3.4]{breit2018stoch}), leads to the existence of the maximal solution $(b,q,\tau_R, \tau)$ by repeating the arguments in \cite[Section 5.3.4]{breit2018stoch} verbatim except for the obvious change in notations for the various time steps and for the solution.
\subsubsection{Uniqueness of maximal strong pathwise solutions}
Let $(b_i,q_i, \tau_R^i, \tau^i)$, $i=1,2$ be two maximal strong pathwise solutions of \eqref{ce}--\eqref{constrt} with the same general data $((\bm{\xi}_i)_{i\in \mathbb{N}}, \bu_h, f, b_0, q_0)$ satisfying \eqref{dataMain}. If we define
\begin{align*}
\Omega_K=\Big\{\omega \in \Omega\, :\, \Vert b_0(\omega) \Vert_{3,2}+ \Vert q_0(\omega) \Vert_{2,2}<K \Big\}
\end{align*}
so that $\Omega=\cup_{K\in \mathbb{R}} \Omega_K$, we can then infer from \eqref{unifGarlen} that for each $i=1,2$,
\begin{equation}
\begin{aligned}
 \mathbb{E}\bigg[\bm{1}_{\Omega_K}\sup_{t\in[0,T\wedge \tau_R^i]}&\big(\Vert  b_i(t)\Vert_{3,2}^4+\Vert  q_i(t )\Vert_{2,2}^4\big) 
 \bigg]
\lesssim1
\end{aligned}
\end{equation}
%
%
%
with a constant depending only on $R,K, \Vert \bu_h\Vert_{3,2}$ and $\Vert f\Vert_{2,2}$. Consequently, we can can apply the pathwise uniqueness result from Section \ref{sec:pathUniqueCutOff2} on $\tau_R^1\wedge\tau_R^2$ to conclude that
\begin{align*}
\mathbb{P}\big(\bm{1}_{\Omega_K} [b_1,q_1](t\wedge \tau_R^1\wedge \tau_R^2) = \bm{1}_{\Omega_K} [b_2,q_2](t\wedge \tau_R^1\wedge \tau_R^2)\ \text{ for all } t\in [0,T] \big) =1.
\end{align*}
We can now use dominated convergence to pass to the limit $R,K \rightarrow \infty$ and obtain
\begin{align*}
\mathbb{P}\big( [b_1,q_1](t\wedge \tau^1\wedge \tau^2) = [b_2,q_2](t\wedge \tau^1\wedge \tau^2)\ \text{ for all } t\in [0,T] \big) =1.
\end{align*}
The two solutions therefore coincide up to the stopping time $\tau^1 \wedge \tau^2$. However since $\tau^1$ and  $\tau^2$ are both maximal times, it  follows that $\tau^1 = \tau^2$ almost surely.
\\
This finishes the proof of Theorem \ref{thm:main}.

\begin{proof}[Proof of Theorem \ref{thm:stability}]
The proof of Theorem \ref{thm:stability} is the same as \eqref{bandq32estUniqSummProdxY} for the system with cut-offs. Here, the stopping times $(\tau_R^j)_{R\in \mathbb{N}}$, $j=1,2$ replaces the role of the cut-offs. 
\end{proof}

\section{Blowup of the Stochastic Thermal Quasigeostrophic Equations}
\label{sec:blowup}
Let $(b,q, (\tau_R)_{R\in \mathbb{N}}, \tau)$ be the unique maximal strong pathwise solution of \eqref{ce}--\eqref{constrt} as constructed in Section \ref{sec:MainWellPose} above for a given set of data. 
As a consequence of this construction, we have seen that,  recall Remark  \ref{rem:weakBlowup}, if $\tau<\infty$, then
\begin{align*}
\limsup_{t\nearrow\tau}\big( \Vert \nabla b(t)\Vert_{\infty} + \Vert \nabla \bu(t) \Vert_{\infty} + \Vert q(t)\Vert_\infty \big)=\infty \qquad \mathbb{P}\text{-a.s.}
\end{align*}
Thus, the three quantities announces the blowup of solution or numerical simulation at the finite time $\tau$. This is however too restrictive and it turns out that we do better.
\\
In the following, we want to generalize this to a blowup criteria of Beale--Kato--Majda-type \cite{beale1984remarks}.
In particular, we show the following result.

\begin{theorem}
\label{thm:bkmTQG}
Let $(b,q, (\tau_R)_{R\in \mathbb{N}}, \tau)$ be a unique maximal strong pathwise solution of \eqref{ce}--\eqref{constrt} with a dataset
$((\bm{\xi}_i)_{i\in \mathbb{N}}, \bu_h, f, b_0, q_0)$ satisfying:
\begin{equation}
\begin{aligned}
\label{dataTQGBlowup}
\bm{\xi}_i&\in W^{4,\infty}_{\mathrm{div}}(\mathbb{R}^2),
\quad \sum_{i\in \mathbb{N}} \Vert \bm{\xi}_i \Vert_{4,\infty}<\infty,\quad\bu_h \in W^{3,2}_{\mathrm{div}}(\mathbb{R}^2), \quad  f\in W^{2,2}(\mathbb{R}^2)\quad \text{and}
\\
&(b_0, q_0) \in  W^{3,2}(\mathbb{R}^2) \times W^{2,2}(\mathbb{R}^2) \text{ is a pair of } \mathcal{F}_0\text{-measurable random variables.}
\end{aligned}
\end{equation}
For $M,N \in \mathbb{N}$,
let $\tau_M^1$ and $\tau_N^2$ be
defined as
\begin{align*}
& \tau_M^1=\inf  \bigg\{ t\geq 0 \, : \,\int_0^t\big( \Vert \nabla b(s)\Vert_{\infty} +
 \Vert q(s)\Vert_\infty \big)\dd s\geq M \bigg\},
\\&
\tau_N^2 =\inf \big\{ t\geq 0 \, : \,
 \Vert b(t)\Vert_{3,2} + \Vert q(t) \Vert_{2,2}\geq N \big\},
\end{align*}
with $\tau^1 = \lim_{M\rightarrow \infty} \tau_M^1$ and $\tau^2 = \lim_{N\rightarrow \infty} \tau_N^2$.
Then
\begin{align}
\label{tau1tau2tau}
\tau^1= \tau^2 =\tau
\end{align}
$ \mathbb{P}$-a.s.
\end{theorem}
To prove Theorem \ref{thm:bkmTQG}, we require some preparations. We begin by constructing a suitable Green's function for the elliptic equation at the end of \eqref{constrt}. This in turn will help us derive a log-Sobolev estimate for the Lipschitz norm of the  velocity which is the key to solving Theorem \ref{thm:bkmTQG}.

\subsection{Log-Sobolev estimate for velocity gradient}
In the following, we let 
 $\nabla_0 :=(\partial_{x_1},\partial_{x_2}, 0)^T$ and $\nabla_0^\perp :=(-\partial_{x_2},\partial_{x_1}, 0)$ be the three-dimensional extensions of the  two-dimensional differential operators $\nabla=(\partial_{x_1},\partial_{x_2})^T$ and $\nabla^\perp :=(-\partial_{x_2},\partial_{x_1})$  by zero respectively. Our goal is to find an estimate for $\bu$ that solves
 \begin{align*}
 \bu = \nabla^\perp \psi, \qquad (\Delta-1)\psi= w, \qquad w=q-f
 \end{align*}
 where $h\in W^{2,2}(\mathbb{T}^2)$ is given. In particular, inspired by \cite{beale1984remarks}, we aim to show that
 \begin{equation}
\begin{aligned}
\label{velovortiEst}
\Vert \nabla \mathbf{u} \Vert_\infty
\lesssim
1
+
(1+2\ln^+(\Vert w \Vert_{2,2}))
\Vert w 
\Vert_{\infty}
\end{aligned}
\end{equation}
where $\ln^+ a = \ln\, a$ if $a\geq1$ and $\ln^+a=0$ otherwise.\\[2mm]
\\{
\noindent We start with $G^{per}$ defined as 
\begin{align*}
    G^{per}(x_1,x_2)=-\sum_{(k_1,k_2)\in \mathbb{Z}^2}\frac{e^{i(k_1x_1+k_2x_2)}}{k_1^2+k_2^2+1}, \qquad (x_1,x_2) \in \mathbb{T}^2.
\end{align*}
Observe that $G^{per}(0,0)=-\infty$ and that $G^{per}\in L^2({\mathbb{T}^2})$ as
\[
\|G^{per}\|_2^2=
\sum_{(k_1,k_2)\in \mathbb{Z}^2}\frac{1}{(k_1^2+k_2^2+1)^2}<\infty.
\]
One can check that 
\begin{align*}
    (\Delta-1)G^{per}(\mathbf{x})=\delta_0.
\end{align*}
More precisely
\begin{align*}
    \int_{\mathbb{T}^2} G^{per}(\mathbf{x})(\Delta -1)w(\mathbf{x})\dx =w(0)
\end{align*}
for any periodic function $w\in W^{2,2}(\mathbb{T}^2)$. In particular
\begin{align}
\label{star}
    \int_{\mathbb{T}^2} G^{per}(\mathbf{x})(\Delta -1)e^{i(l_1x_1 + l_2x_2)} \dd x_1 \dd x_2 = 
    -(l_1^2+l_2^2+1)\int_{\mathbb{T}^2} G^{per}(\mathbf{x})e^{i(l_1x_1 + l_2x_2)} \dd x_1 \dd x_2 
    =1
\end{align}
for any $(l_1,l_2) \in \mathbb{Z}^2$.
Note that \eqref{star} uniquely identifies $G^{per}(\mathbf{x})$. That is, if $\overline{G}(\mathbf{x})$ is periodic, in $L^2(\mathbb{T}^2)$ and
\begin{align}
    \int_{\mathbb{T}^2} \overline{G}(\mathbf{x})(\Delta -1)e^{i(l_1x_1 + l_2x_2)} \dd x_1 \dd x_2 = 1.
\end{align}
Then $G^{per}(\mathbf{x}) = \overline{G}(\mathbf{x})$, $\mathbf{x}\in \mathbb{T}^2$. That is because $G^{per}-\overline{G}\in L^2(\mathbb{T}^2) $ will be perpendicular on all elements of the basis of $L^2(\mathbb{T}^2)$, and therefore it must be 0.   \\[2mm]
Next let
\begin{align*}
    G^{free}(\mathbf{x}) = -\frac{1}{4\pi}\int_{\mathbb{R}}\frac{e^{-\sqrt{\vert \mathbf{x}\vert^2+ r^2}}}{\sqrt{\vert \mathbf{x}\vert^2 +r^2}} \dd r \qquad (x_1,x_2) \in \mathbb{R}^2.
\end{align*}
Observe that $G^{free}(0,0)=-\infty$ and that $G^{free}\in L^2_{loc}({\mathbb{R}^2})$ as For $\mathbf{x} \neq (0,0)$, $r=\vert \mathbf{x}\vert \overline{r}$,
\begin{align*}
    |G^{free}(\mathbf{x})| = \frac{1}{4\pi}\left|\int_{\mathbb{R}}\frac{e^{-\vert \mathbf{x}\vert\sqrt{\overline{r}^2 +1}}}{\sqrt{\overline{r}^2+1}} \dd \overline{r}\right|\lesssim \frac{1}{\sqrt{|x|}},
\end{align*}
since $\sup_{y\ge 0} \sqrt{y}e^{-y}<\infty$. One can check that 
\begin{align*}
    (\Delta-1)G^{free}(\mathbf{x})=\delta_0.
\end{align*}
More precisely, this means that 
\begin{align*}
    \int_{\mathbb{R}^2} G^{per}(\mathbf{x})(\Delta -1)w(\mathbf{x})\dx =w(0)
\end{align*}
for any smooth function $w$ with compact 
support. Note the different way in which we interpret the fundamental solution of the Helmholtz equation on the torus as opposed of that on whole space.    

Consider $\xi=1$ in a neighbourhood of $(0,0)$ and $\xi=0$ in a neighbourhood of the boundary on $\mathbb{T}^2=[-\frac{1}{2}, \frac{1}{2}]\times [-\frac{1}{2}, \frac{1}{2}]$. Say $\xi(x_1,x_2)=\eta(|x_1|) \eta(|x_2|)$ where $\eta:[0,\infty)\rightarrow [0,1]$ is a smooth decreasing function such that
\begin{align*}
\eta(x)
=
\begin{cases}
1 & : \vert x\vert\le \frac{1}{6}\\
0 & : \vert x\vert \ge\frac{1}{3}
\end{cases}.
\end{align*}
Then $\xi G^{free}$ is periodic and $\xi G^{per}\in L^2({\mathbb{T}^2})$. Moreover, we have that 
\begin{equation}\label{star2}
\begin{aligned}
(\Delta-1)(\xi(\mathbf{x})G^{free}(\mathbf{x})) = (\Delta \xi(\mathbf{x}))G^{free}(\mathbf{x}) +2\nabla \xi(\mathbf{x}) \cdot \nabla G^{free}(\mathbf{x}) + \xi(\mathbf{x})(\Delta-1) G^{free}(\mathbf{x}),
\end{aligned}
\end{equation}
where (\ref{star2}) is understood in the sense of (\ref{star}). For smooth $w$, the last term in (\ref{star2}) can be explicitly computed as
\begin{align*}
    \xi(\mathbf{x})(\Delta-1) G^{free}(\mathbf{x}) =\delta_{0}
\end{align*}
because
\begin{align*}
    \int_{\mathbb{T}^2}\xi(\mathbf{x})(\Delta-1) G^{free}(\mathbf{x})w(\mathbf{x})\dx
    &=
    \int_{\mathbb{T}^2}(\Delta-1) G^{free}(\mathbf{x})\xi(\mathbf{x})w(\mathbf{x})\dx
    \\&=
    \int_{\mathbb{R}^2}(\Delta-1) G^{free}(\mathbf{x})\overline{w}(\mathbf{x})\dx
    \\&=\overline{w}(0)= \xi(0)w(0)=w(0),
\end{align*}
where
\begin{align*}
\overline{w}(\mathbf{x})
=
\begin{cases}
\xi(\mathbf{x})w(\mathbf{x}) &  \mathbf{x} \in[-\frac{1}{2},\frac{1}{2}] \times [-\frac{1}{2},\frac{1}{2}],\\
0 &  \mathbf{x} \in \mathbb{R}^2 \setminus [-\frac{1}{2},\frac{1}{2}] \times [-\frac{1}{2},\frac{1}{2}].
\end{cases}
\end{align*}
Define next $\varphi:\mathbb{T}^2\rightarrow \mathbb{R}$. 
\begin{align*}
    \varphi(\mathbf{x}) = (\Delta \xi(\mathbf{x}))G^{free}(\mathbf{x}) +2\nabla \xi(\mathbf{x}) \cdot \nabla G^{free}(\mathbf{x}), \qquad 
    x\in\mathbb{T}^2.
\end{align*}
Then $\varphi$ is periodic, smooth, bounded and equal to zero in the neighbourhood of the boundary  $[-\frac{1}{2},\frac{1}{2}]\times [-\frac{1}{2},\frac{1}{2}]$ as well as on
$[-\frac{1}{4},\frac{1}{4}] \times [-\frac{1}{4},\frac{1}{4}]$. 
Let $\overline{\varphi}$ be the solution of
\begin{align*}
    (\Delta-1)\overline{\varphi}(\mathbf{x}) = - \varphi(\mathbf{x}),
\end{align*}
where the above equation is interpreted as a PDE on the torus. Then $\overline{\varphi}$ is also smooth periodic function on the torus. Define
\begin{align*}
    \overline{G}(\mathbf{x}) = \xi(\mathbf{x}) G^{free}(\mathbf{x}) + \overline{\varphi}(\mathbf{x}).
\end{align*}
Then, again in the sense of (\ref{star}), we have that 
\begin{align*}
    (\Delta-1)\overline{G}(\mathbf{x}) &=(\Delta-1)(\xi(\mathbf{x})G^{free}(\mathbf{x})) - (\Delta -1)\overline{\varphi}(\mathbf{x})
    \\&= \varphi(\mathbf{x}) +\delta_{0}(\mathbf{x}) -\varphi(\mathbf{x})= \delta_{0},
\end{align*}
and therefore, indeed, we have that
\begin{align}
    G^{per}(\mathbf{x})= \xi(\mathbf{x}) G^{free}(\mathbf{x}) + \overline{\varphi}(\mathbf{x}).
\end{align}
In other words $G^{per}$ is the sum of a  truncation of $G^{free}$ and a smooth function.
}\\
Combining this information with  \cite[Proposition 1]{crisan2022breakdown} yields \eqref{velovortiEst}.

\subsection{Proof of blowup}
In order to prove Theorem \ref{thm:bkmTQG}, 
we first need some preliminary estimates for $(b,q)$ subject to a dataset
$((\bm{\xi}_i)_{i\in \mathbb{N}}, \bu_h, f, b_0, q_0)$ satisfying 
\begin{equation}
\begin{aligned}
\label{dataTQGBlowup2}
&\bm{\xi}_i\in W^{4,\infty}_{\mathrm{div}}(\mathbb{T}^2),
\qquad \sum_{i\in \mathbb{N}} \Vert \bm{\xi}_i \Vert_{4,\infty}<\infty,\qquad\bu_h \in W^{3,2}_{\mathrm{div}}(\mathbb{T}^2), \qquad  f\in W^{2,2}(\mathbb{T}^2),
\\
(b_0, &q_0) \in L^\infty\big(\Omega; W^{3,2}(\mathbb{T}^2) \times W^{2,2}(\mathbb{T}^2)\big) \text{ is a pair of } \mathcal{F}_0\text{-measurable random variables.}
\end{aligned}
\end{equation}
rather than \eqref{dataTQGBlowup}. Once we derive the estimate for \eqref{dataTQGBlowup2},  we can conclude that the same estimate hold for \eqref{dataTQGBlowup} by employing the truncation argument in \cite[Section 5.3.3]{breit2018stoch}.
\begin{lemma}
\label{lem:bkmTQG}
Let $(b,q, (\tau_R)_{R\in \mathbb{N}}, \tau)$ be a unique maximal strong pathwise solution of \eqref{ce}--\eqref{constrt} with a dataset
$((\bm{\xi}_i)_{i\in \mathbb{N}}, \bu_h, f, b_0, q_0)$ satisfying \eqref{dataTQGBlowup2}. For any fixed $M \in \mathbb{N}$,
let $\tau_M^1$ be as
defined Theorem \ref{thm:bkmTQG}.
Then for a deterministic $T_*>0$,
\begin{equation}
\begin{aligned}
\mathbb{E} \sup_{t\in[0,\tau_M^1 \wedge T_*]}
&
\ln[\mathrm{e} + \Vert  b(t)\Vert_{3,2}^2 + \Vert  q(t)\Vert_{2,2}^2]
\lesssim
\mathbb{E}\big(2+  \Vert  b_0\Vert_{3,2}^2 + \Vert  q_0\Vert_{2,2}^2\big)
+
 \sqrt{T_*}
\end{aligned}
\end{equation}
with a constant depending only on $(\bm{\xi}_i)_{i\in \mathbb{N}},\bu_h,f,M$
\end{lemma}
\begin{proof}
In the following, we assume that all functions are sufficiently regular. To make the analysis rigorous, one can perform the calculations for the approximating sequence and pass to the limit just as was done in Section \ref{sec:pathSolTrunc}. We don't do that to avoid unnecessary repetitions.

First of all, by inspecting the analysis in Section \ref{subsub:uniEst}, we can conclude that by
applying It\^o's formula to the mapping $t\mapsto \Vert b(t)\Vert_{3,2}^2$ with $\vert \beta \vert\leq 3$, we obtain
\begin{equation}
\begin{aligned}
\dd \Vert  b(t)\Vert_{3,2}^2 
&
= 
\sum_{\vert \beta \vert \leq 3}
\bigg(
2 \int_{\mathbb{T}^2} S_1 \partial^\beta  b \dx\dt 
-  
2 \int_{\mathbb{T}^2}\partial^\beta[(\bm{\xi}_i\cdot\nabla  ) b ]
\partial^\beta  b \dx\dd W^i
\\&
-
\big\Vert \partial^\beta  [ (\bm{\xi}_i\cdot\nabla ) b] \big\Vert_2^2 \dt
-
\int_{\mathbb{T}^2}  \partial^\beta   \big[(\bm{\xi}_i\cdot\nabla)[ (\bm{\xi}_i\cdot\nabla ) b]\big]\, \partial^\beta  b \dx\dt \bigg)
\end{aligned}
\end{equation}
$\mathbb{P}$-a.s for all $t\in[0,T]$ where
\begin{align*}
S_1&:= ( \bu \cdot \partial^\beta  \nabla) b
-
\partial^\beta((\bu \cdot \nabla)  b )
\end{align*}
is such that
\begin{align}
\label{estS1}
\Vert S_1 \Vert_2  &\lesssim
 \Vert \nabla \bu \Vert_{\infty}\Vert b \Vert_{3,2}
+
\Vert \nabla b \Vert_{\infty}(1+ \Vert q \Vert_{2,2}).
\end{align}
Therefore, in combination with Lemma \ref{rem:xiAppendix}, we obtain
\begin{equation}
\begin{aligned}
\label{b32Blowup}
\dd \Vert  b(t)\Vert_{3,2}^2 
&
+2
\big\langle
(\bm{\xi}_i\cdot\nabla  ) b\, ,\,b
\big\rangle_{3,2}\dd W^i
\lesssim
(
 \Vert \nabla \bu \Vert_{\infty}
+
\Vert \nabla b \Vert_{\infty})\big(\mathrm{e}+\Vert b \Vert_{3,2}^2+ \Vert q \Vert_{2,2}^2\big)\dt.
\end{aligned}
\end{equation}
Similarly, we obtain
\begin{equation}
\begin{aligned}
\dd
 \Vert  q(t  )\Vert_{2,2}^2 
=&\sum_{\vert \beta \vert \leq 2} 
\bigg(-2 \int_{\mathbb{T}^2}[ (\bu \cdot \nabla \partial^\beta ) b  
] \partial^\beta  q \dx\dt  
- 
2\int_{\mathbb{T}^2}[ (\bu_h \cdot \nabla \partial^\beta )b  
] \partial^\beta q \dx\dt 
\\&+
2\int_{\mathbb{T}^2}(S_2 +S_3+S_4) \partial^\beta  q \dx\dt 
-  
2\int_{\mathbb{T}^2}\partial^\beta [(\bm{\xi}_i\cdot\nabla  ) (q - b) ]
\partial^\beta  q \dx\dd W^i
\\&
-
\int_{\mathbb{T}^2}\partial^\beta
\big[(\bm{\xi}_i\cdot\nabla)[ (\bm{\xi}_i\cdot\nabla)(q - 2b)]\big] \partial^\beta  q \dx\dt
  -
 \big\Vert \partial^\beta  [ (\bm{\xi}_i\cdot\nabla)(q - b)]\big\Vert_2^2 \dt   \bigg)
\end{aligned}
\end{equation}
where
\begin{align*}
S_2&:= -(\bu \cdot \partial^\beta  \nabla) b
+
\partial^\beta (\bu \cdot \nabla  b ),
\\
S_3&:= (\bu \cdot \partial^\beta  \nabla) q
-
\partial^\beta (\bu \cdot \nabla  q ) ,
\\
S_4&:= - (\bu_h \cdot \partial^\beta  \nabla) b
+
\partial^\beta (\bu_h \cdot \nabla  b )
\end{align*}
are such that
\begin{align}
\Vert S_2 \Vert_2 &\lesssim
\Vert \nabla \bu \Vert_{\infty}\Vert b \Vert_{3,2}
+
\Vert \nabla b \Vert_{\infty} 
 (1+\Vert q \Vert_{2,2}),
\\
\Vert S_3 \Vert_2 &\lesssim
\Vert \nabla \bu \Vert_{\infty}\Vert q \Vert_{2,2}
+
\Vert  q \Vert_{\infty} 
( 1+\Vert  q \Vert_{2,2}),
\\
\Vert S_4 \Vert_2 &\lesssim
\Vert b \Vert_{3,2}
+
\Vert \nabla b \Vert_{\infty} .
\end{align}
Additionally, the following estimates  holds true
\begin{align}
\label{poincare}
\Big\vert \sum_{\vert \beta \vert \leq 2}\big\langle ( \bu \cdot \nabla\partial^\beta b  ) \, ,\,\partial^\beta q  \big\rangle \Big\vert
&\lesssim
\Vert \bu  \Vert_\infty \Vert b \Vert_{3,2}^2
+
\Vert \bu  \Vert_\infty 
\Vert  q \Vert_{2,2}^2
\lesssim
\Vert \nabla\bu  \Vert_\infty \Vert b \Vert_{3,2}^2
+
\Vert \nabla\bu  \Vert_\infty 
\Vert  q \Vert_{2,2}^2,
\\
\Big\vert \sum_{\vert \beta \vert \leq 2}\big\langle ( \bu_h\cdot \nabla\partial^\beta b  ) \, ,\,\partial^\beta  q  \big\rangle \Big\vert
&\lesssim
\Vert b \Vert_{3,2}^2 
+ \Vert  q \Vert_{2,2}^2
\end{align}
since $\bu_h \in W^{3,2}(\mathbb{T}^2)$. Notice that we have used Poincar\'e's inequality in \eqref{poincare}, recall Remark \ref{rem:meanzero}.
\\
Subsequently, by employing  Lemma \ref{rem:xiAppendix}, we obtain
\begin{equation}
\begin{aligned}
\label{q22Blowup}
\dd \Vert  q(t)\Vert_{2,2}^2 
+2
\big\langle
(\bm{\xi}_i\cdot\nabla  ) (q-b)\, ,\,q
\big\rangle_{2,2}\dd W^i
&\lesssim
(1+
 \Vert \nabla \bu \Vert_{\infty}
+
\Vert \nabla b \Vert_{\infty}
+
\Vert q \Vert_{\infty}
)
\\&
\times
\big(\mathrm{e}+\Vert b \Vert_{3,2}^2+ \Vert q \Vert_{2,2}^2\big)\dt.
\end{aligned}
\end{equation}
If we now define
\begin{align*}
&g(t):= \mathrm{e} + \Vert  b(t)\Vert_{3,2}^2 + \Vert  q(t)\Vert_{2,2}^2,
\\
&\mathcal{M}(t):=2\int_0^t
\frac{1}{g(s)}
\Big[
\big\langle
(\bm{\xi}_i\cdot\nabla  ) b\, ,\,b
\big\rangle_{3,2}
+
\big\langle
(\bm{\xi}_i\cdot\nabla  ) (q-b)\, ,\,q
\big\rangle_{2,2}
\Big]
\dd W^i,
\end{align*}
then it follows from \eqref{b32Blowup} and \eqref{q22Blowup} that
\begin{equation}
\begin{aligned}
\dd \ln[g(t)]\leq c\,(1+
 \Vert \nabla \bu \Vert_{\infty}
+
\Vert \nabla b \Vert_{\infty}
+
\Vert q \Vert_{\infty}
)\dt
-
\dd \mathcal{M}(t).
\end{aligned}
\end{equation}
However, from \eqref{velovortiEst} (where $w=f-q$) and the fact that $\ln[g(t)]\geq 1$, we have that  
 \begin{equation}
\begin{aligned}
1+
 \Vert \nabla \bu \Vert_{\infty}
+
\Vert \nabla b \Vert_{\infty}
+
\Vert q \Vert_{\infty}
\lesssim
\ln[g(t)] 
\big( 1
+
\Vert \nabla b \Vert_{\infty}
+
\Vert q \Vert_{\infty}
\big)
\end{aligned}
\end{equation}
and thus,
\begin{equation}
\begin{aligned}
\dd \ln[g(t)] \leq c\,\ln[g(t)] 
\big( 1
+
\Vert \nabla b \Vert_{\infty}
+
\Vert q \Vert_{\infty}
\big)\dt
-
\dd \mathcal{M}(t).
\end{aligned}
\end{equation}
If we now define
\begin{align*}
G(t):= c(1
+
\Vert \nabla b(t) \Vert_{\infty}
+
\Vert q(t) \Vert_{\infty})
\end{align*}
and use It\^o's product rule, we can conclude that
\begin{equation}
\begin{aligned}
e^{-\int_0^t G(\sigma)\ds} \ln[g(t)]
\leq 
\ln[g(0)]
-
\int_0^t e^{-\int_0^s G(\sigma)\ds}\dd \mathcal{M}(s).
\end{aligned}
\end{equation}
By the Burkholder--Davis--Gundy inequality, the first two assumptions in \eqref{dataTQGBlowup2}, and the boundedness of $e^{-\int_0^s G(\sigma)\ds}$, it follows that for  any deterministic $T_*>0$ and $M\in \mathbb{N}$,
\begin{equation}
\begin{aligned}
\mathbb{E} \sup_{t\in[0,\tau_M^1 \wedge T_*]}
\bigg\vert
\int_0^t &e^{-\int_0^s G(\sigma)\ds}\dd \mathcal{M}(s)
\bigg\vert
\\&
\lesssim
\mathbb{E}
\bigg(
\int_0^{\tau_M^1 \wedge T_*}
\bigg\vert
\frac{1}{g(s)}
\sum_{i\in \mathbb{N}}
\Big[
\big\langle
(\bm{\xi}_i\cdot\nabla  ) b\, ,\,b
\big\rangle_{3,2}
+
\big\langle
(\bm{\xi}_i\cdot\nabla  ) (q-b)\, ,\,q
\big\rangle_{2,2}
\Big]
\bigg\vert^2 \dd s \bigg)^\frac{1}{2}
\\&
\lesssim
\mathbb{E}
\bigg(
\int_0^{\tau_M^1 \wedge T_*}
\frac{1}{[g(s)]^2}
\sum_{i\in \mathbb{N}}\Vert \bm{\xi}_i
\Vert_{4,\infty}^2\big(\Vert b \Vert_{3,2}^2
+
\Vert q \Vert_{2,2}^2
\Big)^2
 \dd s \bigg)^\frac{1}{2}
 \\&
\lesssim
\mathbb{E}
\bigg(
\int_0^{\tau_M^1 \wedge T_*}
\frac{1}{[g(s)]^2}[g(s) -\mathrm{e}
]^2
 \dd s \bigg)^\frac{1}{2}
 \lesssim
 \sqrt{T_*}.
\end{aligned}
\end{equation}
Therefore,
\begin{equation}
\begin{aligned}
\mathbb{E} \sup_{t\in[0,\tau_M^1 \wedge T_*]}
\bigg\vert
&
e^{-\int_0^t G(\sigma)\ds} \ln[g(t)] 
\bigg\vert
\lesssim
\mathbb{E} \ln[g(0)]
+
 \sqrt{T_*}.
\end{aligned}
\end{equation}
On the other hand,
\begin{equation}
\begin{aligned}
e^{-c(1+M)}\mathbb{E} \sup_{t\in[0,\tau_M^1 \wedge T_*]}
\big\vert
&
 \ln[g(t)] 
\big\vert
\leq
\mathbb{E} \sup_{t\in[0,\tau_M^1 \wedge T_*]}
\bigg\vert
&
e^{-\int_0^t G(\sigma)\ds} \ln[g(t)]
\bigg\vert.
\end{aligned}
\end{equation}
Combining these two estimates with $ \ln[g(0)]\leq 1+ \Vert b_0\Vert_{3,2}^2 + \Vert q_0\Vert_{2,2}^2$ finishes the proof.
\end{proof}
We are now in a position to prove Theorem \ref{thm:bkmTQG} by adapting the arguments in \cite[Theorem 17]{crisan2019solution} for the Euler equation to our setting.

\begin{proof}[Proof of Theorem \ref{thm:bkmTQG}]
Due to the continuity of the embedding $W^{3,2}\times W^{2,2}(\mathbb{T}^2) \hookrightarrow W^{1,\infty}\times L^{\infty}(\mathbb{T}^2)$, for any deterministic $s>0$, there exists a deterministic $c>0$ such that
\begin{align*}
\int_0^{\tau_N^2 \wedge s}\big( \Vert \nabla b(t)\Vert_{\infty} +
 \Vert q(t)\Vert_\infty \big)\dd t
 &\leq
 c(\tau_N^2 \wedge s)
\sup_{t\in [0,\tau_N^2 \wedge s]}
 (\Vert b(t)\Vert_{3,2} + \Vert q(t) \Vert_{2,2})
 \leq ([cs]+1)N
\end{align*}
where $[a]$ denotes the integer part of $a$.
\begin{figure}[h!]
\vspace{-5pt}
\centering
\includegraphics[width=0.65\textwidth]{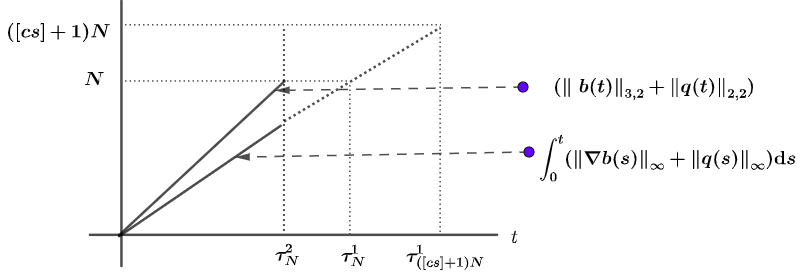}
\vspace{-12pt}
\caption{Relationship between the various stopping times where we set $cs=1$ for simplicity.}
\vspace{-5pt}
\end{figure}

\noindent Therefore, the inequalities $\tau_{N}^2 \leq \tau_{([cs]+1)N}^1 \leq \tau^1$ holds $\mathbb{P}$-a.s. from which we obtain
\begin{align}
\label{tau2tau1}
\tau^2 \leq \tau^1 \qquad \mathbb{P}\text{-a.s.}
\end{align}
Similarly, we have that for any deterministic $\tilde{s}>0$,
\begin{align*}
\int_0^{\tau_R \wedge \tilde{s}}\big( \Vert \nabla b(t)\Vert_{\infty} +
 \Vert q(t)\Vert_\infty \big)\dd t
 &\leq
 (\tau_R \wedge \tilde{s})
\sup_{t\in [0,\tau_R \wedge \tilde{s}]}
\big( \Vert \nabla b(t)\Vert_{\infty}+\Vert \nabla \bu(t)\Vert_{\infty} +
 \Vert q(t)\Vert_\infty \big)
 \leq ([\tilde{s}]+1)R
\end{align*}
and thus, $\tau_R \leq \tau_{([\tilde{s}]+1)R}^1 \leq \tau^1$ holds $\mathbb{P}$-a.s. Therefore,
\begin{align}
\label{tautau1}
\tau \leq \tau^1 \qquad \mathbb{P}\text{-a.s.}
\end{align}
It remains to show the reverse inequalities for \eqref{tau2tau1} and \eqref{tautau1} in order to obtain \eqref{tau1tau2tau}. For the reverse of \eqref{tau2tau1}, we first note from Lemma \ref{lem:bkmTQG} that
\begin{equation}
\begin{aligned}
\mathbb{P}\bigg( \sup_{t\in[0,\tau_M^1 \wedge T_*]}
&
\big( \Vert  b(t)\Vert_{3,2}^2 + \Vert  q(t)\Vert_{2,2}^2\big)
<\infty
\bigg)=1.
\end{aligned}
\end{equation}
Now since
\begin{equation}
\begin{aligned}
\bigg\{&\omega\in \Omega \, : \, \sup_{t\in[0,\tau_M^1(\omega) \wedge T_*]}
\big( \Vert  b(t,\omega)\Vert_{3,2}^2 + \Vert  q(t,\omega)\Vert_{2,2}^2\big)
<\infty \bigg\}
\\&
=
\bigcup_{N\in \mathbb{N}}
\bigg\{\omega\in \Omega \, : \, \sup_{t\in[0,\tau_M^1(\omega) \wedge T_*]}
\big( \Vert  b(t,\omega)\Vert_{3,2}^2 + \Vert  q(t,\omega)\Vert_{2,2}^2\big)
<N \bigg\}
\\&
\subset
\bigcup_{N\in \mathbb{N}}
\big\{ \omega\in \Omega \, : \, \tau_M^1(\omega)\wedge T_* < \tau_N^2(\omega)\big\}
\\&
\subset
\big\{ \omega\in \Omega \, : \, \tau_M^1(\omega)\wedge T_* < \tau^2(\omega)\big\}
\end{aligned}
\end{equation}
holds, it follows that $\tau_M^1\wedge T_* < \tau^2$ holds $\mathbb{P}$-a.s. Furthermore,
\begin{equation}
\begin{aligned}
\big\{ \omega\in \Omega \, : \, \tau^1(\omega) \leq \tau^2(\omega)\big\}
&=
\Big\{ \omega\in \Omega \, : \, \lim_{M\rightarrow \infty}\tau_M^1(\omega) \leq \tau^2(\omega)\Big\}
\\&=
\bigcap_{M\in \mathbb{N}}
\Big\{ \omega\in \Omega \, : \,\tau_M^1(\omega) \leq \tau^2(\omega)\Big\}
\\&=
\bigcap_{M\in \mathbb{N}}\bigcap_{T_*>0}
\Big\{ \omega\in \Omega \, : \,\tau_M^1(\omega) \wedge T_*\leq \tau^2(\omega)\Big\}.
\end{aligned}
\end{equation}
Since all sets in the last equation have full measure, we can conclude that
\begin{align}
\label{tau2tau1Reverse}
\tau^1 \leq \tau^2 \qquad \mathbb{P}\text{-a.s.}
\end{align}
It remains to show the reverse inequality for \eqref{tautau1}. For this, first notice that since
\begin{align*}
\frac{\mathrm{e}}{2} + \Vert \nabla b(t)\Vert_{\infty}^2 + \Vert \nabla \bu(t)\Vert_{\infty}^2 + \Vert  q(t)\Vert_{\infty}^2
\lesssim
\mathrm{e} + \Vert  b(t)\Vert_{3,2}^2 + \Vert  q(t)\Vert_{2,2}^2,
\end{align*}
it follows from Lemma \ref{lem:bkmTQG} that
for a deterministic $T_*>0$,
\begin{equation}
\begin{aligned}
\mathbb{E} \sup_{t\in[0,\tau_M^1 \wedge T_*]}
&
\big(1+ \ln[\mathrm{e} +\Vert \nabla b(t)\Vert_{\infty}^2 + \Vert \nabla \bu(t)\Vert_{\infty}^2 + \Vert  q(t)\Vert_{\infty}^2]\big)
\lesssim
\mathbb{E}\big(2+  \Vert  b_0\Vert_{3,2}^2 + \Vert  q_0\Vert_{2,2}^2\big)
+
 \sqrt{T_*}
\end{aligned}
\end{equation}
holds with a constant depending only on $(\bm{\xi}_i)_{i\in \mathbb{N}},\bu_h,f,M$
. Thus,
\begin{equation}
\begin{aligned}
\mathbb{P}\bigg( \sup_{t\in[0,\tau_M^1 \wedge T_*]}
&
\big( \Vert \nabla b(t)\Vert_{\infty}^2 + \Vert \nabla \bu(t)\Vert_{\infty}^2 + \Vert  q(t)\Vert_{\infty}^2\big)
<\infty
\bigg)=1.
\end{aligned}
\end{equation}
A similar argument leading to \eqref{tau2tau1Reverse} therefore yields
\begin{align}\label{tautau1reverse}
\tau^1 \leq \tau \qquad \mathbb{P}\text{-a.s}
\end{align}
Combing \eqref{tau2tau1}, \eqref{tautau1}, \eqref{tau2tau1Reverse} and \eqref{tautau1reverse} finishes the proof.
\end{proof}

\section{Time discretisation and numerical consistency }

\label{sec:consitency} 
In this section we investigate the application of the \emph{strong stability preserving Runge-Kutta of order 3} (SSPRK3) scheme to the stochastic TQG equations; see \cite{Hesthaven2008} for a description of the scheme applied to deterministic systems. We show that the scheme is numerically consistent. We only consider the semi-discretise case in our analysis. However, for completeness, the mixed finite element scheme we use for the spatial derivatives is included in Appendix \ref{sec: FEM}. 

First let us note that Theorem \ref{thm:main}
ensures the existence of a unique maximal strong pathwise
solution $(b,q, \tau)$ of equations  \eqref{ce}--\eqref{constrt}. In particular, the main result does not ensure 
the existence of a \emph{ global} solution, that it may be possible that 
the set $\tau<\infty$ has positive probability. This is cumbersome as that the consistency of the numerical scheme itself may be affected by the finite time blow up. To avoid this technical complication we show the consistency of the numerical scheme when applied to the \emph{truncated} equation. More precisely we will assume that $(b, q)$ is the solution of the 
system of equations 
\begin{align}
\label{ct}
 {b}(t)&= {b}(0) - \int_0^t \theta_R[( {\bu} \cdot \nabla)  {b} ]\ds -\int_0^t (\bm{\xi}_i\cdot\nabla)  {b} \circ \dd  {W}^i_\sigma,
\\ \label{ct2}
 {q}(t) &=  {q}(0) - \int_0^t\theta_R[( {\bu}\cdot \nabla)(  {q} -  {b})]\ds - \int_0^t(\bu_h \cdot \nabla)  {b} \ds  - \int_0^t(\bm{\xi}_i\cdot\nabla) ( {q}-  {b}) \circ \dd  {W}^i_\sigma
\end{align}
where $R$ is large, but fixed truncation parameter. Following 
Theorem 
\ref{prop:smoothPathwise}, the system \eqref{ct} - \eqref{ct2} has a unique
\emph{global} solution $(b,q)$ such that 
$b \in L^2 \big( \Omega;C([0,T]; W^{3,2}(\mathbb{T}^2)) \big)$ and 
$q \in L^2 \big( \Omega;C([0,T]; W^{2,2}(\mathbb{T}^2)) \big)$. The consistency proof presented below assumes that $b \in L^2 \big( \Omega;C([0,T]; W^{4,2}(\mathbb{T}^2)) \big)$ and 
$q \in L^2 \big( \Omega;C([0,T]; W^{3,2}(\mathbb{T}^2)) \big)$. This holds true provided initial conditions of the system 
are chosen from the same space (i.e., they have the same regularity as required by the consistency condition). 

\subsection{Time discretisation}
Equations \eqref{ct}--\eqref{ct2} can be expressed in a compact form
\begin{align}
\label{abstractTQGtruncated}
\dd \binom{b}{q} =
\mathcal{A}_R
\begin{pmatrix}
 b \\
 q
 \end{pmatrix}\dt
 +
\mathcal{G}_i
\begin{pmatrix}
 b \\
 q
 \end{pmatrix}\circ \dd W^i
,
\end{align}
where
\begin{align}\label{compact matrix operators}
\mathcal{A}_R
:= 
-
\begin{bmatrix}
\theta_R \bu\cdot\nabla & 0    \\[0.3em]
(\bu_h - \theta_R \bu)\cdot\nabla & \theta_R \bu\cdot \nabla
\end{bmatrix}
,
\quad 
\mathcal{G}_i
:= -
\begin{bmatrix}
\bm{\xi}_i\cdot\nabla & 0    \\[0.3em]
-\bm{\xi}_i\cdot\nabla &\bm{\xi}_i\cdot \nabla
\end{bmatrix}.
\end{align}
Furthermore, the conversion from the Stratonovich to It\^o integral $\mathcal{G}_i\mathbf{g}\circ \dd W^i \mapsto \frac{1}{2}\mathcal{G}_i^2\mathbf{g}\dt+\mathcal{G}_i\mathbf{g}\dd W^i$, $\mathbf{g}:=(b,q)^T$ yields the following equivalent It\^o form for \eqref{ct}--\eqref{ct2} 
\begin{align}
\dd b + \theta_R[(\bu \cdot \nabla) b ]\dt
-
\frac{1}{2}
(\bm{\xi}_i\cdot\nabla)(\bm{\xi}_i\cdot\nabla)b 
\dt
+(\bm{\xi}_i\cdot\nabla) b \, \dd W^i =0, \label{eq: truncated b}
\\
\dd q + \theta_R[(\bu\cdot \nabla)( q -b)]\dt -
\frac{1}{2}(\bm{\xi}_i\cdot\nabla)(\bm{\xi}_i\cdot\nabla) (q-2b)\dt 
 + (\bm{\xi}_i\cdot\nabla) (q-b)  \dd W^i= -(\bu_h \cdot \nabla) b \dt. \label{eq: truncated q}
\end{align}

The SSPRK3
time discretisation scheme applied to the truncated STQG system gives the following time stepping equations
\begin{subequations}\label{eq: ssprk3 stqg}
\begin{align}
    \bq^{(1)} & = \bq^{(n)} + \Delta t \ \mathcal{A}_R^{(n,n)} \bq^{(n)} 
    + \mathcal{G}_i \bq^{(n)} \Delta W^i \label{eq: ssprk3 stqg a}
    \\
	\bq^{(2)} 
	&=
	\frac{3}{4}\bq^{(n)} +\frac{1}{4}\left(\bq^{(1)}+\Delta t \ \mathcal{A}_R^{(1,1)}\bq^{(1)}
	+ \mathcal{G}_i\bq^{(1)} \Delta W^i
	\right) \label{eq: ssprk3 stqg b}
	\\
	\bq^{(n+1)} 
	&=
	\frac{1}{3}\bq^{(n)} + \frac{2}{3} \left(
	\bq^{(2)} + \Delta t \  \mathcal{A}_R^{(2,2)}\bq^{(2)}
	+ \mathcal{G}_i\bq^{(2)} \Delta W^i
	\right).
	\label{eq: ssprk3 stqg c}
\end{align} 
\label{eq: ssprk3}
\end{subequations} 
For the $\mathcal{A}_R$'s in \eqref{eq: ssprk3 stqg a} -- \eqref{eq: ssprk3 stqg c}, the added superscripts indicate the time step values of $\bu$ and $\theta_R$ that constitute $\mathcal{A}_R$. More specifically, we write
$\mathcal{A}_R^{(m,k)}$ to mean
\begin{align}
\mathcal{A}_R^{(m,k)}
= 
-
\begin{bmatrix}
\theta_R^{(k)} \bu^{(m)}\cdot\nabla & 0    \\[0.3em]
(\bu_h - \theta_R^{(k)} \bu^{(m)})\cdot\nabla & \theta_R^{(k)} \bu^{(m)}\cdot \nabla
\end{bmatrix},
\end{align}
where 
\begin{equation}
    \bu^{(m)} = \nabla^\perp (\Delta -1)^{-1} (q^{(m)} - f), \qquad
    \theta_R^{(k)} = \theta_R(\|\nabla b^{(k)} \|_\infty + \|\nabla \bu^{(k)} \|_\infty + \| \nabla q^{(k)} \|_\infty).
\end{equation}


We write $S_{\Delta t}$ to denote the time discretisation scheme \eqref{eq: ssprk3 stqg}, i.e. 
\begin{equation}
    \bq^{(n+1)} = S_{\Delta t} \bq^{(n)}.
\end{equation}
By substituting \eqref{eq: ssprk3 stqg a} into \eqref{eq: ssprk3 stqg b}, and then substituting the result into \eqref{eq: ssprk3 stqg c}, we can write the scheme in a one-step form
\begin{align} \label{eq: ssprk3 one-step}
S_{\Delta t} \bq^{(n)}
=
\bq^{(n)}
& +
 \frac23
 \Delta t 
 \mathcal{A}_R^{(n,2)} \bq^{(n)}
 + \frac13 \Delta t \mathcal{A}_R^{(n,n)} \bq^{(n)}
+
 \Delta W^i 
\mathcal{G}_i\bq^{(n)}
 \nonumber
 \\+ &\frac12 
 \Delta W^i \Delta W^j 
 \mathcal{G}_i\mathcal{G}_j \bq^{(n)}
 + H.O.T.
,
\end{align}
where $H.O.T.$ is short for \emph{higher order terms}.
$H.O.T.$ can be written out explicitly to show that it consists of terms up to the order of $\Delta t^7$, and contains terms that has up to three derivatives of $b$ and $q$.

\subsection{Numerical consistency}
We consider the local (one-step) truncation error $e_j(\Delta t)$ of the scheme \eqref{eq: ssprk3}, which is defined by
\begin{equation}\label{eq: local truncation error}
    e_n(\Delta t) = \begin{pmatrix}
    b \\
    q
    \end{pmatrix}(t_{n+1})
     - 
     S_\Delta\begin{pmatrix}
     b \\
     q
     \end{pmatrix}(t_n)
\end{equation}
with the norm
\begin{equation}\label{eq: local truncation error norm}
\| e_n(\Delta t) \|_\myH^2 := \|b(t_{n+1}) - b^{(n+1)}\|_\myH^2 + \|q(t_{n+1}) -q^{(n+1)}\|_\myH^2.
\end{equation}
In addition, since the STQG system is Stratonovich, we require a compatibility condition on $S_{\Delta t}$. In the following definition, we let $S^s_{\Delta t}$ denote the "purely" stochastic part of $S_{\Delta t}$, i.e. $S^s_{\Delta t}$ consists of terms of $S_{\Delta t}$ that contain \emph{only} Brownian increments.
\begin{definition}[Stratonovich compatiblity condition] For some $\alpha > 1$, 
\begin{equation}\label{eq: stratonovich compatibility}
    \expect\left[S^s_{\Delta t} \bq(t_n) \big| \mathcal{F}_{t_j}\right]
    = \frac12 \Delta t \ \mathcal {G}_i\mathcal{G}_i\bq(t_n) + O(\Delta t^\alpha).
\end{equation}
\end{definition}

\begin{definition}[Consistency]\label{def: consistency}
We say $S_{\Delta t}$ is consistent in mean square of order $\alpha > 1$ with respect to \eqref{abstractTQGtruncated} if there exists a constant $c > 0$ that does not depend on $\Delta t \in (0, T]$, and for all $\epsilon > 0$, there exists $\delta > 0 $ such that for all $0< \Delta t< \delta$ and $n \in \{1, 2, \dots, N\}$
\begin{equation}
    \expect(\| e_n(\Delta t) \|_\myH^2) < c \Delta t^\alpha
\end{equation}
and $S_{\Delta t}$ satisfies \eqref{eq: stratonovich compatibility}.
\end{definition}

\begin{proposition}[Consistency]
For the truncated STQG equations \eqref{ct} -- \eqref{ct2}, the time discretisation scheme \eqref{eq: ssprk3 stqg} is consistent in the sense of Definition \eqref{def: consistency} with $\alpha=2$. 
\end{proposition}

\begin{proof}
By combining the equations \eqref{abstractTQGtruncated}, \eqref{eq: ssprk3 one-step} and \eqref{eq: local truncation error norm}, and applying Jensen's inequality we get
\begin{equation}\label{eq: local truncation error bound}
\begin{aligned}
    &\expect\left[\|e_n(\Delta t)\|_\myH^2\right] 
    = \expect\left[\|b(t_{n+1}) - b^{n+1}\|_\myH^2\right] + \expect\left[ \|q(t_{n+1}) -q^{n+1}\|_\myH^2\right] 
    \\
    & \leq 4 \expect\left[ \int_{t_n}^{t_{n+1}} \left\|  \theta_R(s) \bu(s) \cdot \nabla b(s) - \big(\frac23 \theta_R^{(2)}
    +\frac13\theta_R^{(n)}\big)\bu(t_n) \cdot \nabla b(t_n)
    \right \|_\myH^2 \  \dd s \right] 
    \\
    & \quad + 4 \expect\left[ \int_{t_n}^{t_{n+1}} \left\| -\bu_h\cdot \nabla b(s) + \theta_R(s)
    \bu(s)\cdot \nabla (b(s)  -   q(s)) \right.\right.\\
    &\qquad \qquad\qquad \left.\left.
     + \bu_h\cdot \nabla b(t_n) - \big(\frac23 \theta_R^{(2)}
    +\frac13\theta_R^{(n)}\big)\bu(t_n)\cdot \nabla (b(t_n)
     - q(t_n))
    \right \|_\myH^2 \  \dd s\right]
    \\
    & \quad
    + 4 \expect\left[ \int_{t_n}^{t_{n+1}} \left\| 
    \xi_i \cdot \nabla 
    (b(t_n) - b(s) )
    \right \|_\myH^2 \  \dd W_s^i\right]
    \\
    & \quad + 4 \expect\left[ \int_{t_n}^{t_{n+1}} \left\| \xi_i \cdot \nabla (b(s) - b(t_n))  + \xi_i \cdot \nabla  (q(t_n) - q(s)) \dd W_s^i \right\|_\myH^2 \right]
    \\
    & \quad 
    + 4\expect
    \left[ 
    \left\|
    \frac12\int_{t_n}^{t_{n+1}} (\xi_i\cdot \nabla)(\xi_i\cdot\nabla) b(s) \dd s 
    - \frac12\int_{t_n}^{t_{n+1}} \int_{t_n}^{t_{n+1}}(\xi_j\cdot \nabla)(\xi_k\cdot\nabla) b(t_n) \dd W_s^j \dd W_t^k
    \right\|_\myH^2
    \right]
    \\
    & \quad
    + 4 \expect\left[
    \left\|
    \frac12\int_{t_n}^{t_{n+1}}  (\xi_i\cdot\nabla)(\xi_i\cdot\nabla)(q(s)- 2b(s)) \dd s
    \right.\right.
    \\
    & \qquad \qquad \qquad \left.\left.-\frac12\int_{t_n}^{t_{n+1}}\int_{t_n}^{t_{n+1}}
    (\xi_j\cdot\nabla)(\xi_k\cdot\nabla)(q(s)- 2b(s))
    \dd W_s^j \dd W_t^k
    \right\|_\myH^2
    \right]
    \\
    & \quad
    + 4 \expect[\|H.O.T.\|_\myH^2]
\end{aligned}
\end{equation}
Each term in \eqref{eq: local truncation error bound} can be shown to be of $O(\Delta t^2)$. We now go through the arguments for the first term, i.e. 
\begin{equation}
\expect\left[ \int_{t_n}^{t_{n+1}} \left\|  \big(\frac23 \theta_R^{(2)}
    +\frac13\theta_R^{(n)}\big)\bu(t_n) \cdot \nabla b(t_n)- \theta_R(s)\bu(s) \cdot \nabla b(s) \right \|_\myH^2 \  \dd s \right] = O(\Delta t^2).
\end{equation}
Using \eqref{eq: truncated q} we obtain
\begin{align}
\bu(s)
    &=  \bu(t_n) - \myK\left( \int_{t_n}^s \left[\theta_R(\tau)\bu(\tau) \cdot \nabla(q(\tau) - b(\tau)) 
    - 
    \right.
    \right.
    \nonumber
    \\ &
    \qquad \qquad \qquad
    \left.
    \left.
    \bu_h \cdot \nabla b(\tau) \right] \dd\tau 
    - \int_{t_n}^s \xi_i \cdot \nabla \left[q(\tau) - b(\tau) \right] \circ \dd W_{\tau}^i
    \right)
    \label{eq: truncated u}
\end{align}
Thus
\begin{align}
    & \expect\left[ \int_{t_n}^{t_{n+1}} \left\|  \big(\frac23\theta_R^{(2)}+\frac13\theta_R^{(n)}\big)\bu(t_n) \cdot \nabla b(t_n)- \theta_R(s)\bu(s) \cdot \nabla b(s) \right\|_\myH^2 \  \dd s \right]  \\
  = &
  \expect
  \left[ 
            \int_{t_n}^{t_{n+1}} 
  \left\|
        \left(
        \theta_R(s)
    - 
    \big(\frac23 \theta_R^{(2)}
    +\frac13\theta_R^{(n)}\big)\right)
    \bu(t_n) \cdot \nabla b(t_n)
\right.
\right.
            \nonumber
            \\
            & \qquad
            \left.\left.
- 
\theta_R(s)
\left\{\myK
  \left( 
  \int_{t_n}^s 
      \left[
      \theta_R(\tau)
      \bu(\tau) \cdot \nabla(q(\tau) - b(\tau)) - \bu_h \cdot \nabla b(\tau) 
      \right] d\tau 
  \right.
  \right.
  \right.
  \right.
  \nonumber
  \\ &
  \qquad
  \left.
    \left.
            - \int_{t_n}^s \xi_i \cdot \nabla \left[q(\tau) - b(\tau) \right] \circ dW_{\tau}^i
    \right)
    \right\}
            \cdot\nabla \left\{
            b(t_n) - \int_{t_n}^s
            \theta_R(\tau) \bu(\tau)\cdot\nabla b(\tau)d\tau - \int_{t_n}^s \xi_i\cdot\nabla b(\tau) \circ dW_\tau^i
            \right\}
    \nonumber
    \\
    &
    \qquad\qquad\qquad
    \left.
    \left.
    -  \theta_R(s)
       \bu(t_n)\cdot\nabla
       \left\{ 
       \int_{t_n}^s
       \theta_R(\tau)
       \bu(\tau)\cdot\nabla b(\tau)d\tau - \int_{t_n}^s \xi_i\cdot\nabla b(\tau) \circ dW_\tau^i
       \right\}
    \right \|_\myH^2 \  ds 
  \right]\label{eq: order_proof_term_0}
\end{align}
In \eqref{eq: order_proof_term_0}, as written, except for the first term, all other terms consist of deterministic and stochastic integrals. Thus by Cauchy-Schwartz, It\^o isometry and Theorem 
\ref{prop:smoothPathwise}, all terms except the first are of order 2 or higher in $\Delta t$. 

To estimate the first term of \eqref{eq: order_proof_term_0}, we have
\begin{align}
& \expect
\left[
    \int_{t_n}^{t_{n+1}} 
    \left\| 
        \left(
          \theta_R(s) -  \big(\frac23\theta_R^{(2)}+\frac13\theta_R^{(n)}\big)
        \right)
        \bu(t_n) \cdot \nabla b(t_n)
    \right\|_\myH^2 \  
    \dd s 
\right]
\\
\leq
&
2 \expect
\left[
    \int_{t_n}^{t_{n+1}} 
    \left\| 
           \frac23\Big(\theta_R(s) - \theta_R^{(2)}\Big)
           \bu(t_n) \cdot \nabla b(t_n)
    \right\|_2^2
            + 
    \left\|
    \frac13\Big(\theta_R(s) - \theta_R^{(n)} \Big)
        \bu(t_n) \cdot \nabla b(t_n)
    \right\|_2^2
    \  
    \dd s 
\right]. \label{eq: upper bound 1}
\end{align}
Now, we can expand the second intermediate step of the time stepping scheme \eqref{eq: ssprk3 stqg b} to obtain
\begin{equation} \bq^{(2)}
=
\bq^{(n)}
+
 \frac12
 \Delta t 
 \mathcal{A}_R^{(n,n)} \bq^{(n)}
 +
 \frac12 \Delta W^i 
\mathcal{G}_i\bq^{(n)}
  + O(\Delta t \Delta W^i)
,
\end{equation}
in which the terms that constitute $O(\Delta t \Delta W^i)$ depend on $\bq^{(n)}$ only. Then, 
using the fact that $\theta_R$ and the norms (see \eqref{eq: theta_r}) are Lipschitz,
the terms in \eqref{eq: upper bound 1} can be estimated as follows,
\begin{align}
&\expect
\left[
    \int_{t_n}^{t_{n+1}} 
        \left\| 
           \Big(\theta_R(s) - \theta_R^{(2)}\Big)
           \bu(t_n) \cdot \nabla b(t_n)
        \right\|^2_2
        \dd x
    \dd s
\right] 
\\
\leq
&
C_1\ \expect
\left[\left\|
        \bu(t_n) \cdot \nabla b(t_n)
      \right\|_2^2
\right]
\expect\left[
    \int_{t_n}^{t_{n+1}} 
          \Big(
          \|\nabla b(s) - \nabla b^{(n)} \|_{k,2}^2
          + 
          \| \nabla \bu(s) - \nabla \bu^{(n)} \|_{k,2}^2
\right.
\nonumber
\\
&\qquad\qquad\qquad
\left.
        + \|q(s) - q^{(n)} \|_{k,2}^2
        + \|O(\Delta t,\ \Delta W) \|_{k,2}^2
        \Big)
    \dd s
\right]\label{eq: order_proof_term_1}
\end{align}
and
\begin{align}
&\expect
\left[
    \int_{t_n}^{t_{n+1}} 
        \left\| 
           \Big(\theta_R(s) - \theta_R^{(n)}\Big)
           \bu(t_n) \cdot \nabla b(t_n)
        \right\|^2_2
        \dd x
    \dd s
\right] 
\\
\leq
&
C_2\ \expect
\left[\left\|
        \bu(t_n) \cdot \nabla b(t_n)
      \right\|_2^2
\right]
\expect\left[
    \int_{t_n}^{t_{n+1}} 
          \Big(
          \|\nabla b(s) - \nabla b^{(n)} \|_{k,2}^2
          + 
          \| \nabla \bu(s) - \nabla \bu^{(n)} \|_{k,2}^2
\right.
\nonumber
\\
&\qquad\qquad\qquad
\left.
        + \|q(s) - q^{(n)} \|_{k,2}^2
        \Big)
    \dd s
\right]\label{eq: order_proof_term_2}
\end{align}
where we have also used the Sobolev embedding $\mathcal{W}^{k,2}(\mathbb{T}^2)\hookrightarrow L^\infty(\mathbb{T}^2)$, see \cite[Theorem~4.12]{adams2003sobolev}, to introduce $\|\cdot\|_{k,2}$. Substituting in \eqref{ct}, \eqref{ct2} and \eqref{eq: truncated u} for $b(s)$, $q(s)$ and $\bu(s)$ 
in \eqref{eq: order_proof_term_1} and \eqref{eq: order_proof_term_2},
we can show they are both $O(\Delta t^2)$. This follows from Cauchy-Schwartz, It\^o isometry and Theorem 
\ref{prop:smoothPathwise}. 

Therefore, we have
\begin{equation}
    \expect
\left[
    \int_{t_n}^{t_{n+1}} 
    \left\| 
        \left(
          \theta_R(s) -  \big(\frac23\theta_R^{(2)}+\frac13\theta_R^{(n)}\big)
        \right)
        \bu(t_n) \cdot \nabla b(t_n)
    \right\|_\myH^2 \  
    \dd s 
\right] = O(\Delta t^2)
\end{equation}
which concludes the proof.

For the other terms, the same arguments and calculations are applied. We omit their details from this proof.
Though, we note that in addition, for the $H.O.T.$ term, the scheme requires three derivatives of $b(t_n)$ and $q(t_n)$. This is guaranteed by the assumption introduced at the beginning of this section, that is $b\in L^2(\Omega; C([0,T]; W^{4,2}(\mathbb{T}^2)))$ and $q\in L^2(\Omega; C([0,T]; W^{3,2}(\mathbb{T}^2)))$.
\end{proof}

\begin{lemma}
The discretisation \eqref{eq: ssprk3 one-step} satisfies \eqref{eq: stratonovich compatibility}.
\end{lemma}
\begin{proof}
This follows directly from applying $\expect[\ \cdot\ |\mathcal{F}_{t_n}]$ to $S^s_{\Delta t}$.
\end{proof}

\begin{appendices}

\section{Geometric considerations of TQG flows}\label{App-A}

The TQG conservation laws in \eqref{eq:casimirstqg} for divergence-free flow are in the same form as would be satisfied in the dynamics of semidirect-product Lie-Poisson bracket for fluid flow in the variables $(q,b)$ \cite{HMR1998}. Thus, one asks whether TQG can be cast by a change of variables into Hamiltonian form and endowed with a semidirect-product Lie-Poisson bracket.  If this is possible, then even though it does not follow from Hamilton's principle the TQG model would fit into the standard Hamiltonian framework for all ideal fluids with advected quantities, see \cite{HMR1998}. 

We introduce the variable $\varpi$, given by
\[
\varpi := (1-\Delta)\psi + b  = f_1 - q + b
\quad\hbox{so that}\quad
\psi = (1-\Delta)^{-1}(\varpi - b)\,.
\]
Consequently, when written in  the variables $(\varpi,b)$, the energy in \eqref{Erg-TQG} becomes
the TQG Hamiltonian,
\begin{equation}
\mc H_{TQG}(\varpi,b ) =  \frac12 \int_{\mc D} 
 (\varpi - b)(1-\Delta)^{-1}(\varpi - b)
+ (b+h_1)^2/2 \ \dd x \dd y\,.
\label{eq:erg-TQG2}
\end{equation}
One computes the variational derivatives of the Hamiltonian $\mc H_{TQG}(\varpi,b )$ as
\[
\delta \mc H_{TQG}
=
 \int_{\mc D} 
\psi\cdot \delta \varpi 
+ \Big( (b + h_1)/2 - \psi \Big) \delta b \ \dd x\dd y
\,.\]
Now, in the new notation, we can rewrite the TQG equations in standard Hamiltonian form, with a  Lie-Poisson bracket defined on the dual to a semidirect-product  Lie algebra, as discussed for example in \cite{HMR1998}. Explicitly, this is
 \begin{equation}
 \frac{\partial }{\partial t} 
\begin{bmatrix}
\varpi \\ b
\end{bmatrix}
=
\begin{bmatrix}
J(\varpi,\,\cdot\,) & J(b,\,\cdot\,) 
\\ 
J(b,\,\cdot\,)  & 0
\end{bmatrix}
\begin{bmatrix}
{\delta \mc H_{TQG}}/{\delta \varpi} = \psi
\\ 
{\delta \mc H_{TQG}}/{\delta b} = (b + h_1)/2 - \psi
\end{bmatrix}\,.
\label{TQG-Ham3}
\end{equation}
The Hamiltonian operator form of the TQG equations in \eqref{TQG-Ham3} is identical to that of the Euler-Boussinesq convection (EBC) equations in \cite{HolmPan2022}. 
We are now in a position to explain the conservation laws for Hamilton matrix operator in \eqref{TQG-Ham3}. Namely, those conservation laws comprise Casimir functions $C_{\Phi,\Psi}$ whose variational derivatives are null eigenvectors of the Hamilton matrix operator in \eqref{TQG-Ham3} which defines the semidirect-product Lie-Poisson bracket as
 \begin{align}
  \begin{split}
 \frac{\dd }{\dt} \mc{F}(\varpi,b)
 &=
\int_{\mc D}
\begin{bmatrix}
{\delta \mc F}/{\delta \varpi} 
\\ 
{\delta \mc F}/{\delta b} 
\end{bmatrix}^T
\begin{bmatrix}
J(\varpi,\,\cdot\,) & J(b,\,\cdot\,) 
\\ 
J(b,\,\cdot\,)  & 0
\end{bmatrix}
\begin{bmatrix}
{\delta \mc H}/{\delta \varpi} 
\\ 
{\delta \mc H}/{\delta b} 
\end{bmatrix}
\dd x\dd y
=:
\Big\{\mc F\,,\,{\mc H}\Big\}(\varpi,b)
\\&=-
\int_{\mc D} 
\varpi \,J\bigg(\frac{\delta \mc F}{\delta \varpi},\frac{\delta \mc H}{\delta \varpi}\bigg)
+ b \Bigg(J\bigg(\frac{\delta \mc F}{\delta b},\frac{\delta \mc H}{\delta \varpi}\bigg)
- J\bigg(\frac{\delta \mc F}{\delta \varpi},\frac{\delta \mc H}{\delta b}\bigg)\Bigg) \dd x\dd y
\,.
\end{split}
\label{TQG-brkt3}
\end{align}
In particular, for $\mc{F}(\varpi,b)=\int_{\mc D}\Phi(b)+\varpi\Psi(b)\dd x \dd y$ we have 
\begin{align}
  \begin{split}
 \frac{\dd }{\dt} \mc{F}(\varpi,b)
 &=
\Big\{\mc F\,,\,{\mc H}\Big\}(\varpi,b)
\\&=-
\int_{\mc D} 
\varpi \,J\bigg(\Psi(b),\frac{\delta \mc H}{\delta \varpi}\bigg)
+  \Bigg(bJ\bigg(\Phi'(b)+\varpi\Psi'(b),\frac{\delta \mc H}{\delta \varpi}\bigg)
- bJ\bigg(\Psi(b),\frac{\delta \mc H}{\delta b}\bigg)\Bigg) \dd x\dd y
\\&=-
\int_{\mc D} 
\frac{\delta \mc H}{\delta \varpi}J\bigg(\varpi ,\Psi(b)\bigg)
+ \Bigg( \frac{\delta \mc H}{\delta \varpi}J\bigg(b,\Phi'(b)+\varpi\Psi'(b)\bigg)
- \frac{\delta \mc H}{\delta b}J\bigg(b,\Psi(b)\bigg)\Bigg) \dd x \dd y
\,.
\\&=-
\int_{\mc D} 
\frac{\delta \mc H}{\delta \varpi}\Big(J\big(\varpi ,b\big)
+  J\big(b,\varpi\big)\Big)
\Psi'(b) \dd x \dd y
\\& = 0 \quad\hbox{for all}\quad {\mc H}
\,.
\end{split}
\label{TQG-CasimirBrkt3}
\end{align}
This explains the conservation of $C_{\Phi,\Psi}$ in \eqref{eq:casimirstqg}. These quantities are Casimir functions which are conserved for \emph{every} Hamiltonian $\mc{H}(\varpi,b)$ for the semidirect-product Lie-Poisson bracket in \eqref{TQG-brkt3}. This conservation occurs because Lie-Poisson brackets generate coadjoint orbits, and coadjoint orbits are level sets of the bracket's Casimir functions. The requirement that TQG motion takes place on level sets of the Casimir functions limits the function space available to TQG solutions. In particular, the TQG solutions are restricted to stay on the same level set as their initial conditions.

\begin{proposition}\label{LPB-TQG}
The Lie-Poisson bracket in \eqref{TQG-brkt3} satisfies the Jacobi identity. \end{proposition}

\begin{proof}
This proposition could be demonstrated by direct computation using the well-known properties of the Jacobian of function pairs. The proof given here, though, will illustrate the geometric properties of the TQG system in \eqref{trsw-tg} and, thus, place it into the wider class of ideal fluid dynamics with advected quantities. In particular, the proof will identify the Poisson bracket in \eqref{TQG-brkt3} as being defined over domain $\mc D$ on functionals of the \emph{dual}\footnote{Dual with respect to the $L^2$ pairing on $\mc D$} $(f_1\circledS f_2)^*$ of the Lie algebra  $(f_1\circledS f_2)$ of semidirect-product symplectic transformations. Thus, TQG dynamics is understood as coadjoint motion generated by the semidirect-product action of the Lie algebra $(f_1\circledS f_2)$ on function pairs $(f_1;f_2)\in (f_1\circledS f_2)^*$. This proof of coadjoint motion also identifies the potential vorticity and buoyancy, $(\varpi;b)$, as a semidirect-product momentum map.  

The Lie algebra commutator $[\,\cdot\,,\,\cdot\,]$ action for the adjoint (ad) representation of the action of the Lie algebra $f_1\circledS f_2$ on itself is defined by 
\begin{align}
{\rm ad}_{\big(\overline{f}_1;\overline{f}_2\big)}\big(f_1;f_2\big) 
=
\Big[ \big(f_1;f_2\big) , \big(\overline{f}_1;\overline{f}_2\big)\Big] 
:=
\Big( \big[ f_1,\overline{f}_1 \big] ; \big[ f_1,\overline{f}_2 \big] - \big[ \overline{f}_1,f_2 \big]\Big)\,,
\label{SDP-ad-action-fns}
\end{align}
where the commutator $[\,\cdot\,,\,\cdot\,]$ is given by the Jacobian of the functions $f_1$ and $\overline{f}_2$. For example,
\begin{align}
[ f_1,\overline{f}_2 ]:= J(f_1,\overline{f}_2)
\,,\label{SDP-commutator-fns}
\end{align}
which is also the commutator of symplectic vector fields. Thus, the adjoint (ad) action in \eqref{SDP-ad-action-fns} is the semidirect product Lie algebra action among symplectic vector fields. 

The definition in \eqref{SDP-ad-action-fns} of the semidirect product Lie algebra action among functions defined on the plane $\mathbb{R}^2$ enables the Lie-Poisson bracket \eqref{TQG-brkt3} for functionals of $(\varpi,b)$ defined on domain $\mc D$ to be identified with the coadjoint action of this Lie algebra. This is because the variational derivatives of such functionals live in the Lie algebra of symplectic vector fields on domain $\mc D$. Thus,
\begin{align}
\begin{split}
 \frac{\dd }{\dt} \mc{F}(\varpi,b)
 =
\Big\{\mc F\,,\,{\mc H}\Big\}(\varpi,b)
&=
- \Bigg\langle
\big(\varpi ; b\big) \,,\,\Bigg[ \bigg(\frac{\delta \mc F}{\delta \varpi} ; \frac{\delta \mc H}{\delta b}\bigg), 
\bigg(\frac{\delta \mc H}{\delta \varpi} ;\frac{\delta \mc F}{\delta b} \bigg)\Bigg] \Bigg\rangle
\\&=
- \Bigg\langle \big(\varpi ; b\big)
\,,\,
{\rm ad}_{\big(\frac{\delta \mc H}{\delta \varpi};\frac{\delta \mc H}{\delta b}\big)}
\bigg(\frac{\delta \mc F}{\delta \varpi};\frac{\delta \mc F}{\delta b}\bigg)
 \Bigg\rangle
\\&=:
- \Bigg\langle {\rm ad}^*_{\big(\frac{\delta \mc H}{\delta \varpi};\frac{\delta \mc H}{\delta b}\big)}\big(\varpi ; b\big)
\,,\,
\bigg(\frac{\delta \mc F}{\delta \varpi};\frac{\delta \mc F}{\delta b}\bigg)
 \Bigg\rangle
\,,
\end{split}
\label{SDP-LPB-TQG}
\end{align}
in which the angle brackets $\langle\,\cdot\,,\,\cdot\,\rangle$ represent the pairing of the Lie algebra of functions $f_1\circledS f_2$ with its dual Lie algebra $(f_1\circledS f_2)^*$ via the $L^2$ pairing on the domain $\mc D$ as in equation \eqref{TQG-brkt3}. Consequently, we may rewrite the TQG equations in \eqref{TQG-Ham3} in terms of this coadjoint action and thereby reveal its geometric nature,
\begin{align}
\frac{\partial \big(\varpi ; b\big)}{\partial t}
= - \,{\rm ad}^*_{\big(\frac{\delta \mc H}{\delta \varpi};\frac{\delta \mc H}{\delta b}\big)}
\big(\varpi ; b\big)
= - \,{\rm ad}^*_{\big(\psi\,;\, (b + h_1)/2 - \psi\big)}
\big(\varpi ; b\big)
\,.
\label{SDP-TQG-LP}
\end{align}
Thus, TQG dynamics is governed by the semidirect-product coadjoint action shown in equation \eqref{SDP-LPB-TQG} of symplectic vector fields acting on the momentum map $(\varpi ; b)$ defined in the dual space of potential vorticity and buoyancy functions . 
\bigskip

For completeness, we write the coadjoint (ad$^*$) representation of the (right) action of the Lie algebra $f_1\circledS f_2$ on its dual Lie algebra explicitly in terms of the Jacobian operator between pairs of functions as 
\begin{align}
\begin{split}
\Big\langle \big(\varpi ; b\big) \,,\,
{\rm ad}_{\big(\overline{f}_1;\overline{f}_2\big)}\big(f_1;f_2\big) 
\Big\rangle
&:= -\,
\Big\langle \big(\varpi ; b\big) \,,\,
\Big[ \big(f_1;f_2\big) , \big(\overline{f}_1;\overline{f}_2\big)\Big] 
\Big\rangle
\\
\hbox{By \eqref{SDP-ad-action-fns} and \eqref{SDP-commutator-fns}}\quad
&:= -\,
\Big\langle \big(\varpi ; b\big) \,,\,
\Big( J\big( f_1,\overline{f}_1 \big) ; J\big( f_1,\overline{f}_2 \big) 
- J\big( \overline{f}_1,f_2 \big)\Big)
\Big\rangle
\\
\hbox{Defining the SDP pairing}\quad
&:= -\,
\int_{\mc D} 
\varpi \,J\big( f_1 , \overline{f}_1 \big)
+ b \Big( J\big( f_1,\overline{f}_2 \big) - J\big( \overline{f}_1,f_2 \big) \Big)
\,\dd x \dd y
\\& = -\,
\int_{\mc D} 
\varpi \,J\big( f_1,\overline{f}_1 \big)
+  \Big( b J\big( f_1,\overline{f}_2 \big) + b J\big( f_2 , \overline{f}_1 \big) \Big)
\,\dd x \dd y
\\\hbox{By integrating by parts}\quad
& = 
\int_{\mc D} 
f_1\,\Big( J\big( \varpi ,\overline{f}_1 \big) +  J\big( b  ,\overline{f}_2 \big)\Big) 
- f_2J\big( b , \overline{f}_1 \big) 
\,\dd x \dd y
\\
\hbox{Upon the SDP pairing}\quad
&= 
\Big\langle 
\Big( J\big( \varpi,\overline{f}_1 \big) + J\big( b ,\overline{f}_2\big)  \,;\, - J\big( \overline{f}_1,b \big)
\Big)
\,,\,
\big( f_1 ; f_2\big) 
\Big\rangle
\\
\hbox{By \eqref{SDP-LPB-TQG}}\quad
&=: 
\Big\langle {\rm ad}^*_{\big(\overline{f}_1;\overline{f}_2\big)}\big(\varpi ; b\big) \,,\,
\big(f_1;f_2\big) 
\Big\rangle
\quad\hbox{with}\quad
\big(\overline{f}_1;\overline{f}_2\big)
=\Big(\frac{\delta \mc H}{\delta \varpi};\frac{\delta \mc H}{\delta b}\Big)
\,.
\end{split}
\label{SDP-adstar-action-fns}
\end{align}
This calculation confirms that the bracket in \eqref{TQG-brkt3} satisfies the Jacobi identity, by being a linear functional of a Lie algebra commutator which satisfies the Jacobi identity. The calculation thus identifies the bracket in \eqref{TQG-brkt3} as the Lie-Poisson bracket for functionals of $(\varpi ; b)$ defined on the dual $(f_1\circledS f_2)^*$ of the semidirect-product Lie algebra $f_1\circledS f_2$, whose commutator is defined in \eqref{SDP-ad-action-fns}. 
\end{proof}


\paragraph{\bf Summary.}
The TQG model in \eqref{trsw-tg} has been shown to be a Hamiltonian system for the semidirect-product Lie-Poisson bracket defined in equation \eqref{TQG-brkt3}. This result implies that the TQG system will possess all of the geometric properties belonging to the class of ideal fluid models with advected quantities. The geometric properties of this class of ideal fluid models are discussed in detail in \cite{HMR1998}. In particular, the geometric nature of the evolution of the TQG system written in its Hamiltonian form in \eqref{TQG-Ham3} has been revealed in \eqref{SDP-TQG-LP} by identifying its Lie-Poisson bracket with the coadjoint action of the semidirect-product Lie group of symplectic transformations defined in equation \eqref{SDP-TQG-LP}. Namely, the solutions of the TQG system in \eqref{SDP-TQG-LP} evolve by undergoing coadjoint motion along a time-dependent path on the Lie group manifold of semidirect-product symplectic diffeomorphisms acting on the domain of flow $\mc{D}$. This association of the TQG system with the smooth flow of symplectomorphisms on  $\mc{D}$ bodes well for the analytical properties of TQG solutions. Indeed, the preservation of these smooth flow properties obtained via the SALT approach provides a geometric framework for the determination of the corresponding analytical properties for the stochastic counterpart of TQG, which is the primary aim of the present paper. 

\section{Proof of Lemma \ref{rem:xiAppendix}}
\noindent\textbf{Proof of Lemma \ref{rem:xiAppendix}}  \\
 We show that 
\begin{equation*}
  \begin{aligned}
  \langle \mathcal{L}_{\bm{\xi}_i}(q-b), \mathcal{L}_{\bm{\xi}_i}(q-b)\rangle_{k-1,2} + \langle \mathcal{L}_{\bm{\xi}_i}^2(q-2b), q\rangle_{k-1,2}\lesssim_k \|q\|_{k-1,2}^2+\|b\|_{k,2}^2
  \end{aligned}
\end{equation*}
For $k=1$ we have
\begin{equation*}
    \begin{aligned}
    \langle \mathcal{L}_{\bm{\xi}_i}(q-b), \mathcal{L}_{\bm{\xi}_i}(q-b)\rangle + \langle q, \mathcal{L}_{\bm{\xi}_i}^2(q-2b)\rangle  &=  \langle \mathcal{L}_{\bm{\xi}_i}(q-b), \mathcal{L}_{\bm{\xi}_i}(q-b)\rangle + \langle q, \mathcal{L}_{\bm{\xi}_i}^2(q-b)\rangle - \langle q, \mathcal{L}_{\bm{\xi}_i}^2b\rangle  \\
    & = \langle b, \mathcal{L}_{\bm{\xi}_i}^2q\rangle - \langle q+b, \mathcal{L}_{\bm{\xi}_i}^2b\rangle \\
    & = -\langle \mathcal{L}_{\bm{\xi}_i}b, \mathcal{L}_{\bm{\xi}_i}q\rangle + \langle \mathcal{L}_{\bm{\xi}_i}b, \mathcal{L}_{\bm{\xi}_i}q\rangle + \langle \mathcal{L}_{\bm{\xi}_i}b, \mathcal{L}_{\bm{\xi}_i}b\rangle \\
    & = \langle \mathcal{L}_{\bm{\xi}_i}b, \mathcal{L}_{\bm{\xi}_i}b\rangle  \leq \|b\|_{1,2}^2.
    \end{aligned}
\end{equation*}
We used here the fact that $\mathcal{L}_{\bm{\xi}_i}^{\star} = -\mathcal{L}_{\bm{\xi}_i}$ and the fact that condition $
\sum_{i\in \mathbb{N}} \Vert \bm{\bm{\xi}}_i \Vert_{4,\infty}<\infty$ implies 
$$\displaystyle\sum_{i=1}^{\infty} \|\mathcal{L}_{\bm{\xi}_i}f\|_2^2 \leq C\|f\|_{1,2}^2$$
for any $f\in W^{1,2}(\mathbb{T}^2)$. 
For $k\geq 2$ and a multi-index $\beta$ with $|\beta| \leq k-1$ one has
\begin{equation*}
    \begin{aligned}
    \langle \partial^{\beta}\mathcal{L}_{\bm{\xi}_i}(q-b), \partial^{\beta}\mathcal{L}_{\bm{\xi}_i}(q-b)\rangle & +  \langle \partial^{\beta}q, \partial^{\beta}\mathcal{L}_{\bm{\xi}_i}^2(q-2b)\rangle = \langle \partial^{\beta}\mathcal{L}_{\bm{\xi}_i}q,\partial^{\beta}\mathcal{L}_{\bm{\xi}_i}q\rangle + \langle \partial^{\beta}q,\partial^{\beta}\mathcal{L}_{\bm{\xi}_i}^2q\rangle \\
    & -2 \left(\langle \partial^{\beta}\mathcal{L}_{\bm{\xi}_i}q,\partial^{\beta}\mathcal{L}_{\bm{\xi}_i}b\rangle +  \langle \partial^{\beta}q,\partial^{\beta}\mathcal{L}_{\bm{\xi}_i}^2b\rangle  \right)+\langle \partial^{\beta}\mathcal{L}_{\bm{\xi}_i}b,\partial^{\beta}\mathcal{L}_{\bm{\xi}_i}b\rangle \\
    & \leq C(\|q\|_{k-1,2}^2+\|b\|_{k,2}^2).
    \end{aligned}
\end{equation*}

\section{Spatial discretisation}\label{sec: FEM}


In this appendix section, we describe the finite element (FEM) method we use for the STQG system. Our setup applies to problems on bounded domains $\domain$ with Dirichlet boundary conditions 
\begin{equation}\label{eq: tqg boundary conditions numerical}
     \psi = 0,
    \qquad \text{on } \p \domain.
\end{equation}
For $\domain = \mathbb T^2$, the boundary flux terms in the discretised equations are set to zero.

\subsubsection{The stream function equation}\label{sec: stream function discretisation}

Let $H^{1}\left(\domain\right)$ denote the Sobolev $W^{1,2}\left(\domain\right)$
space and let $\left\Vert .\right\Vert _{\partial\domain}$ denote
the $L^{2}\left(\partial\domain\right)$ norm.
Define the space 
\begin{equation}\label{eq: w1 space}
    W^{1}\left(\domain\right):=\left\{ \nu\in H^{1}\left(\domain\right)\left|\left\Vert \nu\right\Vert _{\partial\domain}=0\right.\right\}.
\end{equation}
We look for numerical solutions to the Helmholtz elliptic problem \eqref{constrt} in the space $W^1(\domain)$. 
Define the functionals
\begin{align}
    L (v, \phi) &:=  \langle \nabla v, \nabla \phi \rangle_\domain +\langle v, \phi  \rangle_\domain \label{eq: psi bilinear form}\\
    F_{\cdot}(\phi) &:=  -\langle  \cdot, \phi \rangle_\domain
    \label{eq: psi linear form}
\end{align}
for $v, \phi \in W^1(\domain)$, then the Helmholtz problem can be written as
\begin{equation}\label{eq: TQG weak elliptic}
    L( \psi, \phi) = F_{q  -  f }(\phi).
\end{equation}
We discretise \eqref{eq: TQG weak elliptic} using a continuous Galerkin (CG) discretisation scheme. 

Let $\delta$ be the discretisation parameter, and let $\domain_\delta$ denote a space filling triangulation of the domain, that consists of geometry-conforming non-overlapping elements.
Define the approximation space
\begin{equation} \label{eq: cg space}
    W_\delta^{k}(\domain):=\left\{ \phi_\delta\in W^{1}\left(\domain\right)\ : \ \phi_\delta\in C\left(\domain\right),\left.\phi_\delta\right|_{K}\in\Pi^{k}\left(K\right)\text{ each } K \in \domain_\delta\right\}.
\end{equation}
in which $C(\domain)$ is the space of continuous functions on $\domain$, and  $\Pi^{k}\left(K\right)$ denotes the space of polynomials of degree at most $k$ on element $K\in \domain_\delta$. 

For \eqref{eq: TQG weak elliptic}, given $f_\delta \in W_\delta^k(\domain)$ and $q_\delta \in V_\delta^k(\domain)$ (see \eqref{DG space} for the definition of $V_\delta^k(\domain)$), 
our numerical approximation is the solution $  \psi_\delta \in W_\delta^k(\domain)$ that satisfies 
\begin{equation}
L (  \psi_\delta, \phi_\delta) = F_{q_\delta - f_\delta}(\phi_\delta)
\label{eq: TQG psi eqn weak discretised}
\end{equation}
for all test functions $\phi_\delta \in W_\delta^k(\domain)$. 
For a detailed exposition of the numerical algorithms that solves the discretised problem \eqref{eq: TQG psi eqn weak discretised} 
 we point the reader to \cite{Gibson2019, Brenner2008}.

To handle the noise terms in the hyperbolic equations, let $\zeta_i \in W^1(\domain)$ denote the stream function of $\boldsymbol{\xi}$, i.e. $\zeta_i = \nabla^\perp \boldsymbol{\xi}_i$. So the discretised $\zeta_i^\delta$ are in the same space as $\psi_\delta$.

\subsubsection{Hyperbolic equations}\label{sec: buoyancy equation discretisation}
We choose to discretise the hyperbolic buoyancy \eqref{ce} and potential vorticity \eqref{me} equations using a discontinuous Galerkin (DG) scheme. For a detailed exposition of DG methods, we refer the interested reader to \cite{Hesthaven2008}.

Define the DG approximation space, denoted by $V_\delta^k(\domain)$, to be the element-wise polynomial space,
\begin{equation}\label{DG space}
    V_\delta^k(\domain) = \left\{ \left. v_\delta \in L^2(\domain) \right|\forall K \in \domain_\delta,\ \exists\phi_\delta\in \Pi^k (K):\ \left.v_\delta\right|_{K}=\left.\phi_\delta\right|_K\right\}.
\end{equation}
We look to approximate $b $ and $q $ in the space $V_\delta^k(\domain)$. Essentially, this means our approximations of $b $ and $q $ are each an direct sum over the elements in $\domain_\delta$. Additional constraints on the numerical fluxes across shared element boundaries are needed to ensure conservation properties and stability. 
Further, note that $W^k_\delta(\domain) \subset V^k_\delta(\domain)$. This inclusion is needed for ensuring numerical conservation of Casimirs \eqref{eq:casimirstqg}.

For the buoyancy equation \eqref{ce}, we obtain the following variational formulation
\begin{align}
\langle \dd b , \nu_\delta \rangle _{K} & = \langle b  ({\bf u}\dd t + \boldsymbol{\xi}_i \circ \dd W_t^i) ,\nabla\nu_\delta\rangle _K - \langle b  ({\bf u}\dd t + \boldsymbol{\xi}_i \circ \dd W_t^i) \cdot{\normal},\nu_\delta\rangle _{\partial K},\quad K\in \domain_\delta
\label{eq: TQG b alpha equation weak}
\end{align}
where $\nu_\delta\in V^k_\delta(\domain)$ is any test function, $\p K$ denotes the boundary of $K$, and $\normal$ denotes the unit normal vector to $\p K$. Let $b _\delta$ be the approximation of $b $ in $V_\delta^k$, and let $\bu _\delta = \nabla^\perp \psi _\delta$ for $\psi _\delta \in W^k_\delta$. Similarly, let $\boldsymbol{\xi}_i^\delta = \nabla^\perp \zeta_i^\delta$ for $\zeta_i^\delta \in W^k_\delta$. Our discretised buoyancy equation over each element is given by
\begin{equation}
    \langle \dd b _\delta, \nu_\delta \rangle _{K} = \langle b _\delta ({\bf u} _\delta \dd t + \boldsymbol{\xi}_i^\delta \circ \dd W^i),\nabla\nu_\delta\rangle _K - \langle b _\delta({\bf u} _\delta \dd t + \boldsymbol{\xi}_i^\delta \circ \dd W^i)\cdot {\normal},\nu_\delta\rangle _{\partial K}, \quad K \in \domain_\delta.
    \label{eq: TQG b alpha eqn weak discretised}
\end{equation}

Similarly, let $q _\delta \in V_\delta^k(\domain)$ be the approximation of $q $, and let $h_\delta \in W_\delta^k$. We obtain the following discretised variational formulation that corresponds to \eqref{me}, 
\begin{align}
\langle \dd q _\delta, \nu_\delta \rangle _{K} 
&= \langle (q _\delta - b _\delta)({\bf u} _\delta \dd t + \boldsymbol{\xi}_i^\delta \circ \dd W^i),\nabla\nu_\delta\rangle _{K}
-\langle (q _\delta - b _\delta) ({\bf u} _\delta \dd t + \boldsymbol{\xi}_i^\delta \circ \dd W^i) \cdot{\normal},\nu_\delta\rangle _{\partial K} \nonumber \\
& \qquad - \frac12\langle \nabla \cdot (b _\delta \nabla^\perp h_\delta), \nu_\delta \rangle _K,
\qquad K\in \domain_\delta,
\label{eq: TQG omega alpha eqn weak discretised}    
\end{align}
for test function $\nu_\delta \in V_\delta^k(\domain)$. 

At this point, we only have the discretised problem on single elements. To obtain the global approximation, we sum over all the elements in $\domain_\delta$. In doing so, the $\p K$ terms in \eqref{eq: TQG b alpha eqn weak discretised} and \eqref{eq: TQG omega alpha eqn weak discretised} must be treated carefully.
Let $\p K_\text{ext}$ denote the part of cell boundary that is contained in $\p \domain$. Let $\p K_\text{int}$ denote the part of the cell boundary that is contained in the interior of the domain $\domain\backslash \p \domain$.
On $\p K_\text{ext}$ we simply impose the PDE boundary conditions. However, on $\p K_\text{int}$ we need to consider the contribution from each of the neighbouring elements.
By choice, the approximants $\psi _\delta, \zeta_i^\delta \in W_\delta^k$ are continuous on $\p K$. And since
\begin{equation}
    (\bu _\delta + \boldsymbol{\xi}_i^\delta) \cdot \normal = \nabla^\perp (\psi_\delta+ \zeta_i^\delta) \cdot \normal = -\nabla (\psi_\delta+ \zeta_i^\delta) \cdot \hat \tau = -\frac{\dd(\psi_\delta+ \zeta_i^\delta)}{\dd\hat\tau},
\end{equation}
where $\hat\tau$ denotes the unit tangential vector to $\p K$, $(\bu _\delta + \boldsymbol{\xi}_i^\delta) \cdot \normal$ is also continuous. This means  $(\psi_\delta+ \zeta_i^\delta) \cdot \normal$ in \eqref{eq: TQG b alpha eqn weak discretised} (also in \eqref{eq: TQG omega alpha eqn weak discretised}) is single valued. However, due to the lack of global continuity constraint in the definition of $V_k (\domain)$, $b _\delta$ and $q _\delta$ are multi-valued on $\p K_\text{int}$. Thus, as our approximation of $q $ and $b $ over the whole domain is the sum over $K \in \domain_\delta$,  we have to constrain the flux on the set $\left(\bigcup_{K\in\domain_\delta} \p K \right)\setminus \p \domain$. This is done using appropriately chosen numerical flux fields in the boundary terms of \eqref{eq: TQG b alpha eqn weak discretised} and \eqref{eq: TQG omega alpha eqn weak discretised}.

Let 
{$\nu^{-}:=\lim_{\epsilon\uparrow0} \nu ({\bf x}+\epsilon\normal)$} 
and 
{ $\nu^{+}:=\lim_{\epsilon\downarrow0}\nu ({\bf x}+\epsilon\normal)$}, 
for ${\bf x}\in\partial K$, be the inside and outside (with respect to a fixed element $K$) values respectively, of a function $\nu$ on the boundary.
Let $\hat{f}$ be a \emph{numerical flux} function 
that satisfies the following properties:
\begin{enumerate}[(i)]
    \item consistency 
        \begin{equation}\label{eq: nflux consistency}
         \hat f (\nu, \nu, (\bu _\delta \dd t+ \boldsymbol{\xi}_i^\delta \circ \dd W^i) \cdot {\normal}) 
         = 
         {\nu } (\bu _\delta \dd t+ \boldsymbol{\xi}_i^\delta \circ \dd W^i)\cdot\normal 
         \end{equation}
    
    \item conservative
        \begin{equation}\label{eq: nflux conservative}
            \hat f (\nu^+, \nu^-, (\bu _\delta\dd t + \boldsymbol{\xi}_i^\delta \circ \dd W^i)\cdot {\normal}) = - \hat{f} (\nu^-, \nu^+, -(\bu _\delta \dd t  + \boldsymbol{\xi}_i^\delta \circ \dd W^i)\cdot \normal)
        \end{equation}
    
    \item $L^2$ stable in the enstrophy norm with respect to the buoyancy equation, see \cite[Section 6]{Bernsen2006}.
\end{enumerate}
With such an $\hat f$, we replace 
$b _\delta (\bu _\delta \dd t+ \boldsymbol{\xi}_i^\delta \circ \dd W^i) \cdot \normal$ by the numerical flux $\hat{f}(b_\delta^{ , +},b_\delta^{ , -},(\bu _\delta \dd t+ \boldsymbol{\xi}_i^\delta \circ \dd W^i) \cdot\normal)$  in \eqref{eq: TQG b alpha eqn weak discretised}.
Similarly, in \eqref{eq: TQG omega alpha eqn weak discretised}, we replace $(q _\delta - b _\delta)(\bu _\delta \dd t+ \boldsymbol{\xi}_i^\delta \circ \dd W^i)  \cdot \normal$ by $\hat{f}((q _\delta - b _\delta)^{+}, (q _\delta - b _\delta)^{-}, (\bu _\delta \dd t+ \boldsymbol{\xi}_i^\delta \circ \dd W^i)  \cdot \normal)$.

\begin{remark}
For a general nonlinear conservation law, one has to solve what is called the \emph{Riemann problem} for the numerical flux, see \cite{Hesthaven2008} for details.
In our setup, we use the following local Lax-Friedrichs flux, which is an approximate Riemann solver,
\begin{equation}
\hat{f}(\nu^+, \nu^-, (\bu _\delta \dd t+ \boldsymbol{\xi}_i^\delta \circ \dd W^i)\cdot{\normal}) = (\bu _\delta \dd t+ \boldsymbol{\xi}_i^\delta \circ \dd W^i)\cdot{\normal}
\{\{ \nu  \}\} - \frac{|(\bu _\delta \dd t+ \boldsymbol{\xi}_i^\delta \circ \dd W^i)\cdot{\normal}|}{2}\llbracket\nu\rrbracket
\end{equation}
where
\begin{equation}
\{\{\nu\}\}:=\frac{1}{2}(\nu^{-}+\nu^{+}),\qquad
\llbracket\nu\rrbracket:={\normal}^{-}\nu^{-}+{\normal}^{+}\nu^{+}.
\end{equation}
\end{remark}

Finally, our goal is to find $b _\delta, q _\delta \in V_\delta^k(\domain)$ such that for all {$\nu_\delta\in V_\delta^k(\domain)$} we have
\begin{align}
\sum_{K\in \domain_\delta} \langle \dd b _\delta, \nu_\delta\rangle _{K} 
&=
\sum_{K\in \domain_\delta}
\big\{
\langle b _\delta\nabla^{\perp} \tilde{\psi} _\delta, \nabla \nu_\delta\rangle _{K}  -
\langle \hat{f}_{b }(b_\delta^{ , +},b_\delta^{ , -},\nabla^{\perp}\tilde{\psi} _\delta \cdot {\normal}), \nu_\delta^{-}\rangle _{\partial K}
\big\}, 
\label{eq: TQG b eqn weak discretised numerical flux}
\\
\sum_{K\in \domain_\delta} \langle \dd q _\delta, \nu_\delta \rangle _{K} 
&= 
\sum_{K\in \domain_\delta} 
\big\{
\langle (q _\delta - b _\delta)\nabla^{\perp}\tilde{\psi}
_\delta,\nabla\nu_\delta\rangle _{K} - \frac12\langle \nabla \cdot (b _\delta \nabla^\perp h_\delta), \nu_\delta \rangle _K
\label{eq: TQG omega eqn weak discretised numerical flux}
\\
&\qquad -
\langle \hat{f}_{q }((q _\delta - b _\delta)^{+}, (q _\delta - b _\delta)^{-}, \nabla^\perp \tilde{\psi}_\delta  \cdot {\normal}),
\nu_\delta^ -
\rangle _{\partial K}
\big\}
\nonumber
\end{align}
where $\tilde{\psi} _\delta := \psi _\delta \dd t + \zeta_i^\delta\circ \dd W^i$ was introduced for notation simplicity. 

\begin{remark}
In \eqref{eq: TQG b eqn weak discretised numerical flux} and \eqref{eq: TQG omega eqn weak discretised numerical flux} we do not explicitly distinguish external and internal boundaries, $\p K_\text{ext}$ and $\p K_\text{int}$, because for the boundary condition \eqref{eq: tqg boundary conditions numerical}, the $\p K_\text{ext}$ terms vanish. 
\end{remark}

\end{appendices}


\section*{Acknowledgements}
This work has been partially supported by European Research Council (ERC) Synergy grant STUOD-856408. We are grateful to T. Beale for thoughtful correspondence about the derivation of the celebrated BKM result in the case of periodic boundary conditions.


\begin{thebibliography}{10}
\providecommand{\url}[1]{#1}
\csname url@samestyle\endcsname
\providecommand{\newblock}{\relax}
\providecommand{\bibinfo}[2]{#2}
\providecommand{\BIBentrySTDinterwordspacing}{\spaceskip=0pt\relax}
\providecommand{\BIBentryALTinterwordstretchfactor}{4}
\providecommand{\BIBentryALTinterwordspacing}{\spaceskip=\fontdimen2\font plus
\BIBentryALTinterwordstretchfactor\fontdimen3\font minus
  \fontdimen4\font\relax}
\providecommand{\BIBforeignlanguage}[2]{{%
\expandafter\ifx\csname l@#1\endcsname\relax
\typeout{** WARNING: IEEEtran.bst: No hyphenation pattern has been}%
\typeout{** loaded for the language `#1'. Using the pattern for}%
\typeout{** the default language instead.}%
\else
\language=\csname l@#1\endcsname
\fi
#2}}
\providecommand{\BIBdecl}{\relax}
\BIBdecl

\bibitem{Holm2015}
\BIBentryALTinterwordspacing
D.~D. Holm, ``Variational principles for stochastic fluid dynamics,''
  \emph{Proceedings of the Royal Society A: Mathematical, Physical and
  Engineering Sciences}, vol. 471, no. 2176, p. 20140963, 2015. [Online].
  Available: \url{http://dx.doi.org/10.1098/rspa.2014.0963}
\BIBentrySTDinterwordspacing

\bibitem{cotter2018modelling}
\BIBentryALTinterwordspacing
C.~Cotter, D.~Crisan, D.~Holm, W.~Pan, and I.~Shevchenko, ``Modelling
  uncertainty using stochastic transport noise in a 2-layer quasi-geostrophic
  model,'' \emph{Foundations of Data Science}, vol.~2, no.~2, p. 173, 2020.
  [Online]. Available: \url{https://doi.org/10.3934/fods.2020010}
\BIBentrySTDinterwordspacing

\bibitem{cotter2019numerically}
\BIBentryALTinterwordspacing
C.~Cotter, D.~Crisan, D.~D. Holm, W.~Pan, and I.~Shevchenko, ``Numerically
  modeling stochastic lie transport in fluid dynamics,'' \emph{SIAM Multiscale
  Modeling \& Simulation}, vol.~17, no.~1, pp. 192--232, 2019. [Online].
  Available: \url{https://doi.org/10.1137/18M1167929}
\BIBentrySTDinterwordspacing

\bibitem{holm2019stochastic}
\BIBentryALTinterwordspacing
D.~D. Holm and E.~Luesink, ``Stochastic wave-current interaction in thermal
  shallow water dynamics,'' \emph{J. Nonlinear Sci.}, vol.~31, no.~29, 2021.
  [Online]. Available: \url{https://doi.org/10.1007/s00332-021-09682-9}
\BIBentrySTDinterwordspacing

\bibitem{HLP2021}
\BIBentryALTinterwordspacing
D.~D. Holm, E.~Luesink, and W.~Pan, ``Stochastic mesoscale circulation dynamics
  in the thermal ocean,'' \emph{Physics of Fluids}, vol.~33, no.~4, p. 046603,
  2021. [Online]. Available: \url{https://doi.org/10.1063/5.0040026}
\BIBentrySTDinterwordspacing

\bibitem{BeronVera2021}
\BIBentryALTinterwordspacing
F.~Beron-Vera, ``Nonlinear saturation of thermal instabilities,'' \emph{Physics
  of Fluids}, vol.~33, no.~3, p. 036608, 2021. [Online]. Available:
  \url{https://doi.org/10.1063/5.0045191}
\BIBentrySTDinterwordspacing

\bibitem{BeronVera2021arXiv}
\BIBentryALTinterwordspacing
F.~Beron-Vera, ``Geometry of shallow-water dynamics with thermodynamics,'' \emph{arXiv
  preprint arXiv:2106.08268}, 2021.

\bibitem{HH2021}
\BIBentryALTinterwordspacing
D.~D. Holm and R.~Hu, ``Stochastic effects of waves on currents in the ocean
  mixed layer,'' \emph{Journal of Mathematical Physics}, vol.~62, no.~7, p.
  073102, 2021. [Online]. Available:
  \url{https://aip.scitation.org/doi/10.1063/5.0045010}
\BIBentrySTDinterwordspacing

\bibitem{HuPatching2022}
R.~Hu and S.~Patching, ``Variational stochastic parameterisations and their
  applications to primitive equation models,'' in \emph{Stochastic Transport in
  Upper Ocean Dynamics}.\hskip 1em plus 0.5em minus 0.4em\relax Springer
  International Publishing, 2023, pp. 135--158.

\bibitem{breit2018local}
\BIBentryALTinterwordspacing
D.~Breit, E.~Feireisl, and M.~Hofmanov\'{a}, ``Local strong solutions to the
  stochastic compressible {N}avier-{S}tokes system,'' \emph{Comm. Partial
  Differential Equations}, vol.~43, no.~2, pp. 313--345, 2018. [Online].
  Available: \url{https://doi.org/10.1080/03605302.2018.1442476}
\BIBentrySTDinterwordspacing

\bibitem{breit2018stoch}
\BIBentryALTinterwordspacing
D.~Breit, E.~Feireisl, and M.~Hofmanov\'{a}, \emph{Stochastically forced compressible fluid flows}, ser. De Gruyter
  Series in Applied and Numerical Mathematics.\hskip 1em plus 0.5em minus
  0.4em\relax Berlin: De Gruyter, 2018, vol.~3.

\bibitem{breit2019stochastic}
D.~Breit and P.~Mensah, ``Stochastic compressible {E}uler equations and
  inviscid limits,'' \emph{Nonlinear Anal.}, vol. 184, pp. 218--238, 2019.

\bibitem{glatt2009strong}
N.~Glatt-Holtz and M.~Ziane, ``Strong pathwise solutions of the stochastic
  navier-stokes system,'' \emph{Adv. Differential Equations}, vol.~14, no. 5-6,
  pp. 567--600, 2009.

\bibitem{glatt2012local}
N.~Glatt-Holtz and V.~Vicol, ``Local and global existence of smooth solutions
  for the stochastic euler equations with multiplicative noise,'' \emph{Ann.
  Probab.}, vol.~42, no.~1, pp. 80--145, 2014.

\bibitem{mensah2019theses}
\BIBentryALTinterwordspacing
P.~Mensah, ``The stochastic compressible navier-stokes system on the whole
  space and some singular limits,'' Ph.D. dissertation, Heriot--Watt
  University, 2019. [Online]. Available:
  \url{https://www.ros.hw.ac.uk/handle/10399/4210}
\BIBentrySTDinterwordspacing

\bibitem{beale1984remarks}
J.~T. Beale, T.~Kato, and A.~Majda, ``Remarks on the breakdown of smooth
  solutions for the {$3$}-{D} {E}uler equations,'' \emph{Comm. Math. Phys.},
  vol.~94, no.~1, pp. 61--66, 1984.

\bibitem{crisan2019solution}
D.~Crisan, F.~Flandoli, and D.~D. Holm, ``Solution properties of a 3{D}
  stochastic {E}uler fluid equation,'' \emph{J. Nonlinear Sci.}, vol.~29,
  no.~3, pp. 813--870, 2019.

\bibitem{crisan2022breakdown}
D.~Crisan and P.~Mensah, ``Blow-up of strong solutions of the thermal
  quasi-geostrophic equation,'' in \emph{Stochastic Transport in Upper Ocean
  Dynamics}.\hskip 1em plus 0.5em minus 0.4em\relax Springer International
  Publishing, 2023, pp. 1--14.

\bibitem{CHLMP2021}
D.~Crisan, D.~Holm, E.~Luesink, P.~Mensah, and W.~Pan, ``Theoretical and
  computational analysis of the thermal quasi-geostrophic model,'' \emph{arXiv
  preprint arXiv:2106.14850}, 2021.

\bibitem{crisan2019well}
\BIBentryALTinterwordspacing
D.~Crisan and O.~Lang, ``Well-posedness for a stochastic 2{D} euler equation
  with transport noise,'' \emph{Stoch PDE: Anal Comp}, 2022. [Online].
  Available: \url{https://doi.org/10.1007/s40072-021-00233-7}
\BIBentrySTDinterwordspacing

\bibitem{klainerman1981singular}
\BIBentryALTinterwordspacing
S.~Klainerman and A.~Majda, ``Singular limits of quasilinear hyperbolic systems
  with large parameters and the incompressible limit of compressible fluids,''
  \emph{Comm. Pure Appl. Math.}, vol.~34, no.~4, pp. 481--524, 1981. [Online].
  Available: \url{https://doi.org/10.1002/cpa.3160340405}
\BIBentrySTDinterwordspacing

\bibitem{grafakos2008classical}
L.~Grafakos, \emph{Classical Fourier analysis}, 2nd~ed., ser. Graduate Texts in
  Mathematics.\hskip 1em plus 0.5em minus 0.4em\relax Springer, 2008, vol. 249.

\bibitem{karatzas1988brownian}
\BIBentryALTinterwordspacing
I.~Karatzas and S.~Shreve, \emph{Brownian motion and stochastic calculus}, ser.
  Graduate Texts in Mathematics.\hskip 1em plus 0.5em minus 0.4em\relax
  Springer-Verlag, 1988, vol. 113. [Online]. Available:
  \url{https://doi.org/10.1007/978-1-4612-0949-2}
\BIBentrySTDinterwordspacing

\bibitem{Hof}
D.~Breit and M.~Hofmanov\'{a}, ``Stochastic {N}avier--{S}tokes equations for
  compressible fluids,'' \emph{Indiana Univ. Math. J.}, vol.~65, no.~4, pp.
  1183--1250, 2016.

\bibitem{jakubowski1998short}
A.~Jakubowski, ``Short communication: The almost sure {S}korokhod
  representation for subsequences in nonmetric spaces,'' \emph{Theory Probab.
  Appl.}, vol.~42, no.~1, pp. 167--174, 1998.

\bibitem{kim2011on}
\BIBentryALTinterwordspacing
J.~Kim, ``On the stochastic quasi-linear symmetric hyperbolic system,''
  \emph{J. Differential Equations}, vol. 250, no.~3, pp. 1650--1684, 2011.
  [Online]. Available: \url{https://doi.org/10.1016/j.jde.2010.09.025}
\BIBentrySTDinterwordspacing

\bibitem{debussche2011local}
A.~Debussche, N.~Glatt-Holtz, and R.~Temam, ``Local martingale and pathwise
  solutions for an abstract fluids model,'' \emph{Phys. D}, vol. 240, no.~14,
  pp. 1123--1144, 2011.

\bibitem{gyongy1996existence}
I.~Gy\"ongy and N.~Krylov, ``Existence of strong solutions for it\^o's
  stochastic equations via approximations,'' \emph{Probab. Theory Related
  Fields}, vol. 105, no.~2, pp. 143--158, 1996.

\bibitem{doob1994measure}
J.~Doob, \emph{Measure theory}, ser. Graduate Texts in Mathematics.\hskip 1em
  plus 0.5em minus 0.4em\relax Springer-Verlag, 1994, vol. 143.

\bibitem{Hesthaven2008}
J.~Hesthaven and T.~Warburton, \emph{Nodal Discontinuous Galerkin
  Methods}.\hskip 1em plus 0.5em minus 0.4em\relax Springer New York, 2008.

\bibitem{adams2003sobolev}
R.~A. Adams and J.~J.~F. Fournier, \emph{Sobolev spaces}, 2nd~ed., ser. Pure
  and Applied Mathematics (Amsterdam).\hskip 1em plus 0.5em minus 0.4em\relax
  Amsterdam: Elsevier/Academic Press, 2003, vol. 140.

\bibitem{HMR1998}
\BIBentryALTinterwordspacing
D.~D. Holm, J.~E. Marsden, and T.~S. Ratiu, ``The {E}uler--{P}oincar{\'e}
  equations and semidirect products with applications to continuum theories,''
  \emph{Advances in Mathematics}, vol. 137, no.~1, pp. 1--81, 1998. [Online].
  Available: \url{https://doi.org/10.1006/aima.1998.1721}
\BIBentrySTDinterwordspacing

\bibitem{HolmPan2022}
D.~D. Holm and W.~Pan, ``Deterministic and stochastic {E}uler--{B}oussinesq
  convection,'' 
  \emph{Phys. D}, vol. 444,
  pp. 133584, 2023.

\bibitem{Gibson2019}
T.~Gibson, A.~McRae, C.~Cotter, L.~Mitchell, and D.~Ham, \emph{Compatible
  Finite Element Methods for Geophysical Flows}.\hskip 1em plus 0.5em minus
  0.4em\relax Springer International Publishing, 2019.

\bibitem{Brenner2008}
S.~Brenner and L.~Scott, \emph{The Mathematical Theory of Finite Element
  Methods}.\hskip 1em plus 0.5em minus 0.4em\relax Springer New York, 2008.

\bibitem{Bernsen2006}
\BIBentryALTinterwordspacing
E.~Bernsen, O.~Bokhove, and J.~J. van~der Vegt, ``A (dis)continuous finite
  element model for generalized 2d vorticity dynamics,'' \emph{Journal of
  Computational Physics}, vol. 211, no.~2, pp. 719--747, 2006. [Online].
  Available: \url{https://doi.org/10.1016/j.jcp.2005.06.008}
\BIBentrySTDinterwordspacing

\end{thebibliography}
 \end{document}